%% file: paper_arxiv_v3.tex
\documentclass[pre,showpacs,amsmath,amssymb,floatfix,onecolumn]{revtex4-1}
\usepackage{graphicx}
\usepackage{epsfig}
\usepackage{dcolumn}
\usepackage{bm}
\usepackage{color}
\usepackage{amsthm}
\usepackage{algorithmic}
\usepackage{algorithm}

\def\(({\left(}
\def\)){\right)}
\def\[[{\left[}
\def\]]{\right]}

\newtheorem*{theorem}{Theorem}
\newcommand{\bF}{{\textbf {F}}}

\newcommand{\bX}{{\textbf {X}}}

\newcommand{\by}{{\textbf {y}}}
\newcommand{\bw}{{\textbf {w}}}
\newcommand{\be}{\begin{equation}}
\newcommand{\ee}{\end{equation}}
\newcommand{\bea}{\begin{eqnarray}}
\newcommand{\eea}{\end{eqnarray}}

\begin{document}

\title{Phase transitions and sample complexity in Bayes-optimal matrix
  factorization} \author{Yoshiyuki Kabashima$^{1}$, Florent
  Krzakala$^{2}$, Marc M\'ezard$^{3}$, Ayaka Sakata$^{1,4}$, and Lenka
  Zdeborov\'a$^{5,*}$}

\affiliation{
$^1$ {Department of Computational Intelligence \& Systems Science,
  Tokyo Institute of Technology, Yokohama 226-8502, Japan.}\\
$^2$ Laboratoire de Physique Statistique,  CNRS UMR 8550 Ecole Normale
Sup\'erieure, PSL Research University \& Universit\'e Pierre et Marie Curie, Sorbonne Universit\'es.\\ 
$^3$ D\'epartement de Physique, Ecole Normale Sup\'erieure, PSL Research University, Paris. \\
$^4$ The Institute of Statistical Mathematics, Tachikawa 190-8562, Japan.\\
$^5$ Institut de Physique Th\'eorique, CEA Saclay and CNRS, Gif-sur-Yvette, France.\\
$^*$ To whom correspondence shall be sent: lenka.zdeborova@cea.fr\\
}

\begin{abstract}
  We analyse the matrix factorization problem. Given a noisy
  measurement of a product of two matrices, the problem is to estimate
  back the original matrices. 
 It arises in many applications such as
  dictionary learning, blind matrix calibration, sparse principal
  component analysis, blind source separation, low rank matrix
  completion, robust principal component analysis or factor
  analysis. It is also important in machine learning: unsupervised
  representation learning can often be studied through matrix
  factorization.  
We use the tools of statistical mechanics -- the cavity
  and replica methods -- to analyze the achievability and computational tractability
  of the inference problems in the setting of Bayes-optimal inference,
  which amounts to assuming that the two matrices have random
  independent elements generated from some known distribution, and
  this information is available to the inference algorithm. In this
  setting, we compute the minimal mean-squared-error achievable in
  principle in any computational time, and the error that can be
  achieved by an efficient approximate message passing algorithm. The
  computation is based on the asymptotic state-evolution analysis of
  the algorithm. The performance that our analysis predicts, both in
  terms of the achieved mean-squared-error, and in terms of sample
  complexity, is extremely promising and motivating for a further
  development of the algorithm \footnote{Part of the results
    discussed in this paper were presented at the 2013 IEEE
    International Symposium on Information Theory in Istanbul.}.
\end{abstract}

\date{\today}
\maketitle

\tableofcontents

\newpage
\section{Introduction}

We study in this paper a variety of questions which all deal with the
general problem of matrix factorization. Generically, this problem is
stated as follows: Given a  $M \times P$ dimensional matrix $Y$, that
was obtained from noisy element-wise measurements of a matrix $Z$, one seeks a factorization $Z=FX$, where the $M\times N$ dimensional  matrix $F$ and the
$N\times P$ dimensional  matrix $X$ must satisfy some specific
requirements like sparsity, low-rank or non-negativity. 

From a machine learning point of view, matrix factorization can be
applied to unsupervised learning of data representation
\cite{bengio2013representation}. The success of machine learning,
including recent progress such as deep learning \cite{lecun2015deep},
depends largely on data representations. Explicit approaches to
efficient data
representation, such as matrix factorization, are hence of wide relevance. 
Other applications that can be formulated as matrix factorization include dictionary learning or sparse coding
\cite{olshausen1996emergence,OlshausenField97,kreutz2003dictionary},
sparse principal component analysis \cite{zou2006sparse},
blind source separation \cite{belouchrani1997blind}, low rank matrix
completion \cite{candes2009exact,candes2010power} or robust
principal component analysis \cite{candes2011robust}, that will be described below. 

Theoretical limits on when matrix
factorization is possible and computationally tractable are still rather poorly
understood. In this work we make a step towards this understanding by
predicting the limits of matrix factorization and its algorithmic tractability when $Z$ is created
using randomly generated matrices $F$ and $X$, and measured
element-wise via a known noisy output channel $P_{\rm out}(Y|Z)$. Our results are
derived in the limit where $N, M, P \to \infty$ with fixed ratios
$M/N = \alpha$, $P/N = \pi$. We predict the existence of sharp phase
transitions in this limit and provide the explicit formalism to locate
them.

We use two types of methods in this paper. The first one is based on a generalization of
approximate message passing (AMP) \cite{DonohoMaleki09} to the matrix
factorization problem, and on its asymptotic analysis which is known
in statistical physics as the cavity method
\cite{MezardParisi87b,MezardMontanari09}, and has been called state
evolution in the context of compressed sensing
\cite{DonohoMaleki09}. The second method that we use in the following
is the replica method. These two methods are widely believed to be exact
in the context of theoretical statistical physics, but most of the
results that we shall obtain in the present work are not rigorously
established. Our predictions have thus the status of conjectures. A
first cross-check of the correctness of these conjectures is the fact that the two methods give
identical results. This has been understood first in the
context of spin glasses \cite{MezardParisi87b}.

This work builds upon some previous steps that we
described in earlier reports
\cite{SakataKabashimaNew,krzakala2013phase}. The message passing
algorithm related to our analysis was first presented in
\cite{krzakala2013phase} and is very closely related to the Big-AMP
algorithm developed and tested in
\cite{SCHNITER-BIG,parker2013bilinear,parker2014bilinear}; relations and differences with
Big-AMP will be mentioned in several places throughout the paper. Our
main focus here, beside the detailed derivation of the algorithm, is
the asymptotic analysis and phase diagrams which were not studied in
\cite{SCHNITER-BIG,parker2013bilinear,parker2014bilinear}. We also
discuss several variants of the AMP algorithm that are interesting for
theoretical reasons. Several of these
variants, however, have convergence problems when implemented
straightforwardly. For a robust implementation of the algorithm that
can be used on practical benchmarks we refer to the works \cite{parker2013bilinear,parker2014bilinear}.

Our general method provides a unifying framework for the study of
computational tractability and identifiability of various matrix factorization
problems. The first step for this synergy is the formulation of the problem via a
graphical model (see Fig. \ref{FactorGraph}) that is amenable to analysis using
the present methods.  

The phase diagrams that we shall derive establishes for each problem two types of thresholds
in the plane $\alpha-\pi$ : the threshold where the problem of matrix factorization
ceases to be solvable in principle, and the threshold where AMP ceases to
find the best solution. In most existing works the computation of
phase transitions was treated separately for each of the various
problems. For instance, redundant dictionaries for sparse
representations and low rankness are usually thought as two different
kinds of dimensional reduction. Interestingly, in our work a wide
class of problems is treated within one unified formalism;
this is theoretically interesting in the context of recent
developments \cite{elad2012sparse}.

\subsection{Statement of the problem}

In a general matrix factorization problem 
one measures some information about matrix elements of the product of
two unknown matrices $F \in {\mathbb R}^{M \times N}$ and $X \in
{\mathbb R}^{N \times P}$, whose matrix elements will be denoted
$F_{\mu i}$ and $X_{il}$. Let us denote the product $Z = F X \in
{\mathbb R}^{M \times P}$, with elements 
\be
z_{\mu l} =  \sum_{i=1}^N  F_{\mu i} X_{il}  \label{z=Fx}  \, .  
\ee
The element-wise
measurement
$y_{\mu l}$ of $z_{\mu l}$ is then specified by some known probability
distribution function $P_{\rm out}(y_{\mu l}|z_{\mu l})$, so that:
\bea
P_{\rm out}(Y|Z)= \prod_{\mu,l} P^{\mu l}_{\rm out}(y_{\mu l}|z_{\mu l}) \, .
\label{PY}
\eea
  The goal of
matrix factorization is to estimate both matrices $F$ and $X$ from the
measurements $Y$.

In this paper we will treat this problem in the framework of Bayesian
inference. In particular we will assume that the matrices $F$ and $X$
were both generated from a known separable probability distribution 
\bea
       P_F(F) &=& \prod_{\mu=1}^M \prod_{i=1}^N   P^{\mu i}_F(F_{\mu i})\, ,  \label{PF}\\
       P_X(X) &=&\prod_{i=1}^N  \prod_{l=1}^P   P^{il}_X(X_{i l})\, . \label{PX} 
\eea

Although we restrict to separable prior probability distributions  $P_F(F)$, $P_X(X)$, it turns
out that these priors can encode a broad range
constraints such as, for instance sparsity. This is the reason why so
many different problems can be studied within our scheme. The output channel $P_{\rm
  out}(y,z)$ can be of a rather generic nature, thus including various kinds
of additive and multiplicative noise; in
the spirit of activation functions from neural networks it can degrade
the information included in the measurement by e.g. keeping only the
sign of elements of $z$.  

In the following we shall mostly study the case where the
distributions $P^{\mu i}_F$ are all identical (there is a single
distribution  $P_F(F_{\mu i})$), and the distributions $P^{il}_X$ for
various $il$ (as well as $P_{\rm out}^{\mu l}$ for various $\mu l$) are
also all identical.
Our approach can be generalized to the case where
$P^{\mu i}_F$, $P^{il}_X$, $P_{\rm out}^{\mu l}$ depend in a
known way on the indices $\mu i$, $il$, and $\mu l$, provided that this
dependence is through parameters that themselves are taken from
separable probability distributions. Examples of such dependence
include the blind matrix calibration or the factor analysis. On the
other hand our theory does not cover the case of an arbitrary matrix
$\tilde F_{\mu i}$ for which we would have $P_F^{\mu i}=
\delta(F_{\mu i} - \tilde F_{\mu i})$.

The posterior distribution of $F$ and $X$ given the measurements $Y$ is written as 
\be
    P(F,X| Y ) = \frac{1}{{\cal Z}(Y)} P_F(F) P_X(X) P_{\rm out}(Y|FX)= \frac{1}{{\cal Z}(Y)}   \prod_{\mu,i}  P_F(F_{\mu i})
    \prod_{i,l}   P_X(X_{i l}) \prod_{\mu,l}  P_{\rm
  out}\left(y_{\mu l}| \sum_i F_{\mu i} X_{i l}\right)\, , \label{eq:post}
\ee
where ${\cal Z}(Y)$ is the normalization constant, known as the partition
function in statistical physics. 

Notice that, while the original problem of finding $F$ and $X$, given
the measurements $Y$, is not well determined (because of the
possibility to obtain, from a given solution, an infinity of other
solutions through the transformation $F\to F U^{-1}$ and $X\to U X$, where
$ U$ is any $N \times N$ nonsingular matrix), the fact of using well
defined priors $P^{\mu i}_F$ and $P^{il}_X$ actually lifts the degeneracy: the
problem of finding the most probable $F,X$ given the measurements and
the priors is well defined. In case the priors $P^{\mu i}_F$ and $P^{i
l}_X$ do not depend on the indices $\mu l$ and $i l$ we are left with
a permutational symmetry between the $N$ column of $F$ and $N$ rows
of $X$. Both in the algorithm and the asymptotic analysis this
symmetry is broken and one of the $N\!$ solutions is chosen at random.

Typically, in most applications, the distributions $P_{\rm
  out}$, $P_F$ and $P_X$ will depend on a set of parameters (such as
the mean, variance, sparsity, noise strength, etc.) that we
usually will not write explicitly in the general case, in order to simplify the
notations. The prior knowledge of these parameters is not necessarily required in our
approach: these parameters can be learned via an expectation-maximization-like
algorithm that we will discuss briefly in section \ref{sec:EM}.


Note also that Eq.~\ref{z=Fx} can be multiplied by an arbitrary constant: with
a corresponding change in the output function the problem will not be
modified. In the derivations of this paper we choose the above constant in such a way that the elements of matrices
$X$, $Y$, and $Z$ are of order $O(1)$, whereas the elements of $F$
scale in a consistent way, meaning that the mean of
each $F_{\mu i}$ is of order $O(1/N)$ and its variance is also of
order $O(1/N)$. 

\subsection{Bayes-optimal inference}
\label{Sec:OB}

Our paper deals with the general case of incomplete information. This
happens when the reconstruction assumes that the matrices $X$, $F$
and $Y$ were generated with some distributions $P_X$, $P_F$ and
$P_{\rm out}(Y|Z)$, whereas in reality the matrices were generated
using some other distributions $P_{X^0}$, $P_{F^0}$ and
$P^0_{\rm out}(Y|Z)$. The message passing algorithm and its asymptotic
evolution will be derived in this general case. 

However, our most important results concern the Bayes-optimal setting,
i.e. when we assume that 
\be
   P_{X^0}=P_X, \quad \quad  P_{F^0}=P_F, \quad \quad  P^0_{\rm out}(Y|Z)=P_{\rm out}(Y|Z)\, .
\ee
In this case, an estimator $X^{\star}$ that minimizes the mean-squared
error (MSE) with respect to the original signal $X^0$, defined as
\be
{\rm MSE}(X|Y)=\int \text{d} F^0 \, \text{d} X^0 \left[ \frac{1}{PN}
  \sum_{il}(X_{il}-X^0_{il})^2\right] P(F^0,X^0|Y)\, ,
\label{MSE_X_def}
\ee 
is obtained from marginals of $X_{il}$ with
respect to the posterior probability measure $P(F,X | Y)$, i.e.,
\be
X^{\star}_{il}=\int \text{d}  X_{il} \,  X_{il}\,  \nu_{il}(X_{il}) \,
, \quad \quad {\rm where} \quad \quad \nu_{il}(X_{il}) \equiv
\int_{\{F_{\mu j}\}} \int_{\{X_{jn}\}_{jn\neq il}} P( F,X| Y) \, , \label{average_marginal}
\ee
is the marginal probability distribution of the variable $il$. The
mean squared error achieved by this optimal estimator is called the
minimum mean squared error (MMSE) in this paper. 

A
similar result holds for the 
estimator of $F^0$ that minimizes the mean-squared error
\be
{\rm MSE}(F|Y)=\int \text{d} F^0 \, \text{d} X^0 \left[ \frac{1}{M}
  \sum_{\mu i}(F_{\mu i}-F^0_{\mu i})^2\right] P(F^0,X^0|Y)\, ,
\label{MSE_F_def}
\ee 
which is obtained from the mean of $F_{\mu i}$ with
respect to the posterior probability measure $P(F,X | Y)$.  
In the remainder of this article we will be using these estimators.

\subsection{Statement of the main result}
\label{sec:main_results}

The main result of this paper are explicit formulas for the MMSE
achievable in the Bayes optimal setting (as defined above) for the
matrix factorization problem in the ``thermodynamic limit'', i.e. when $N, M, P \to
 \infty$ with fixed ratios
$M/N = \alpha$, $P/N = \pi$. When sparsity is involved we consider
that a finite {\it fraction} of matrix elements are non-zero. Similarly, when we treat matrices
with low ranks we consider again the ranks to be a finite fraction of
the total dimension. We also derive the AMP-MSE, i.e. the mean square
error achievable by the
approximate message passing algorithm as derived in this paper. 

So far we were characterizing the output channel by the conditional
distribution $P^0_{\rm out}(y|z^0)$ or $P_{\rm out}(y|z)$. It
will be useful to think of the output as a deterministic function
$h$ of $z$ and of random variables $w$, i.e. $y_{\mu l} = h(z_{\mu
  l},w_{\mu l})
=h^0(z^0_{\mu l},w_{\mu l}^0)$. The random (``noise'') variables $w$ and $w^0$ are specified
by their probability distributions $P(w)$ and $P_0(w^0)$. 
We can relate $P_{\rm out}$ to $h$ as follows
\bea 
P_{\rm out}(y|z) &=& \int P(w) \,  {\rm d}
w \, \delta[y-h(z,w)]\, , \label{eq:hPout} \\
P^0_{\rm out}(y|z^0) &=& \int P_0(w^0) \,  {\rm d}
w^0 \, \delta[y-h^0(z^0,w^0)]\, , 
\eea


To compute the MMSE and AMP-MSE we need to analyze the fixed points of
the following iterative equation. 
\bea
           m_X^{t+1} &=& \int {\rm d}X P_{X}(X)  \int
         {\cal D}\xi \, \,  f_X^2\left[  \frac{1}{
             \alpha m_F^t \hat m^t}  , \frac{\alpha m_F^t
             \hat m^t\,  X + \xi \sqrt{\alpha m_F^t \hat m^t}}{
             \alpha m_F^t \hat m^t}   \right] \, ,\label{eq:mx_NL_i} \\
           m_F^{t+1} &=& \int {\rm d}F P_{F}(F)  \int
         {\cal D}\xi \,  \,  f_F^2\left[  \frac{1}{
             \pi m_X^t \hat m^t}  , \frac{\pi m_X^t
             \hat m^t\,  \sqrt{N} F + \xi \sqrt{\pi  m_X^t \hat m^t}}{
             \pi m_X^t \hat m^t}   \right] \, ,\label{eq:mF_NL_i} \\ 
             \hat m^t &=& - \int {\rm d}w \, P(w)  \int
         {\rm d} p\,  {\rm
               d}z   \frac{  e^{-\frac{p^2}{2m_F^t m_X^t}}  e^{-\frac{(z-p)^2}{2[ \langle(z^0)^2 
             \rangle - m_F^t m_X^t]}} }{2\pi \sqrt{ m_F^t m_X^t (   \langle(z^0)^2 
             \rangle - m_F^t m_X^t  ) }} \,    \partial_p g_{\rm out}(p,h(z,w),
             \langle(z^0)^2 
             \rangle - m_F^t m_X^t ) \, , \label{eq:mhat_NL_i} 
\eea
where $ {\cal D}\xi $ is a notation for a Gaussian probability measure ${\rm d} \xi
e^{-\xi^2/2}/\sqrt{2\pi}$. We denoted $\langle(z^0)^2 \rangle =
N  \mathbb{E}  [ (F^0)^2 ]  \mathbb{E} [(X^0)^2 ]$. Here $f_X$ and
$f_F$ are the so-called input functions, they are defined using the prior
distributions $P_X$ and $P_F$ as
\bea
         f_X(\Sigma,T) &\equiv&  \frac {\int  {\rm d}X \, X P_X(X)
           e^{-\frac{(X-T)^2}{2\Sigma}}  }{   \int  {\rm d}X \,
           P_X(X)  e^{-\frac{(X-T)^2}{2\Sigma}}  }  \\
 f_F(Z,W) &\equiv& \sqrt{N}  \frac {\int  {\rm d}F \, F P_F(F)
           e^{-\frac{(\sqrt{N}F-W)^2}{2Z}}  }{   \int  {\rm d}F \,
           P_F(F)  e^{-\frac{(\sqrt{N} F-W)^2}{2Z}}  }  \, .
\eea
The output function $g_{\rm out}$ is defined using the
output probability $P_{\rm out}$ as
\be
       g_{\rm out} (\omega, y , V) \equiv \frac{ \int {\rm d}z P_{\rm
        out}(y|z)\,  (z-\omega) \,   e^{-\frac{(z-\omega)^2}{2V}}  }{  V \int {\rm d}z P_{\rm
        out}(y|z) e^{-\frac{(z-\omega)^2}{2V}}  } \, . \label{eq:def_gout_i}
\ee

The MSE is computed as ${\rm MSE}_F= \mathbb{E} [(F^0)^2 ] -
m_F$ and ${\rm MSE}_X= \mathbb{E} [(X^0)^2 ] -
m_X$ with $m_F$ and $m_X$ being fixed points of
(\ref{eq:mx_NL_i}-\ref{eq:mhat_NL_i}). 

The AMP-MSE is obtained from a fixed point reached with a so-called
{\it uninformative} initialization. The uninformative initialization
does not use any information about the 
seeked matrices $F^0,X^0$, it uses only the prior
distributions, it is defined as
\be
     m_F^{t=0} = N \mathbb{E}[ F
]^2\, ,  \quad \quad   m_X^{t=0} = \mathbb{E}[ X]^2\, .  \label{eq:DE_NL_init_i}
\ee
In case the prior distribution depends on another random variables,
e.g. in case of matrix calibration, we take additional average with
respect to that variable. 
If the above initialization gives $m_F^{t=0}=0$ and $m_X^{t=0}=0$ then
this is a fixed point of the above iterative equations. This is due to the
permutational symmetry between the columns of matrix $F$ and rows of matrix $X$. To obtain a nontrivial
fixed point we initialize at $m_F^{t=0}=\eta$ for some very small
$\eta$, corresponding to an infinitesimal prior information about the
matrix elements of the matrix $F$. 

To compute the MMSE we need to initialize the iterations in an {\it
  informative } way, using the knowledge of the true $X^0$ and
$F^0$. This informative initialization is defined as an infinitesimal
perturbation of 
\be
     m^{t=0}_F=N  \mathbb{E}[ (F^0)^2
]\, ,  \quad \quad  m^{t=0}_X= \mathbb{E}[ (X^0)^2]\, . \label{eq:init_inf}
\ee
If the resulting fixed point agrees with the AMP-MSE then this is also
the MMSE. In case that this informative initialization leads to a
different fixed point than the AMP-MSE then the MMSE is given by the
one of them for which the following free entropy  is larger
\bea
&& \phi(m_F,m_X,\hat m_F = \pi m_X \hat m,\hat m_X = \alpha m_F \hat m )=\\
&& \alpha \pi \int {\rm d}y\,  {\cal D} \xi \, 
{\cal D}u^0 P_{\rm out} \left (y| \sqrt{Q_F^0 Q_X^0-m_F m_X} u^0
  +{\sqrt{m_F m_X}} \xi \right) 
\log \left(\int {\cal D} u \, P_{\rm out} \left (y| \sqrt{Q_F^0 Q_X^0-m_F m_X} u +\sqrt{m_F m_X} \xi \right ) \right ) \nonumber \\
&&+ \alpha \left(- \frac{\hat{m}_F m_F}{2}+
\int {\cal D}\xi\,  {\rm d}F^0e^{-\frac{N\hat{m}_F}{2}(F^0)^2+\sqrt{N\hat{m}_F}\xi F^0} P_{F}(F^0)
\log\left (\int {\rm d}F e^{-\frac{N\hat{m}_F}{2}F^2+\sqrt{N\hat{m}_F}
    \xi F} P_F(F)\right )  \right) \nonumber \\
&& + \pi \left (- \frac{\hat{m}_X m_X}{2}+
\int {\cal D}\xi \, {\rm d}X^0
e^{-\frac{\hat{m}_X}{2}(X^0)^2+\sqrt{\hat{m}_X}\xi X^0}P_{X}(X^0)
\log\left (\int {\rm d}X e^{-\frac{\hat{m}_X}{2}X^2+\sqrt{\hat{m}_X} \xi X} P_X(X) \right )  \right ), \nonumber 
\label{NishimoriRSfreeenergy_i}
\eea

Sections \ref{sec:AMP}, \ref{Sec:Bethe}, \ref{Sec:AA} are devoted to the derivation of these formulas for MMSE
and AMP-MSE. In section \ref{sec:phase} we then give a number of examples of
phase diagram for specific applications of the matrix factorization problem.  

We reiterate at this point that whereas our results are exact results
within  the
theoretical physics scope of the cavity/replica method, we do not
provide their proofs and their mathematical status is that of
conjectures. We devote
section \ref{sec:nonrigorous} to explaining the reasons of why this is so.  The situation is comparable to the predictions-conjectures that were
made for the satisfiability threshold \cite{MezardParisiZecchina02}, which have been partly
established rigorously recently \cite{DingSlySung14}, or the predictions-conjectures that were
made for the CDMA problem \cite{Tanaka02,GuoVerdu05}, and were later proved by \cite{montanari2006analysis,BayatiMontanari10}.   

\subsection{Applications of matrix factorization} 
\label{sec:examples}

Several important problems which have received a lot of attention
recently are special cases of the matrix factorization problem as set above. In this paper we will
analyse the following ones. 

\paragraph{Dictionary learning.} This is a basic example of finding a
representation of data that is advantageous in some
way. Representation learning \cite{bengio2013representation} is a concept behind many machine learning
applications, including the deep learning framework which has become
very popular in recent time
\cite{lecun2015deep}. In the context of representation learning
dictionary learning is sometimes referred to as {\it sparse coding} \cite{bengio2013representation}. 

Dictionary learning relies on the fact that
signals of interest are often sparse in some basis; this property is widely
used in data compression and more recently in compressed sensing. A lot
of work has been devoted to analyzing bases in which different data are sparse.
The goal of dictionary learning is to infer a basis in which the data
are sparse based purely on a large number of samples from the
data. The $M \times P$ matrix $Y$ then represents the $P$ samples of
$M$-dimensional data. The goal is to decompose $Y=F X + W$ into a $M
\times N$ matrix $F$, and a $N\times P$ sparse matrix $X$, $W$ is the noise. 

In this paper we will analyse the following teacher-student scenario of
dictionary learning. We will generate a random Gaussian matrix $F^0$ with iid
elements of zero mean and variance $1/N$, and a random Gauss-Bernoulli
matrix $X^0$ with fraction $0<\rho<1$ of non-zero elements. The non-zero
elements of $X^0$ will be iid Gaussian with mean $\overline X$ and
variance~$\sigma$.  The noise $W^0$ with elements $w^0_{\mu l}$ is also
iid Gaussian with zero mean and variance $\Delta$. We hence have 
\bea
        P_F(F_{\mu i}) &=& P_{F^0}(F_{\mu i}) = \frac{1}{\sqrt{2\pi/
            N}} e^{-\frac{N F^2_{\mu
              i}}{2}} \, , \label{eq:DL_PF}\\
       P_X(X_{il}) &=& P_{X^0}(X_{i l}) = (1-\rho) \delta(X_{il}) + \frac{\rho}{\sqrt{2
           \pi  \sigma}} e^{-\frac{(X_{il}-\overline X)^2}{2\sigma}}
       \, ,  \label{eq:DL_PX}\\ 
       P_{\rm out}(y_{\mu l} | z_{\mu l}) &=& \frac{1}{\sqrt{2\pi
           \Delta}} e^{-\frac{(y_{\mu l}-z_{\mu l})^2}{2\Delta}} \, , \label{eq:DL_Pout}
\eea
where $\delta(X) $ is the Dirac delta function. 
The goal is to infer $F^0$ and $X^0$ from the knowledge of $Y$ with
the smallest possible number of samples $P$.

In the noiseless case, $\Delta=0$, exact reconstruction might be
possible only when the observed information is larger that the
information that we want to infer. This provides a simple counting
bound on the number of needed samples~$P$:
\be
           P \ge \frac{\alpha }{ \alpha - \rho   } N\, .  \label{eq:bound_DL}
\ee

Note that the above assumptions on $P_F$, $P_X$ and $P_{\rm out}$ likely do not hold in any realistic
problem. However, it is of theoretical interest to analyze the average theoretical
performance and computational tractability of such a problem, as it gives a well
defined benchmark. Moreover, we anticipate
that if we develop an algorithm working well in the above case it
might also work well in many real applications where the above
assumptions are not satisfied, in the same spirit as the approximated
message passing algorithm derived for compressed sensing with zero
mean Gaussian measurement matrices \cite{DonohoMaleki09} works also
 for other kinds of matrices. 

Typically, we would look for an invertible basis $F^0$ with
$N=M$, in that case we speak of a {\it square} dictionary. However, with the compressed-sensing application in mind, it is
also very interesting to consider that $Y$ might be under-sampled
measurements of the actual signal, corresponding then to
$\alpha=M/N<1$.  Hence we will be interested in the whole range $0 < \alpha
\le 1$. The regime of $\alpha<1$ corresponds to an {\it overcomplete}
dictionary, in which case each of the $P$ measurements $\vec y_l$ is a sparse linear
combination of the columns (atoms) of the dictionary. 

We remind that the $N$ columns of the matrix $F^0$ can always be permuted
arbitrarily and multiplied by $\pm 1$. This is also true for rows of
the matrix $X^0$: all these operations do not  change $Y$, nor the posterior probability distribution. This is hence an intrinsic
freedom in the dictionary learning problems that we have to keep in
mind. Note that many works consider the dictionary to be column
normalized, which lifts part of the degeneracy in some optimization
formulations of the problem. In our setting the equivalent of column normalization
is asymptotically determined by the properties of the prior
distribution.

\paragraph{Blind matrix calibration.}

In dictionary learning one does not have any specific information about
the elements $F^0_{\mu i}$. However in some applications of compressed sensing
one might have an approximate knowledge of the measurement
matrix: it is often possible to use known
samples of the signal $X$ in order to calibrate the matrix in a supervised
manner (i.e. using known training samples of the signal). Sometimes, however, the known training samples are not
available and hence the only way to calibrate is to measure a number
of unknown samples and perform their reconstruction and calibration of
the matrix at the same time, such a scenario is called the {\it blind
  calibration}. 

In blind matrix calibration, the properties of the signal and the output function are the same as
in dictionary learning,
eqs.~(\ref{eq:DL_PX}-\ref{eq:DL_Pout}). As for the matrix
elements $F^0_{\mu i}$ one knows a noisy estimation
$F'_{\mu i}$. In this work we will assume that this estimation was
obtained from $F^0_{\mu i}$ as follows
\be
        F'_{\mu i} = \frac{F^0_{\mu i} + \sqrt{\eta} \xi_{\mu i} }{
          \sqrt{1+\eta} } \, ,\label{eq:eta_BMC}
\ee
where $\xi_{\mu i}$ is a Gaussian random variables with zero mean and
variance $1/N$. This way, if the matrix elements $F^0_{\mu i}$ have
zero mean and variance $1/N$, then the same is true for  the elements
$F'_{\mu i}$. 

The control parameter $\eta$ is then quantifying how well one
knows the measurement matrix.
It provides a way to interpolate between the pure
compressed sensing $\eta=0$, where one knows the measurement matrix $F^0$, and the dictionary learning problems $\eta
\to \infty$. Explicitly, the prior distribution of a given
element of the matrix $F$ is 
\be
       P_{F}(F_{\mu i}) = {\cal N}\left(\frac{F'_{\mu
             i}}{\sqrt{1+\eta}},\frac{\eta}{N(1+\eta)}\right)\ , \label{eq:PF_cal}
\ee
where ${\cal N}(a,b)$ is a Gaussian distribution with mean $a$ and
variance $b$.

\paragraph{Low-rank matrix completion}

Another special case of matrix factorization that is often studied is
the low-rank matrix completion. In that case one ``knows'' only a small (but
in our case finite when $N\to \infty$) fraction $\epsilon$ of elements of the $M \times P$ matrix
$Y$. Also one knows which elements are known and which are not; let us
call ${\cal M}$ the set on which elements are known, it is a set of
size $\epsilon M P$. In this case 
the output function is:
\bea 
                P_{\rm out}(y_{\mu l}| z_{\mu l}) &=&
               \frac{1}{\sqrt{2\pi \Delta}} e^{-\frac{(y_{\mu l}-z_{\mu l})^2}{2\Delta}} \quad {\rm if} \quad  \mu l
                \in {\cal M}\ , \nonumber\\
                &=& \frac{1}{\sqrt{2\pi }} e^{-\frac{y_{\mu l}^2}{2}}   \quad {\rm if} \quad  \mu l 
                \notin {\cal M}\, .
\eea
The precise choice on the function on the second line is arbitrary as
long as it does not depend on the $z_{\mu l}$. In what follows we will
assume that the $\epsilon M P$ known elements were chosen uniformly at
random.   

In low rank matrix completion $N$ is small compared to $M$ and
$P$, hence both $\pi$ and $\alpha$ are relatively large. Note,
however, that the limit we analyse in this paper keeps $\pi=O(1)$ and
$\alpha=O(1)$ while $N\to \infty$, whereas in many previous works on low-rank
matrix completion the rank was considered to be $O(1)$ and hence $\alpha$
and $\pi$ of order $O(N)$. Compared to those works the analysis here applies
to ``not-so-low-rank" matrix completion.  
The question is what fraction $\epsilon$ of elements of $Y$ needs
to be known in order to be able to reconstruct the two
matrices $F^0$ and $X^0$. 

For negligible measurement noise, $\Delta=0$, a simple counting bound
gives that the fraction of known elements we need for reconstruction
is at least 
\be
    \epsilon \ge \frac{\alpha + \pi}{\alpha \pi} \, .  \label{eq:bound_MC}
\ee

Again we will study the student-teacher scenario when a low-rank $Z$
is generated from $X^0$
and $F^0$ having iid elements distributed according to
eq. (\ref{eq:DL_PF}) and (\ref{eq:DL_PX}) with $\rho=1$ (no
sparsity). To construct $Y$ we keep a random fraction $\epsilon$ of 
elements of $Z$, and the goal is to reconstruct $X^0$ and $F^0$ from that
knowledge. 

We also note that variants on matrix completion where the output
channel $P_{\rm out}(y,z)$ keeps only the sign of $z$ were also
considered, called {\it 1-bit matrix completion}, and have some interesting
properties reported in \cite{davenport20141,davenportoverview15}.  

\paragraph{Sparse PCA and blind source separation}

Principal component analysis (PCA) is a well-known tool for dimensional
reduction. One usually considers the singular value decomposition (SVD) of a
given matrix and keeps a given number of largest values, thus
minimizing the mean square error between the original matrix and its
low-rank approximation. The SVD is computationally tractable, and
provides the minimization of the mean square error between the original matrix and its
low-rank approximation. However, with additional constraints there is
no general computationally tractable approach. 

A variant of PCA that is relevant for a number practical application requires
that one of the low-rank components is sparse. The goal is then to
approximate a matrix $Y$ by a product $F X$ where $F$ is a tall
matrix, and $X$ a wide sparse matrix. The teacher-student scenario for
sparse PCA that we will analyse in this paper uses
eq.~(\ref{eq:DL_PF}-\ref{eq:DL_Pout}) and the matrix dimensions are
such that $\pi=P/N$ and $\alpha=M/N$ are both large, but still of
order $O(1)$, and comparable one to the
other. Hence it is only the region of interest for  $\alpha$ and $\pi$
that makes this problem different from the dictionary learning. Note
that, in the same way as for the matrix completion, many works in the
literature consider $N = O(1)$ whereas here we have $N \to \infty$ in
such a way that $\pi = P/N= O(1)$ and $\alpha=M/N = O(1)$.  Hence we
work with low rank, but not as low as most of the existing
literature. For a statistical physics study following the lines of the
present paper where the rank is $O(1)$ see \cite{lesieur2015mmse}.

In the zero measurement noise case, $\Delta=0$, the simple counting
bound gives that the rank $N$ for which the reconstruction problem may
be solvable needs to be smaller than 
\be
   N \le \frac{MP}{M+ \rho P}\, .  \label{eq:SPCA_bound}
\ee

One important application where sparse PCA is relevant is the blind
source separation problem. Given $N$ $\rho$-sparse (in some known
basis) source-signals of dimension $P$, they are mixed in an unknown
way via a matrix $F$ into $M$ channel measurements $Y$. In blind
source separation typically both the number of sources $N$ and the number of
channels (sensors) $M$ are small compared to $P$. When $N<M$ we obtain an
overdetermined problem which may be solvable even for $\rho=1$. More
interesting is the undetermined case with the number of sensors smaller
than the number of sources, $M<N$, which would not be
solvable unless the signal is sparse $\rho<1$ (in some basis), in that
case the bound (\ref{eq:SPCA_bound}) applies.

\paragraph{Robust PCA}

Another variant of PCA that arises often in practice is the robust PCA, where the
matrix $Y$ is very close to a low rank matrix $F X$ plus a sparse
full rank matrix. The interpretation is that $Y$ was created as low
rank but then a small fraction of elements was distorted by a large additive
noise. The resulting $Y$ is hence not low rank. 

In this paper we will analyse a case of robust PCA when $F$ and $X$
are generated from eq.~(\ref{eq:DL_PF}-\ref{eq:DL_PX}) with $\rho=1$
and the output function is 
\be
    P_{\rm out}(y_{\mu l} | z_{\mu l}) = \epsilon \frac{1}{\sqrt{2\pi
           \Delta_s}} e^{-\frac{(y_{\mu l}-z_{\mu l})^2}{2\Delta_s}}  +
       (1-\epsilon) \frac{1}{\sqrt{2\pi \Delta_l}   } e^{-\frac{(y_{\mu
             l}-z_{\mu l})^2}{2 \Delta_l   }} \, ,\label{eq:RPCA}
\ee
where $\epsilon$ is the fraction of elements that were not largely distorted, 
$\Delta_s \ll 1$ is the small measurement noise on the non-distorted
elements, $\Delta_l$ is the large measurement noise on the distorted
elements. We will require $\Delta_l \approx \overline X^2+\sigma$ to
be comparable to the variance of $z_{\mu l}$ such that there is no
reliable way
to tell which elements were distorted by simply looking at the
distribution of $y_{\mu l}$. The parameters regime we are interested in here is $\pi$ and
$\alpha$ both relatively large and comparable one to another. 

In robust PCA in the zero measurement noise case, $\Delta_s=0$, the simple counting
bound gives that the fraction of non-distorted elements $\epsilon$ under which the
reconstruction may still
be solvable needs to satisfy the same bound as for matrix completion
(\ref{eq:bound_MC}). Indeed this counting bound does not distinguish between the
case of matrix completion when the position of known elements of $Y$ is known
and the case of RPCA when their positions are unknown.

\paragraph{Factor analysis }

One major objective of multivariate data analysis is to infer an appropriate mechanism from
which observed high-dimensional data are generated. Factor analysis (FA) is a representative methodology for this
purpose. Let us suppose that a set of $M$-dimensional vectors $\by_1, \by_2,\ldots, \by_P$ is given and its mean is set to zero by a
pre-processing. Under such a setting, FA assumes that each observed vector $\by_l$ is generated by $N(\le M)$-dimensional
{\em common factor} $\bX_l$ and $M$-dimensional {\em unique factor} $\bw_l$ as $\by_l = F\bX_l+\bw_l$, where $F \in  \mathbb{R}^{M\times N}$
 is 
termed the {\em loading matrix}. The goal is to
 determine the entire set of $F$, $X = (\bX_1, \bX_2, \ldots , \bX_P )$, and $W = (\bw_1,\bw_2, \ldots ,\bw_P )$ from only 
$Y = (\by_1, \by_2, \ldots , \by_P )$.
Therefore, FA is also expressed as a factorization problem of the form of $Y = FX +W$ in matrix terms.

The characteristic feature of FA is to take into account the site dependence of the output function as
\begin{eqnarray}
P_{\rm out}(y_{\mu l}|z_{\mu l}, \psi_\mu) =
\frac{1}{\sqrt{2\pi \psi_{\mu}}} e^{-\frac{(y_{\mu l}-z_{\mu
      l})^2}{2\psi_\mu}}\, ,
\label{FAout}
\end{eqnarray}
where  $\psi_\mu$ denotes the variance of the $\mu$-th component of the unique factor. In addition, it is normally assumed
that $\bX_l$ $(l = 1, 2, \ldots, P)$ independently obeys a zero mean multivariate Gaussian distribution, the variance of
which is set to the identity matrix in the basic case. These assumptions make it possible to express the log-likelihood
of $F, \Psi$ given $Y$ in a compact manner as $\log P(Y |F, \Psi) = 
\log \left (
\int \prod_{\mu=1,l=1}^{M,P}P_{\rm out}(y_{\mu l}|z_{\mu l},\psi_\mu)
\prod_{i=1,l=1}^{N,P} P_X(X_{i l}) {\rm d}X \right )
=-\frac{P}{2} 
\log \det \left (FF^{\rm T}+\Psi \right ) -\frac{1}{2} 
\sum_{l=1}^P
\by_l^{\rm T} \left (FF^{\rm T}+\Psi \right )^{-1} \by_l
+ const$, where 	$\Psi = (\psi_\mu \delta_{\mu \nu}) \in \mathbb{R}^{ M\times M}$. In a standard scheme, 
$F$ and $\Psi$, which parameterize the generative mechanism of the observed data $Y$, are determined by maximizing the
log-likelihood function \cite{Joreskog1969}. After obtaining these, the posterior distribution $P(X|Y, F^{\rm ML},	\Psi^{\rm ML})$ is used to estimate
the common factor $X$, where $F^{\rm ML}$	and $\Psi^{\rm ML}$ are the maximum likelihood estimators of $F$ and $\Psi$, respectively. Finally,
unique factor $W$ is determined from the relation $Y = FX +W$. Other heuristics to minimize certain discrepancies between
the sample variance-covariance matrix $P^{-1}Y Y^{\rm T}$ and $FF^{\rm T} +\Psi$ are also conventionally  used for determining $F$ and $\Psi$.
As an alternative approach, we employ the Bayesian inference for FA.

\paragraph{Non-negative matrix factorization }
In many application of matrix factorization it is known that both the
coefficient of the signals $X$ and the coefficients of the
dictionary $F$ must be non-negative. In our setting this can be taken
into account easily by imposing the distributions $P_X$ and $P_F$ to have
non-negative support. 
In the student-teacher scenario of this paper we can
hence consider the nonzero elements of $X$ to be $P_X(X)=0$ for
$X<0$ and $P_X(X)= N(\overline X, \sigma)$ for $X>0$, and analogously
for $P_F$.

\subsection{Related work and positioning of our contribution} 
Matrix factorization and its special cases as mentioned above are
well-studied 
problems with extensive theoretical and algorithmic literature
that we cannot fully cover here. We will hence only give
examples of relevant works and will try to be exhaustive only
concerning papers that are very closely related to our work (i.e. papers
using message-passing algorithms or analyzing the same phase
transitions in the same scaling limits).

The dictionary learning problem was identified in the work of
\cite{olshausen1996emergence,OlshausenField97} in the context of image
representation in the visual cortex, and the problem was studied
extensively since \cite{kreutz2003dictionary}.  Learning of
overcomplete dictionaries for sparse representations of data has many
applications, see e.g.  \cite{lewicki2000learning}. 
MOD \cite{Engan99} and K-SVD \cite{AharonElad06b} are two representative algorithms for the dictionary learning. 
Several authors studied the identifiability of
the dictionary under various (in general weaker than ours) assumptions,
e.g. \cite{AharonElad06,Vainsencher11,JenattonGribonval2012,spielman2012exact,arora2013new,agarwal2013exact,gribonval2013sample}.
An interesting view on the place of sparse and redundant
representations in todays signal processing is given in
\cite{elad2012sparse}.

A nice overview of recent progress on low-rank matrix factorization
from incomplete observations is given in \cite{davenportoverview15}. 
The two closely related problems of sparse principal component analysis or blind
source separation is also explored in a number of works, see e.g. 
\cite{lee1999blind,zibulevsky2001blind,bofill2001underdetermined,georgiev2005sparse,d2007direct}. 
A short survey on the topic with relevant references can be found in
\cite{gribonval2006survey}. 
Matrix completion is another problem that belongs to the class treated
in this paper. Again, many important works were devoted to this
problem giving theoretical guarantees, algorithm and applications, see
e.g. \cite{candes2009exact,candes2010power,candes2010matrix,keshavan2010matrix}. 
Another related problem is the robust principal component analysis
that was also studied by many
authors; algorithms and theoretical limits were analyzed in \cite{chandrasekaran2009sparse,yuan2009sparse,wright2009robust,candes2011robust}.

Our work differs from the mainstream of existing literature in its
general positioning. Let us mention here some of the main differences
with most other works.

\begin{itemize} 
   \item  Most existing works concentrate on finding theoretical
     guarantees and algorithms that work in the worst possible case of
     matrices $F$ and $X$. In our work we analyze the typical cases when elements of $F$ and
     $X$ are generated at random. Arguably a worst case analysis is
     useful for practical applications in that it provides some
     guarantee. On the other hand, in some large-size applications one
     can be confronted in practice with a
     situation which is closer to the typical case that we study
     here. Our typical-case analysis provides
     results that are much tighter than those usually obtained in
     literature in terms of both achievability and computational tractability.  For
     instance our results imply much smaller sample complexity, or much smaller
     mean-squared error for a given signal-to-noise ratio, etc. 
   \item Our main contribution is the asymptotic phase diagram of
     Bayes-optimal inference in matrix factorization. Special cases of
     our result cover important problems in signal processing and
     machine learning. In the present work we do not concentrate on the
     validation of the associated approximate message-passing
     algorithm on practical examples, nor on its comparison with
     existing algorithms. A contribution in this direction can be
     found in \cite{krzakala2013phase,parker2014bilinear}.
   \item A large part of existing machine-learning and
     signal-processing literature provides theorems.
     Our work is based on statistical physics methods that are
     conjectured to give exact results. While many results obtained
     with these methods on a variety of problems have been indeed
     proven later on, a general proof that these methods are rigorous is not known yet.
   \item Many existing algorithms are based on convex relaxations of
     the corresponding problems. In this paper we analyze the
     Bayes-optimal setting. Not surprisingly, this setting gives a much
     better performance than convex relaxation. The reason why it is
     less studied is that it is often considered as hopelessly
     complicated from the algorithmic perspective; however some recent
     results using message-passing algorithms which stand at the core
     of our analysis have shown very good performance in the
     Bayes-optimal problem. Besides, the Bayes-optimal offers an
     optimal benchmark for testing algorithmic methods.
   \item When treating low-rank matrices we consider the rank to be a finite
     fraction of the total dimension, whereas most of existing
     literature considers the rank to be a finite constant. The AMP
     algorithm derived in this paper can of course be used as a
     heuristics also in the constant rank case, but our results about
     MMSE are exact only in the limit where the rank is a finite
     fraction of the total dimension. Also note that for the special
     case of rank one the AMP algorithm was derived and
     analysed rigorously \cite{rangan2012iterative,deshpande2014information}. 
   \item When treating sparsity we consider the number of non-zeros is
     a finite fraction of the total dimension, whereas existing
     literature often considers a constant number of non-zeros. Again the AMP
     algorithm derived in this paper can of course be used as a
     heuristics also in the constant number of non-zeros case, but our results about
     MMSE are exact only in the limit where the density is a finite
     fraction of the total dimension.  
\end{itemize}

The paper is organized as follows: 
First, Sec.~\ref{sec:background},
we explain the statistical physics background
for the Bayes-optimal inference.
In Sec.~\ref{sec:AMP} we give
a detailed derivation of the approximate message passing algorithm for the matrix
factorization problem. This algorithm is equivalent to maximization
of the Bethe free entropy, whose expression is discussed in Sec.~\ref{Sec:Bethe}. These first two sections thus state our algorithmic
approach to matrix factorization.
The asymptotic performance of the AMP algorithm and the Bayes-optimal MMSE are
analyzed in Sec.~\ref{Sec:AA} using two technics: i) the state
evolution of the AMP algorithm and ii) the replica
method. As usual, the two methods are found to agree and to give the same
predictions.
We then use these results in order to study some exemples of matrix
factorization problems in Sec.~\ref{sec:phase}. In particular we derive
the phase diagram, the MMSE and the sample complexity of dictionary
learning, blind matrix calibration, sparse PCA, blind source
separation, low rank matrix completion, robust PCA and factor
analysis.
Our results are summarized and discussed in the conclusion in
Sec.~\ref{sec:conclusion}.

\section{Elements of the background from statistical physics}
\label{sec:background}

\subsection{The Nishimori identities}
\label{sec:Nishi}

There are important identities that hold for the Bayes-optimal
inference and that simplify many of the calculations that follow. In
the physics of disordered systems these identities are known as the
Nishimori identities \cite{OpperHaussler91,Iba99,NishimoriBook}. 
Basically, they
follow from the fact that the true
signal $F^0$, $X^0$ is an equilibrium configuration with respect to the Boltzmann
measure $P(F,X|Y)$ (\ref{eq:post}). Hence many properties of the true
signal $F^0$, $X^0$ can be computed by using averages over
the distribution $P(F,X|Y)$ even if one does not  know $F^0,X^0$  precisely (\ref{eq:post}).

In order to derive the Nishimori identities, we need to define three types of
averages: the thermodynamic average, the double thermodynamic
average, and the disorder average.
\begin{itemize}
\item
 Consider a function $A(F,X)$ depending on a ``trial''
configuration $F$ and $X$. We define the ``thermodynamic average'' of
$A$ given $Y$ as:
\be
       \langle A(F,X) \rangle_{F,X|Y} \equiv \int {\rm d} X \,  {\rm d} F  \,
       A(F,X)\,  P(F,X
       |Y)\, ,
\ee
where $ P(F,X|Y)$ is given by Eq.~(\ref{eq:post}). The thermodynamic
average $\langle A(F,X) \rangle_{F,X|Y}$ is a function of $Y$.

\item
Similarly, for a
function $A(F_1,X_1,F_2,X_2)$ that depends on {\it two} trial configurations
$X_1$, $F_1$ and $X_2$, $F_2$, we define the ``double thermodynamic
average'' of 
$A$ given $Y$ as:
\be
       \langle \langle A(F_1,X_1,F_2,X_2) \rangle \rangle_{F_1,X_1,F_2,X_2|Y} \equiv \int
       {\rm d} X_1 \,  {\rm d} F_1 \,  {\rm d} X_2 \,  {\rm d} F_2 \,
       A(F_1,X_1,F_2,X_2) \, P(F_1,X_1
       |Y) P(F_2,X_2
       |Y) \, ;
\ee
this is again a function of $Y$.

\item
For a function $B$
that depends on the measurement $Y$ and on the true signal $F^0,X^0$, we define the ``disorder average'' as
\be
             [B(F^0,X^0,Y)]_{F^0,X^0,Y}  \equiv \int {\rm d}Y \, {\rm d}X^0 \,  {\rm d}F^0 
             \,  P_X(X^0) \, P_F(F^0) \,  P_{\rm out}(Y| F^0 X^0) \, B(F^0,X^0,Y) \, .
\ee
Note that if the quantity $B$ depends only on $Y$, then we have
 \be
[B(Y)]_Y \equiv \int
{\rm d}Y {\cal Z}(Y) B(Y)\, .
\label{disav_simple}
\ee
This is simply because the partition function ${\cal
  Z}(Y)$ is  ${\cal Z}(Y) = \int {\rm d}X^0 \,
{\rm d}F^0 \,  P_X(X^0) \, P_F(F^0) \, P_{\rm out}(Y| F^0 X^0) $.
\end{itemize}

Let us now derive the Nishimori identities. We consider a
function $A(F,X,F^0,X^0)$ that depends on the trial configuration
$F,X$ and on the true signal $F^0,X^0$. Its thermodynamic average
$\langle A(F,X,F^0,X^0) \rangle_{F,X|Y}$ is a function of the
measurement $Y$ and  the true signal $F^0,X^0$. The disorder average
of this quantity can be written as
\bea
 &&[ \langle A(F, X,F^0,X^0) \rangle_{F,X|Y} ]_{F^0,X^0,Y}\nonumber \\
& &= 
 \int {\rm d}F^0 \, \, {\rm d}X^0 \,  {\rm d}Y
 \ P_X(X^0)\, P_F(F^0) \,  P_{\rm out}(Y|F^0 X^0)   \int {\rm d} F  \,  {\rm d} X\, A(F,X,F^0,X^0)\,
   P(F,X|Y)\nonumber \\ &&=
\int {\rm d}Y\, {\cal  Z}(Y) \int
 {\rm d}F^0 \, \, {\rm d}X^0 
{\rm d} F  \,  {\rm d} X
\,  A(F^0,X^0,F,X) \frac{ P_X(X^0) P_F(F^0) P_{\rm out}(Y| F^0 X^0)
}{{\cal Z}(Y)} P(F,X|Y)\ .
\eea
In this last expression, renaming $F,X$ to $F_1,X_1$ and $F^0,X^0$ to
$F_2,X_2$,
 we see that the average over $F,X,F^0,X^0$ is nothing but
the double thermodynamic average:
 \bea
[ \langle A(F, X,F^0,X^0) \rangle_{F,X|Y} ]_{F^0,X^0,Y} &=&\int {\rm
  d}Y\,  {\cal Z}(Y)\int  {\rm d}F_1\,  {\rm d}X_1 \,  {\rm
  d}F_2 \, {\rm d}X_2 \, A(F_1,X_1,F_2,X_2)  P(F_1, X_1|Y)
P(F_2,X_2|Y)
\nonumber \\
&&= [ \langle \langle A(F_1,X_1,F_2,X_2) \rangle \rangle_{F_1,X_1,
  F_2,X_2|Y} ]_{Y} \ ,
\eea
where in the last step we have used the form (\ref{disav_simple}) of
the disorder average.

The identity
\bea
[ \langle A(F, X,F^0,X^0) \rangle_{F,X|Y} ]_{F^0,X^0,Y} = [ \langle \langle A(F_1,X_1,F_2,X_2) \rangle \rangle_{F_1,X_1, F_2,X_2|Y} ]_{Y} \label{Nishi_gen}
\eea
is the general form of the Nishimori
identity. Written in this way, it holds for many inference problems where the
model for signal generation is known.

Self-averaging is a property that is assumed to hold for many
quantities in statistical physics. Self-averaging means that  in
the thermodynamic limit (as introduced in the first paragraph of
section \ref{sec:main_results}) 
the distribution of $\langle A(F,X,F^0,X^0) \rangle_{F,X|Y}$
concentrates with respect to the realization of disorder $Y,F^0,X^0$.
In other words for large system sizes the quantity
$\langle A(F,X,F^0,X^0) \rangle_{F,X|Y}$ converges with probability
one to its average over disorder, $[ \langle A(F, X,F^0,X^0)
\rangle_{F,X|Y} ]_{F^0,X^0,Y} $.  Self-averaging implies the existence of
the thermodynamic limit and is in general very challenging to prove rigorously. Self-averaging makes the
Nishimori identity very useful in practice.

To give one particularly useful example, let us define $m_X=(1/(PN))\, \sum_{il}
X^0_{il}  X_{il} $ and $q_X= (1/(PN))\, \sum_{il} X^1_{il}  X^2_{il}
$. The
Nishimori identity states that  $[\langle m_X\rangle_{F,X|Y} ]_{F^0,X^0,Y}=[\langle \langle
q_X\rangle\rangle]_{F_1,X_1, F_2,X_2|Y} ]_{Y}$. Assuming that the
quantities $m_X$ and $q_X$ are self-averaging,  we obtain in the
thermodynamic limit, for almost all $Y$: $\langle m_X\rangle_{F,X|Y}=\langle \langle
q_X\rangle\rangle_{F_1,X_1, F_2,X_2|Y}$. Explicitly, this gives:
\bea
(1/(PN))\, \sum_{il}
X^0_{il} \langle  X_{il}\rangle_{F,X|Y}=(1/(PN))\, \sum_{il} \left(\langle
X_{il}\rangle_{F,X|Y} \right)^2 + o_N(1)\ . \label{eq:Nish_simplest}
\eea
where $o_N(1)$ is going to zero as $N\to \infty$. 
This is a remarkable identity concerning the mean
of $X_{il}$ with the posterior distribution $\nu_{il}(X_{il})$.  The left-hand
side measures the overlap between this mean and the sought true value
$X^0_{il}$. The right-hand side measures the self overlap of the mean,
which can be estimated without any knowledge of the  true value
$X^0_{il}$, by generating two independent samples from $P(F,X|Y)$. 

By symmetry, all these examples apply also to averages and
functions of the matrix $F$: 
\bea
(1/M)\, \sum_{\mu l}
F^0_{\mu l} \langle  F_{\mu l}\rangle_{F,X|Y}=(1/M)\, \sum_{il} \left(\langle
F _{\mu l}\rangle_{F,X|Y} \right)^2 + o_N(1)\ .
\eea

Validity of the relation (\ref{eq:Nish_simplest}) also follows
straightforwardly from the self-averaging property and an elementary
relation for conditional expectations $ \mathbb{E} (
\mathbb{E}(X|Y) X) = \mathbb{E}( (\mathbb{E}(X|Y))^2) $.

Another example of a Nishimori identity used in our derivation is
eq. (\ref{eq:gout_NL}). Note that this one is crucial to maintain certain
variance-related variables strictly positive as explained in section
\ref{Sec:Nish}. Without eq. (\ref{eq:gout_NL}) the AMP algorithm may have
pathological numerical behavior causing that some by definition strictly positive quantities
become negative. Deeper understanding of this problem in cases where Nishimori
identities cannot be imposed is under investigation. 
 
Nishimori identities also greatly simplify the state evolution
equations of section \ref{sec:AA_NL}. In general 
state evolution involves nine coupled equations. However, after imposing Nishimori identities,
eqs. (\ref{eq:NL_qm}-\ref{eq:NL_chi}), we are left with only three
coupled equations.

\subsection{Easy/hard and possible/impossible phase transitions}

Fixed points of
equations (\ref{eq:mx_NL_i}-\ref{eq:mhat_NL_i}) allow us to evaluate the
MMSE of the matrix $F$ and of the
matrix $X$. We
investigate two fixed points: The one that is reached from the
``uninformative'' initialization (\ref{eq:DE_NL_init_i}), and the fixed
point that is reached 
with the informative (planted) initialization (\ref{eq:init_inf}). 
If these two initializations lead to a different fixed point
it is the one with the largest value of the Bethe free entropy
(\ref{NishimoriRSfreeenergy_i}) that corresponds to the Bayes-optimal
MMSE. The reason for this is that the Bethe free entropy is the exponent of the
normalization constant in the posterior probability distribution: a
larger Bethe free entropy is hence related to a fixed point that
corresponds to an exponentially larger probability. 
Furthermore,  in cases where the uninformative initialization does not
lead to this Bayes-optimal MMSE we anticipate that the optimal solution is
hard to reach for a large class of algorithms.

Depending on the application in question and the value of parameters
($\alpha,\pi,\rho,\epsilon,\dots$) we can sometimes identify a phase
transition, i.e. a sharp change in behavior of the MMSE. As in
statistical physics it is instrumental to distinguish two kinds of
phase transitions 
\begin{itemize}
\item {\it Second order phase transition:} In this case there are two
  regions of parameters. In one, the recovery performance
  is poor (positive MMSE); in the other one the recovery is perfect
  (zero MMSE). This situation can arrive only in  zero noise, with
  positive noise there is a smooth ``crossover'' and the transition
  between the phase of good and poor recovery is not sharply defined. 
 Interestingly, as we will see, in all examples where we observed the
 second order phase transition, its location coincides with the simple
 counting bounds discussed in Section \ref{sec:examples}.  
 In the phase with poor MMSE, there is simply not enough
 information in the data to recover the original signal (independently
 of the recovery algorithm or its computational complexity). 
 Experience from statistical physics tells us that in problems where
 such a
 second order phase transition happens, and more generally in cases
 where the state evolution (\ref{eq:mx_NL_i}-\ref{eq:mhat_NL_i}) has a unique fixed point, there is no
 fundamental barrier that would prevent good algorithms to
 attain the MMSE. And in particular our analysis suggest that in the
 limit of large system sizes the AMP algorithm derived in this paper
 should be able to achieve this Bayes-optimal MMSE. 

\item {\it First order phase transition:} A more subtle phenomenology
  can be observed in the low-noise regime of dictionary learning, sparse PCA
  and blind matrix calibration. In some region of parameters, that we
  call the ``spinodal" region, the
  informative (planted) and uninformative initializations do not
  lead to the same fixed point of the state evolution equations. 
  The spinodal regime itself may divide into two parts - the one where
  the uninformative fixed point has a larger free entropy, and the
  ``solvable-but-hard" phase
  where the informative fixed point has a larger free entropy. The
  boundary between these two parts is the first order phase
  transition point. We anticipate that in the solvable-but-hard phase
  the Bayes-optimal MMSE is not achievable by the AMP algorithm nor by
  a large class of other known algorithms. 
  The first order phase transition is associated with a
  discontinuity in the MMSE. The MSE reached from the uninformative
  initializations will be denoted AMP-MSE and is also discontinuous. 
\end{itemize}

In the case of the first order phase transition it will hence be
useful to distinguish in our notations between the minimal achievable
MSE that we denote MMSE, and the MSE achievable by the AMP-like
algorithms that we denote AMP-MSE. When MMSE$=$AMP-MSE in the large
$N$ limit we say that AMP is asymptotically optimal. The region where
in the large size limit MMSE$<$AMP-MSE is called the {\it spinodal}
regime/region.

\subsection{Why our results are not rigorous? Why do we conjecture
  that they are exact?}
\label{sec:nonrigorous}

In this section we aim at summarizing the assumptions that we make but
do not prove along the derivation of our results presented in sections
\ref{sec:AMP}, \ref{Sec:Bethe}, \ref{Sec:AA}. In section
\ref{sec:phase} we then give the MMSE
and AMP-MSE as obtained from numerical evaluation of eqs.
(\ref{eq:mx_NL_i}-\ref{eq:mhat_NL_i}). 

\subsubsection{The statistical-physics strategy for solving the problem}

Let us first state the general strategy from a physics point of
view. We use two complementary approaches, the cavity method and the
replica method. 

Our main goal is to derive the MMSE for the matrix
factorization problem in the Bayes-optimal setting. Our main tool is
the cavity method, and it presents two advantages with respect to the
replica method:
 1) during the derivation of the MMSE we obtain the AMP
algorithm as a side-product and  2) based on experience of the last
three decades, it is likely that a rigorous
proof of our conjectures will follow very closely this path. In the cavity
method we first
assume that there is a fixed point of the belief propagation equations
that correctly describes the posterior probability distribution, and that the BP equations (initialized randomly, or in the case where a first order
phase transition is present, initialized close to the planted
configuration) converge to it. Then we are interested in the average
evolution of the
BP iterations. This can be understood through a statistical analysis of the BP equations
in which we keep only the leading terms in the large $N$ limit, an approach
called the cavity method \cite{MezardParisi87b} in the physics literature.  The result of this analysis are the state evolution
equations that should be asymptotically exact, they do not depend on the size of the system $N$ anymore.  
There is
one known way by which the assumptions made in this derivation can
fail, this way is
called replica-symmetry-breaking (RSB). Fortunately, in the Bayes
optimal inference,  RSB is excluded, as we explain below. 

Our second analysis uses the replica method. In the replica approach one
computes the average of the $n$-th power of the normalization of the
posterior probability distribution. Then one uses the limit $n\to 0$ to
obtain the average of its logarithm. Doing this, one needs to evaluate a
saddle point of a certain potential and one assumes that this saddle point is
restricted to a certain class that is called ``replica symmetric",
eq. (\ref{RSorderparameter}). Again this assumption is justified in 
the Bayes optimal inference, because we know that there is no RSB in
this case.

Our two analyses, using the cavity method and the replica method, lead
to the same set of
closed equations for the MMSE of the problem. 
This important crosscheck, and the very reasonable nature of the
assumptions described below (based on our experience in using this
kind of approach in various settings) lead us to conjecture
that the results described in this paper are exact. 

\subsubsection{What are the assumptions?}

Let us start with the cavity method, which is in general much more
amenable to a rigorous analysis. Let us hence describe the assumptions
made  in section \ref{sec:AMP}. 
\begin{itemize}
\item
Our main assumption is that in the
leading order in $N$ the true marginal probabilities can be obtained
from a fixed point of the belief propagation equations (\ref{message_m})-(\ref{message_tn}) as written for
the factor graph in Fig.~\ref{FactorGraph}. For this to be possible
the incoming messages on the right hand side of
eqs. (\ref{message_m})-(\ref{message_tn}) must be
conditionally independent (in the thermodynamic limit, as
defined is section \ref{sec:main_results} ), in which case we can write the joint
probabilities as products over the messages. If the factor graph was a tree
this assumption would be obviously correct. On factor graphs that are locally
tree-like these assumptions can be justified (often also rigorously) by
showing that the correlations decay fast enough (see for instance \cite{MezardMontanari09}).
The factor graph in
Fig.~\ref{FactorGraph} is far from a tree. However, from the point of view of conditional independence
between messages, this factor graph resembles the one of compressed
sensing for which the corresponding results were proven rigorously
\cite{BayatiMontanari10,DonohoJavanmard11}. It is hence reasonable to
expect that matrix factorization belongs to the same class of problems
where BP is asymptotically exact in the above sense.  

\item 
We assume that the iterations of equations
(\ref{message_m})-(\ref{message_tn}) converge for very large systems
($N\to \infty$) to this fixed point that
describes the true marginal probabilities, provided that the iteration
is started from the right initial condition (in the presence of a first order phase
transition we need to consider initialization of the iteration in the
planted configuration in order to reach the right fixed point). 
\end{itemize}

Given these assumptions the rest of section \ref{sec:AMP} is
justified by the fact that in the derivation we only neglect terms that do not contribute to
the final MSE in the thermodynamic limit. 

The main assumptions made in the replica calculation is
that ``self-averaging'' applies (this is a statement on the concentration
of the measure, which basically assumes that the averaged properties are
the same as the properties of a generic single instance of the problem) and that we can exchange the limits $\lim_{n\to 0}$ and $\lim_{N\to
  \infty}$. On top of these, it relies on the replica symmetric
calculation, which is justified  here because we are interested in
evaluating Bayes optimal MMSE.  Unfortunately in the mathematical literature there is so
far very little progress in proving these basic assumptions of the
replica method, even in problems that are much easier than matrix
factorization. It can be remembered that the original motivation for
introducing the cavity method \cite{MezardParisiVirasoro86} was precisely to provide an alternative
way to the replica computations, which could be stated in clear
mathematical terms.

In this paper we follow the physics strategy and we hence provide a
conjecture for evaluating the MMSE in matrix
factorization. We do anticipate that further work using the path of
the cavity approach will
eventually provide a proof of our conjecture. 
Let us mention that rigorous results exist for closely related
problems, namely the problem of compressed sensing
\cite{BayatiMontanari10,DonohoJavanmard11}, and also rank-one matrix
factorization \cite{rangan2012iterative,deshpande2014information}. None of these
apply directly to our case where the rank is extensive and where the
matrix $F$ is not known. A non-trivial generalization of these works will be needed in order to prove our results.

\subsubsection{Cavity and replica method for other problems}
The cavity and replica
methods, as we use them use in the present paper,  rely on the set of assumptions listed above. It is worth to note that over
the years they have been applied successfully to a wide range of problems. 
There is nowadays a range of results that can be ``derived'' (and in many
of those cases that was the original derivation)
using the cavity or/and the replica methods and that have
been proven fully rigorously. The list is long, but to cite a couple
of interesting examples we can mention the proof of the satisfiability threshold
conjecture \cite{DingSlySung14}, performance of low-density parity check codes
\cite{RichardsonUrbanke08}, the optimal solution of the linear assignment problem
\cite{aldous2001zeta}, the analysis of the compressed sensing problem
\cite{DonohoMaleki09,BayatiMontanari10}, and many more.

\subsection{Absence of replica symmetry breaking in Bayes-optimal inference}

The study of the Bayes-optimal inference, i.e. the case when the prior
probability distribution matches the true distribution with which the
data have been generated, is fundamentally
simpler than the general case. The fundamental reason for this
simplicity
 is that, in the Bayes-optimal case, a major complication referred to as {\it static replica symmetry
breaking} (RSB) in statistical physics literature
does not appear  \cite{NishimoriBook}. RSB is a source of computational complications in
many optimization and constraint satisfaction problems such as the
random K-satisfiability \cite{MezardParisiZecchina02} or models of
spin glasses \cite{MezardParisi87b}. 

One common way to define RSB is via the overlap distribution $P(q)$
\cite{NishimoriBook,MezardParisi87b,MezardMontanari09}. We define an
overlap between two configurations $X^1$ and $X^2$ as $q_X(X^1,X^2) = (1/(PN)) \sum_{il}
X^1_{il} X^2_{il} $. The overlap distribution function is then defined
(for a given realization of the problem $X^0, F^0, Y$) as the
probability distribution of  $q_X(X^1,X^2)$ over the posterior measure
$P(X,F|Y)$. For problems that do not have any permutational symmetry
(such as the permutation of the $N$ rows of $X$ and columns of $F$), we say the system is replica symmetric when $\lim_{N\to \infty} P(q) = \delta (q-q^0)$ for almost
all choices of  $X^0, F^0, Y$, and that there is replica symmetry
breaking otherwise. The number of steps of RSB then correspond to the
number of additional delta peaks in the limiting distribution
$P(q)$. When the limiting distribution $P(q)$ has a continuous support
then we talk about full-RSB. For systems with permutational symmetry
the corresponding number of peaks gets simply multiplied by the size of the symmetry group. 

Let us further define a so-called magnetization of configuration $X^1$
with respect to some reference configuration $X^0$ as
$m_X(X^0,X^1)= (1/(PN)) \sum_{il} X^1_{il} X^2_{il}$. Such a
magnetization and its variance can be computed from first and second
derivative of the free entropy density. The derivative is taken with
respect to a so-called magnetic field which is an auxiliary parameter
introduced in the posterior probability for this purpose. In
statistical physics, we expect (and in some case, actually prove) some
standard properties of the free entropy such as its self-averaging and
its analyticity (except at phase transition points) with respect to
such parameters. From this reasoning it follows that whereas the mean
of the magnetization is of order $O(1)$, its variance is of order
$O(1/N)$. Hence in the limit $N\to \infty$ the distribution (over
$P(X,F|Y)$) of magnetization is a delta peak.  If the reference
configuration $X^0$ is the planted configuration, then self-averaging
implies that the distribution of $m_X(X^0,X^1)$ with respect to both
its arguments converges to a delta function. From the Nishimori
identities it follows that the same has to be true for the
distribution of the overlaps $P(q)$. Therefore only a $P(q)$
converging to a delta function in the thermodynamic limit is allowed
in the Bayes optimal inference setting.
In a more technical side, it also follows from the Nishimori
identities that the two-point correlations between variables decay
sufficiently fast as proven in \cite{Montanari08estimating}, this
excludes the existence of the static RSB for the posterior measure
$P(F,X|Y)$.

Let us stress however that when there is a mismatch between the prior
distributions and the distributions from which the signal was truly
generated, then replica symmetry breaking can happen. The distribution
of overlaps may become non-trivial even if the distribution of
magnetization is not, and in order to provide some exact results or
conjectures for this situation our analysis would have to be revised
accordingly. This might turn out as very non-trivial (see for instance
the case of compressed sensing with the $\ell_p$-norm reconstruction
when $p<1$ \cite{KabashimaWadayama09}) and is left for future work.

\subsection{Learning of hyper-parameters}
\label{sec:EM}

To obtain the main results of this paper we assume that the priors $P_X$
and $P_F$ are known, including their related parameters. In practice,
even if one knows a parametric family of probability
distributions that approximates well the matrix elements of $F$ and $X$,
one often does not have the full information about the parameters of this
distribution (typically its mean and variance). In the same manner, one
might know the nature of the measurement noise, but without a precise
knowledge of its amplitude. 
In this subsection we remark that learning of the parameters can
be rather straightforwardly included in the present algorithmic
approach. However, its detail implementation and analysis, as well as the analysis of the deterioration of performance when the
prior distributions are not known, is left for future work. 

In order to learn the hyper-parameters, a common technique in
statistical inference is the expectation maximization
\cite{Dempster}, where one updates iteratively the hyper-parameters in
order to maximize the posterior likelihood, i.e. the normalization
${\cal Z}(Y)$ in
the posterior probability measure. In the context of approximate
message passing the expectation maximization has been derived and
implemented e.g. in \cite{KrzakalaMezard12,VilaSchniter11}. It turns
out that the expectation maximization update of hyper-parameters is
analogous to the update where one imposes the Nishimori
identities as was done for compressed sensing in
\cite{KrzakalaMezard12}. 
In other words one updates the hyper-parameters in such a
way that the Nishimori identities are satisfied after every 
iteration. For instance, the mean and variance of the distribution $P_X$
is updated to correspond to the empirical mean and variance of the estimators
of elements $X_{il}$. The variance of the noise in a Gaussian output
channel is updated to be the same as the squared difference between the
observed matrix $Y$ and the estimator of the product $FX$.
Generically, for mismatched priors the Nishimori identities are not
satisfied. At the same time the various hyper-parameters appear
explicitly in these identities. Therefore, one way of performing
learning of parameters is iteratively setting new values of the
hyper-parameters in order to satisfy the Nishimori identities. This can be straightforwardly implemented within the GAMP code.   In
compressed sensing following this strategy is equivalent to an
expectation maximization learning algorithm \cite{KrzakalaMezard12}.

It is interesting to note that, as the Bayes-optimal setting has
several nice properties in terms of simplicity of the analytical
approach and in terms of convergence of the message-passing algorithms, bringing the
iterations back on the Nishimori line by doing expectation
maximization improves quite generically the convergence of the
algorithm. In our opinion this is one of the properties observed in
\cite{parker2013bilinear,parker2014bilinear}.

\section{Approximate message passing for matrix factorization}
\label{sec:AMP}

\subsection{Approximate belief propagation for matrix factorization}
\label{sec:mes_pas}

Bayes inference amounts to the computation of marginals of the
posterior probability (\ref{eq:post}). In order to make it computationally tractable we have to resort to
approximations. In compressed sensing, the Bayesian approach combined
with a belief propagation (BP)
reconstruction algorithm leads to the so-called approximate message
passing (AMP) algorithm. It was first derived in \cite{DonohoMaleki09}
for the minimization of $\ell_1$, and subsequently generalized in
\cite{DonohoMaleki10,Rangan10b}. We shall now adapt
the same strategy to the case of matrix factorization. 

\begin{figure}[!ht]
\centering
\includegraphics[width=4in]{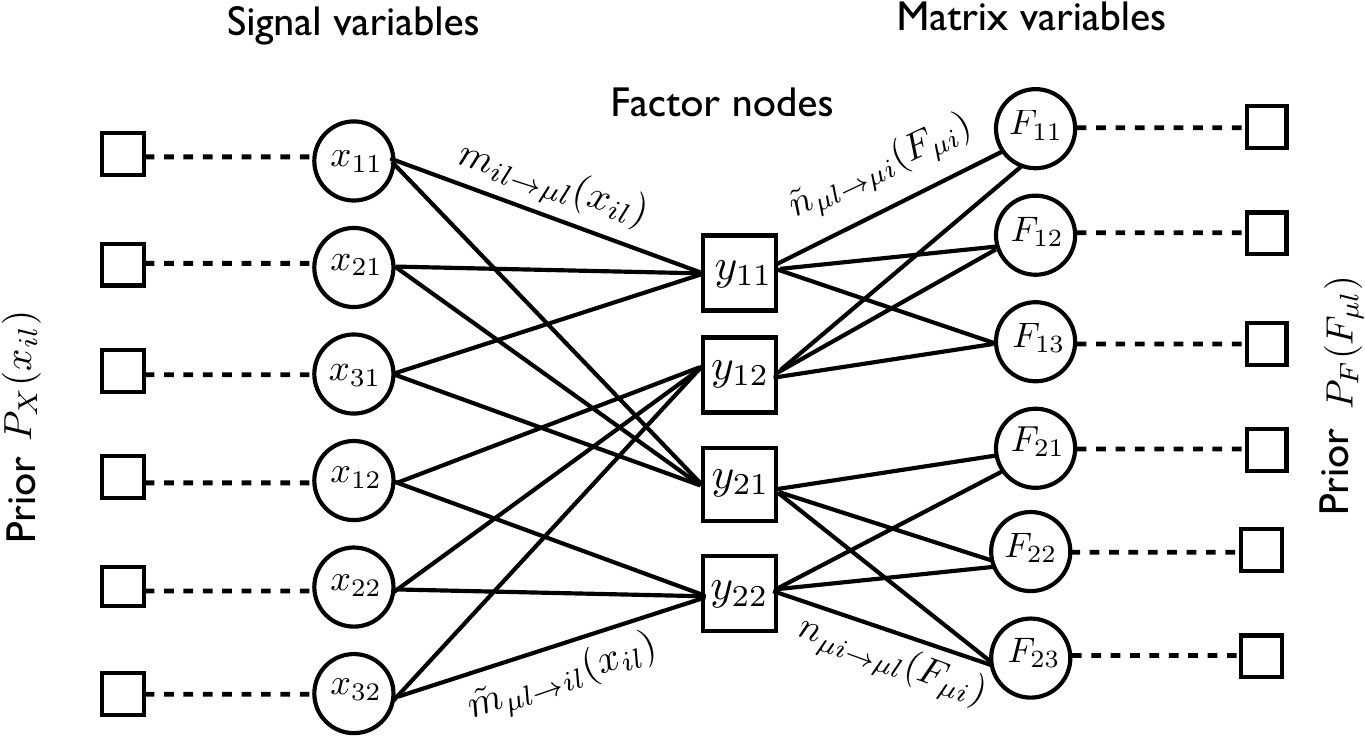}
\caption{Factor graph used for the belief propagation inference, here
  drawn using $N=3$, $P=2$ and $M=2$. The factor nodes are associated
  to the probability $P_{\rm out}(y_{\mu l\!}| \!\sum_{i} F_{\mu i}
  X_{il})$.}
\label{FactorGraph}
\end{figure}

The factor graph corresponding to the posterior
probability (\ref{eq:post}) is depicted in Fig.~\ref{FactorGraph}. The
canonical BP iterative equations \cite{KschischangFrey01} are written using
messages $m_{il\to\mu l}(X_{il})$, $n_{\mu i \to\mu l}(F_{\mu i})$ from variables to factors, and using messages ${\tilde
m}_{\mu l \to i l}(X_{il})$, ${\tilde n}_{\mu l \to \mu i}(F_{\mu i})$ from factors to
variables. On tree graphical models the messages are defined as
marginal probabilities of their arguments conditioned to the fact that
the variable/factor to which the message is incoming is not present in
the graph. The following BP
equations provide the exact values for these
conditional marginals on trees 
\bea m_{il\to\mu l} (t+1,X_{il})
&=& \frac{1}{{\cal Z}_{il\to\mu l}} P_X(X_{il}) \prod_{\nu (\neq \mu)}^M
\tilde{m}_{\nu l \to  i l} (t,X_{il}) \, , \label{message_m}\\
n_{\mu i \to\mu l} (t+1,F_{\mu i}) &=&\frac{1}{{\cal Z}_{\mu i \to\mu l}}  P_F(F_{\mu i}) \prod_{n (\neq
  l)}^P
\tilde{n}_{\mu n  \to  \mu i} (t,F_{\mu i }) \, ,\label{message_n}\\
\tilde{m}_{\mu l \to i l} (t,X_{il}) &=& \frac{1}{{\cal Z}_{\mu l \to i
    l}}    \int \prod_{j(\neq i)}^N
{\rm d}X_{jl} \prod_{k}^N dF_{\mu k}  \,    P_{\rm out}(y_{\mu l}| \sum_{k}^N
F_{\mu k} X_{kl}  )  \prod_k^N n_{\mu k \to \mu l} (t,F_{\mu k}) \prod_{j (\neq i)}^N m_{jl \to
  \mu l} (t,X_{jl}) \, , \label{message_tm}\\
\tilde{n}_{\mu l \to \mu i} (t,F_{\mu i}) &=&  \frac{1}{{\cal Z}_{\mu l
    \to \mu i
    }}  \int \prod_{j}^N {\rm d}X_{jl} \prod_{k
  (\neq i)}^N  dF_{\mu k} \, 
 P_{\rm out}(y_{\mu l}| \sum_{k}^N
F_{\mu k} X_{kl}  )  \prod_{k(\neq i)}^N n_{\mu k \to \mu l} (t,F_{\mu k}) \prod_{j}^N m_{jl \to
  \mu l} (t,X_{jl}) \, ,\label{message_tn}
\eea
where ${\cal Z}_{il\to\mu l}$, ${\cal Z}_{\mu i \to\mu l}$, ${\cal
  Z}_{\mu l \to i l}$, ${\cal Z}_{\mu l \to \mu i}$ are normalization
constants ensuring that all the messages are probability
distributions, $t\in {\mathbb N}$ is denoting the iteration
time-step, and the notation $\prod_{\nu (\neq \mu)}^M$ means a product
over all integer values of $\nu$ in $\{1,\dots,M\}$, except the value $\mu$.

Of course, the factor graph of matrix factorization, shown in
Fig.~\ref{FactorGraph}, is extremely far from a tree. The above
BP equations can, however, still be asymptotically exact if the
dependence between the incoming messages is
negligible in the leading order in $N$. This indeed happens in compressed
sensing (where the matrix $F$ is a known matrix, generated randomly with zero mean), as follows from the rigorously proven success of approximate
message passing \cite{DonohoMaleki09,BayatiMontanari10}. In the
present case of matrix factorization, we do not have any rigorous proof yet, but
based on our experience from studies of mean field spin glass systems \cite{MezardParisi87b,MezardMontanari09},
we conjecture that the fixed points of the above belief propagation
equations describe asymptotically exactly (in the same sense as for
compressed sensing) the performance of
Bayes-optimal inference. Hence the analysis of the fixed points of the
above equations leads to the understanding of information-theoretic
limitations for matrix factorization. The associated phase transitions
describe possible algorithmic barriers. This analysis is the main goal
and result of the present paper. 

The above BP iterative equations are written for probability
distributions over real values variables and the $2N-1$-uple integrals from the
r.h.s. are mathematically intractable in this form. We now define means and
variances of the variable-to-factor messages as 
\bea
  {\hat{x}}_{il \to \mu l}(t) &=& \int {\rm d}X_{il} \, m_{il \to \mu l}
  (t,X_{il})\, 
X_{il} \, ,\label{eq_def_fa} \\
c_{il\to \mu l}(t) + {\hat x}^2_{il\to \mu l}(t)&=& \int {\rm d}X_{il} \,
m_{il \to \mu l} (t,X_{il}) \, X^2_{il} \, ,\label{eq_def_fc}\\
{\hat{f}}_{\mu i \to \mu l}(t) &=& \sqrt{N} \int {\rm d}F_{\mu i} \, n_{\mu i\to \mu l}
(t,F_{\mu i})\,  F_{\mu i}\, , \label{eq_def_fr}\\
s_{\mu i\to \mu l}(t) + {\hat f}^2_{\mu i\to \mu l}(t)&=& N \int {\rm d}F_{il} \, n_{\mu i\to
  \mu l} (t,F_{\mu i}) \, F^2_{\mu i} \, . \label{eq_def_fs}
\eea 
Notice that the factors $\sqrt{N}$ in the definition of $r$, and $N$
in the definition of $s$, have been introduced in order to ensure that
all the messages $a,c,r,s$ are of order $O(1)$ in the thermodynamic
limit.

Using this scaling, we shall now show that the BP
equations can be simplified in the thermodynamic limit, and that they
can actually be written as a closed set of equations involving only
messages $a,c,r,s$.  Our general aim is to design an algorithm which in some region
of parameters will asymptotically match the performance of the
exact (computationally intractable) Bayes-optimal inference. Belief propagation provides
such an algorithm, but in order to make it computationally efficient, writing it in terms of the messages $a,c,r,s$ is crucial,
provided one is careful not to loose any terms in the asymptotic
analysis of the thermodynamic limit to 
leading order.

Let us define the Fourrier transform of the output function 
\be
         \hat{P}_{\rm out}(y,\xi) = \frac{1}{\sqrt{2\pi}}
         \int_{-\infty}^{\infty}   P_{\rm out}(y|z)  e^{-i \xi  z}  {\rm
           d} z \, ,
\ee 
to rewrite the update equation for message $\tilde{m}_{\mu l \to i l}
(t,X_{il})$ as 
\bea
     \tilde{m}_{\mu l \to i l} (t,X_{il}) &=& \frac{1}{\sqrt{2\pi}{\cal Z}_{\mu l \to i
    l}}  \int {\rm d}\xi  \,  \hat{P}_{\rm out}(y_{\mu l}, \xi )  \int
{\rm d} F_{\mu i}  \,  e^{i \xi 
F_{\mu i} X_{il}  }   n_{\mu i \to \mu l} (t,F_{\mu i})  \nonumber \\ && \prod_{j
(\neq i)}^N \left[   \int  {\rm d}X_{jl}  {\rm d}F_{\mu j} \,  e^{i \xi 
F_{\mu j} X_{jl}}  n_{\mu j \to \mu l} (t,F_{\mu j}) m_{jl \to
  \mu l} (t,X_{jl}) \right]\, .
\eea
In order to perform the integral in the square-bracket we recall that
the elements of matrix $F=O(1/\sqrt{N})$ and hence we can expand the
exponential to second order, use definitions
(\ref{eq_def_a}-\ref{eq_def_s}) and re-exponentiate the result without
loosing any leading order terms in $\tilde{m}_{\mu l \to i l}
(t,X_{il})$. The whole square bracket then becomes 
\be
      \exp\left\{ i\frac{\xi}{\sqrt{N}}  {\hat{f}}_{\mu j \to \mu l}(t) {\hat{x}}_{jl \to \mu l}(t) -
        \frac{\xi^2}{2 N} \left[  s_{\mu j \to \mu l}(t) c_{jl \to \mu l}(t)
        + s_{\mu j \to \mu l}(t) {\hat x}^2_{jl \to \mu l}(t)+ {\hat f}^2_{\mu j \to
          \mu l}(t) c_{jl \to \mu l}(t) 
  \right] \right\}  \, .
\ee
Next we perform the integral over variable $\xi$ which is simply a
Gaussian integral. This gives for the message 
\be
 \tilde{m}_{\mu l \to i l} (t,X_{il})  =  \frac{1}{{\cal Z}_{\mu l \to
     i l}} \int  {\rm d} F_{\mu i} \, n_{\mu i \to \mu l} (t,F_{\mu i})
 \frac{1}{\sqrt{2\pi V^t_{\mu i l}}} \int  {\rm d}z \, P_{\rm out}(y_{\mu
 l} | z)  e^{- \frac{(z- F_{\mu i} X_{il} - \omega^t_{\mu i l} )^2}{2
    V^t_{\mu i l}  }} \, , \label{eq:tildem_aux}
\ee
where we introduced auxiliary variables that are both of order $O(1)$
\bea
       V^t_{\mu i l} &\equiv&\frac{1}{N}\sum_{j (\neq i)}^N   \left[  s_{\mu j \to \mu l}(t) c_{jl \to \mu l}(t)
        + s_{\mu j \to \mu l}(t) {\hat x}^2_{jl \to \mu l}(t)+ {\hat f}^2_{\mu j \to
          \mu l}(t) c_{jl \to \mu l}(t)   \right] \, , \label{eq:def_Vmuil}\\
      \omega^t_{\mu i l}  &\equiv&\frac{1}{\sqrt{N}}\sum_{j (\neq i)}^N  {\hat{f}}_{\mu j \to \mu l}(t)
      {\hat{x}}_{jl \to \mu l}(t) \, . \label{eq:def_omuil}
\eea
The last integral to be performed in (\ref{eq:tildem_aux}) is the one
over the matrix element $F_{\mu i}$. Using again the fact that
$F_{\mu i} = O(1/\sqrt{N})$, we expand the exponential in which
$F_{\mu i} $ appears to second order and perform the integration to obtain 
\bea
     \tilde{m}_{\mu l \to i l} (t,X_{il})  &=&    \frac{1}{{\cal Z}_{\mu l \to
     i l}} 
 \frac{1}{\sqrt{2\pi V^t_{\mu i l}}} \int  {\rm d}z \, P_{\rm out}(y_{\mu
 l }| z)  e^{- \frac{(z- \omega^t_{\mu i l} )^2}{2
     V^t_{\mu i l}  }} \nonumber \\ && \left\{   1+ \frac{z- \omega^t_{\mu i l}}{\sqrt{N}V^t_{\mu i
       l} } {\hat{f}}_{\mu i \to \mu l}(t) X_{il}   - \frac{X_{il}^2}{2N} \left[ \frac{1}{V^t_{\mu i
       l}}  - \frac{ (z- \omega^t_{\mu i l })^2  }{(V^t_{\mu i
       l})^2}  \right]    \left[ {\hat f}^2_{\mu i \to \mu l}(t) + s_{\mu
 i \to \mu l}(t)  \right]  \right\}  \, .\label{eq:m_again}
\eea
Following the notation of \cite{Rangan10b} we now define the
output-function as  
\be
    g_{\rm out} (\omega, y , V) \equiv \frac{ \int {\rm d}z P_{\rm
        out}(y|z)\,  (z-\omega) \,   e^{-\frac{(z-\omega)^2}{2V}}  }{  V \int {\rm d}z P_{\rm
        out}(y|z) e^{-\frac{(z-\omega)^2}{2V}}  } \, . \label{eq:def_gout}
\ee
The following useful identity 
holds for the average of
$(z-\omega)^2/V^2$ in the above measure 
\be
  \frac{ \int {\rm d}z P_{\rm
        out}(y|z) \, (z-\omega)^2\,   e^{-\frac{(z-\omega)^2}{2V}}  }{ V^2  \int {\rm d}z P_{\rm
        out}(y|z) e^{-\frac{(z-\omega)^2}{2V}}  }=
        \frac{1}{V} + \partial_\omega  g_{\rm out} (\omega, y , V)  +
        g^2_{\rm out} (\omega, y , V) \, . \label{eq:id_gout}
\ee 
With definition (\ref{eq:def_gout}), and re-exponentiating the $X_{il}$-dependent
terms in (\ref{eq:m_again}) while keeping all the leading order terms, we obtain finally
that $\tilde{m}_{\mu l \to i l} (t,X_{il})$ is a Gaussian probability
distribution 
\be
        \tilde{m}_{\mu l \to i l} (t,X_{il}) = \sqrt{ \frac{  A^t_{\mu
              l \to i l }   }{2\pi N} }  e^{ -\frac {X_{il}^2}{2N}
  A^t_{\mu l \to i l} +  B^t_{\mu l \to i l} \frac{X_{il}}{\sqrt{N}} - \frac{(B^t_{\mu l \to
      i l } )^2}{2  A^t_{\mu l \to i l } } } \label{eq_def_AB}\\
\ee
with
\bea
           B^t_{\mu l \to i l } &=&  g_{\rm out} (\omega^t_{\mu i l}, y_{\mu
             l} , V^t_{\mu i l}) \,  {\hat{f}}_{\mu i \to \mu l}(t) \, ,\label{eq:B} \\
           A^t_{\mu l \to i l } &=&  - \partial_{\omega}  g_{\rm
             out} (\omega^t_{\mu i l}, y_{\mu l} , V^t_{\mu i l}) \,
           \left[  {\hat f}^2_{\mu i \to \mu l}(t) + s_{\mu i \to \mu l}(t)
           \right] - g^2_{\rm out} (\omega^t_{\mu i l}, y_{\mu l} ,
           V^t_{\mu i  l})  \,  s_{\mu i \to \mu l}(t) \, . \label{eq:A}
\eea
In a completely analogous way we obtain that the message
$\tilde{n}_{\mu l \to \mu i} (t,F_{\mu i})$ is also a Gaussian
distribution 
\be
\tilde{n}_{\mu l \to \mu i} (t,F_{\mu i}) =\sqrt{ \frac{  S^t_{\mu l
      \to \mu i  }   }{2\pi} }  e^{ -\frac {F_{\mu i}^2}2
  S^t_{\mu l \to \mu i} +  R^t_{\mu l \to \mu i} F_{\mu i} - \frac{(R^t_{\mu l \to \mu i})^2}{2  S^t_{\mu l \to \mu i}} }\label{eq_def_RS}
\ee
with 
\bea
        R^t_{\mu l \to \mu i} &=&g_{\rm out} (\omega^t_{\mu i l}, y_{\mu
             l} , V^t_{\mu i l}) \,   {\hat{x}}_{il \to \mu l}(t) \, ,\label{eq:R}\\
       S^t_{\mu l \to \mu i} &=&  - \partial_{\omega}  g_{\rm
             out} (\omega^t_{\mu i l}, y_{\mu l} , V^t_{\mu i l}) \,  \left[c_{i l\to \mu l}(t) +
      {\hat x}^2_{i l \to \mu l}(t)\right] -  g^2_{\rm out} (\omega^t_{\mu i l}, y_{\mu l} ,
           V^t_{\mu i  l})  \, c_{i l \to \mu l}(t) \, .\label{eq:S}
\eea

At this point we follow closely the derivation of AMP from
\cite{KrzakalaMezard12} and define the probability distributions
\bea
         {\cal M}_X(\Sigma,T,X) =  \frac{1}{\hat Z_X(\Sigma,T)}
         P_X(X) \frac{1}{\sqrt{2\pi \Sigma}}
         e^{-\frac{(X-T)^2}{2\Sigma}} \, , \label{eq:Mx}\\
  {\cal M}_F(Z,W,F) =  \frac{1}{\hat Z_F(Z,W)} P_F(F)
  \frac{1}{\sqrt{2\pi Z}} e^{-\frac{(\sqrt{N} F-W)^2}{2Z}} \, , \label{eq:MF}
\eea
where $\hat Z_X(\Sigma,T)$, and  $\hat Z_F(Z,W)$ are normalizations.
We define the average and
variance of ${\cal M}_X$ and ${\cal M}_F$ as
\bea
    f_X(\Sigma,T) &\equiv& \int {\rm d}X\,  X \,  {\cal
      M}_X(\Sigma,T,X)\, , \quad \quad
    f_c(\Sigma,T) \equiv \int {\rm d}X\,  X^2 \,  {\cal M}_X(\Sigma,T,X) -
    f^2_a(\Sigma,T) \, , \label{f_ac_gen} \\
   f_F(Z,W) &\equiv& \sqrt{N} \int {\rm d}F\,  F \,  {\cal
     M}_F(Z,W, F)\, ,\quad \quad 
    f_s(Z,W) \equiv  N \int {\rm d}F\,  F^2 \,  {\cal
      M}_F(Z,W,F) -
    f^2_r(Z,W) \, .  \label{f_rs_gen}
\eea
These are the input auxiliary function of \cite{Rangan10b}.
It is instrumental to notice that
\be
      f_X (\Sigma,T) = T  + \Sigma \frac{{\rm d}}{{\rm d} T}
      \log{\hat Z_X(\Sigma,T) }  \, , \quad \quad 
      f_c (\Sigma,T) = \Sigma  \frac{{\rm d}}{{\rm d} T}   f_X
      (\Sigma,T) \, . \label{eq:deriv_fa}
\ee
and analogously for $f_F$ and $f_s$
\be
     f_F (Z,W) = W  + Z \frac{{\rm d}}{{\rm d} W}
      \log{\hat Z_F(Z,W) }  \, , \quad \quad 
      f_s (Z,W) = Z  \frac{{\rm d}}{{\rm d} W}   f_F
      (Z,W) \, . \label{eq:deriv_fr}
\ee
With these definition we obtain from (\ref{message_m},\ref{message_n}) using
(\ref{eq_def_AB},\ref{eq_def_RS}) that 
\bea
{\hat{x}}_{i l\to \mu l}(t+1)&=&f_X\left(\frac{N}{\sum_{\nu(\neq \mu)}A^t_{\nu l \to
    i l }},\frac{\sqrt{N}\sum_{\nu(\neq \mu)}B^t_{\nu l \to i l}}{\sum_{\nu(\neq \mu)}A^t_{\nu l \to
    i l }} \right)\, , 
\label{BP_a_closed}  \\
c_{i l\to \mu l}(t+1)&=&f_c\left(\frac{N}{\sum_{\nu (\neq \mu)}A^t_{\nu l \to
    i l}},\frac{\sqrt{N}\sum_{\nu(\neq \mu)}B^t_{\nu l \to i l}}{\sum_{\nu(\neq \mu)}A^t_{\nu l \to
    i l }}\right) \, , 
\label{BP_v_closed} \\ 
{\hat{f}}_{\mu i\to \mu l}(t+1)&=& f_F\left(\frac{N}{(\sum_{n(\neq l)}S^t_{\mu n \to
   \mu i}}, \frac{\sqrt{N}\sum_{n(\neq l)}R^t_{\mu n \to
   \mu i}}{\sum_{n(\neq l)}S^t_{\mu n \to
   \mu i}}\right)\, , \label{BP_r_closed} \\ 
s_{\mu i\to \mu l}(t+1)&=&   f_s\left(\frac{N}{\sum_{n(\neq l)}S^t_{\mu n \to
   \mu i}}, \frac{\sqrt{N}\sum_{n(\neq l)}R^t_{\mu n \to
   \mu i}}{\sum_{n(\neq l)}S^t_{\mu n \to
   \mu i}}\right) \, ,   \label{BP_s_closed}
\eea
It is clear from the above expressions that all the messages $a,c,r,s$ and
$A,C,R,S$ scale as $O(1)$ in the thermodynamic limit. For instance, as $A$
are positive, the quantity $\sum_{\nu(\neq \mu)}A^t_{\nu l \to
    i l }$ is $O(1)$. 
On the other hand, the message ${\hat{f}}_{\mu i \to \mu
    l}(t) $ is an estimate of $ \sqrt{N} F_{\mu i}$; this estimate is
  $O(1)$, but a sum like $(1/\sqrt{N}) \sum_{\nu(\neq \mu)} {\hat{f}}_{\nu  i \to \nu
    l}(t) $ is an estimate of $ \sum_\nu F_{\nu i}$. As $P_F$ has 
  mean variance of order $O(1/N)$, this sum is actually of
  $O(1)$. The same argument suggests that $(1/\sqrt{N})\sum_{\nu(\neq
    \mu)}B^t_{\nu l \to i l}$ is $O(1)$.
Recalling eqs. (\ref{eq:B},\ref{eq:A}) and (\ref{eq:R},\ref{eq:S}), we
have derived that in the thermodynamic limit  the general belief propagation
equations simplify into a closed set of equations in the messages
which are the means and variances $a,r,c,s$ defined in
(\ref{eq_def_fa}-\ref{eq_def_fs}). To iterate this message passing
algorithm we initialize as
\bea
{\hat{x}}_{il\to \mu l}(0) &=& \int {\rm d}X \, X P^{il}_X(X) \, , \\
c_{il\to \mu l}(0) &=& \int {\rm d}X \, X^2 P^{il}_X(X) - {\hat x}^2_{il\to \mu l}(0)  \, , \\
{\hat{f}}_{\mu i\to \mu l}(0) &=& \sqrt{N} \int {\rm d}F \, F P^{\mu l}_F(F) \, , \\
s_{\mu i\to \mu l}(0) &=& N \int {\rm d}F \, F^2 P^{\mu l}_F(F) -
{\hat f}^2_{\mu i\to \mu l}(0)\, , 
\eea
then we compute $V^t_{\mu i l}$ and $\omega^t_{\mu il}$ from
(\ref{eq:def_Vmuil}-\ref{eq:def_omuil}), then we compute $B^t$, $A^t$,
$R^t$ and $S^t$ according to (\ref{eq:B}-\ref{eq:A}) and (\ref{eq:R}-\ref{eq:S}) using
definition of $g_{\rm out}$ (\ref{eq:def_gout}). Finally we update the messages
according to (\ref{BP_a_closed}-\ref{BP_s_closed}) and iterate. Notice, however, that we work with
$O(N^3)$ messages, each of them takes $N$ steps to update, and hence the computational complexity of this
algorithm is relatively high. In the next section we will write a
simplification that reduces this complexity.

From the fixed point of the belief propagation equations one can also
compute the approximated marginal probabilities of the posterior,
defined as
\bea 
m_{il} (t+1,X_{il})
&=& \frac{1}{{\cal Z}_{il}} P_X(X_{il}) \prod_{\nu}^M
\tilde{m}_{\nu l \to  i l} (t,X_{il}) \, , \label{full_m}\\
n_{\mu i} (t+1,F_{\mu i}) &=&\frac{1}{{\cal Z}_{\mu i}}  P_F(F_{\mu i}) \prod_{n}^P
\tilde{n}_{\mu n  \to  \mu i} (t,F_{\mu i }) \, ,\label{full_n}
\eea 
One again defines the mean and variance of these two messages,
${\hat{x}}_{il}(t+1)$, $c_{il}(t+1)$, ${\hat{f}}_{\mu i}(t+1)$ and $s_{\mu i}(t+1)$
analogously to (\ref{eq_def_a}-\ref{eq_def_s}):
\bea
  {\hat{x}}_{il} (t+1) &=& \int {\rm d}X_{il} \, m_{il}
  (t+1,X_{il})\, 
X_{il} \, ,\label{eq_def_a} \\
c_{il}(t+1) + {\hat x}^2_{il}(t+1)&=& \int {\rm d}X_{il} \,
m_{il} (t+1,X_{il}) \, X^2_{il}\, ,\label{eq_def_c}\\
{\hat{f}}_{\mu i }(t+1) &=& \sqrt{N}\int {\rm d}F_{\mu i} \, n_{\mu i}
(t+1,F_{\mu i})\,  F_{\mu i} \, ,\label{eq_def_r}\\
s_{\mu i}(t+1) + {\hat f}^2_{\mu i}(t+1)&=& N\int {\rm d}F_{il} \, n_{\mu i}
(t+1,F_{\mu i}) \, F^2_{\mu i} \, . \label{eq_def_s}
\eea 
Those quantities are
then expressed as 
\bea
{\hat{x}}_{i l}(t+1)&=&f_X\left(\frac{N}{\sum_{\nu}A^t_{\nu l \to
    i l }},\frac{\sqrt{N}\sum_{\nu}B^t_{\nu l \to i l}}{\sum_{\nu}A^t_{\nu l \to
    i l }} \right)\, , 
\label{BP_a_full} \\ 
c_{i l}(t+1)&=&f_c\left(\frac{N}{\sum_{\nu}A^t_{\nu l \to
    i l}},\frac{\sqrt{N}\sum_{\nu}B^t_{\nu l \to i l}}{\sum_{\nu}A^t_{\nu l \to
    i l }}\right) \, , 
\label{BP_c_full} \\ 
{\hat{f}}_{\mu i}(t+1)&=& f_F\left(\frac{N}{\sum_{n}S^t_{\mu n \to
   \mu i}}, \frac{\sqrt{N}\sum_{n}R^t_{\mu n \to
   \mu i}}{\sum_{n}S^t_{\mu n \to
   \mu i}}\right)\,  ,\\
 \label{BP_r_full}
s_{\mu i}(t+1)&=& f_s\left(\frac{N}{\sum_{n}S^t_{\mu n \to
   \mu i}}, \frac{\sqrt{N}\sum_{n}R^t_{\mu n \to
   \mu i}}{\sum_{n}S^t_{\mu n \to
   \mu i}}\right) \, .   \label{BP_s_full}
\eea

\subsection{GAMP for matrix factorization}
\label{Sec:GAMP}

The message-passing form the AMP algorithm for matrix factorization
derived in the previous section uses $2 N M P$ messages, one between each variable
component $(il)$ and $(\mu i)$ and each measurement $(\mu l)$, in each
iteration. In fact, exploiting again the simplifications which take
place in the thermodynamic limit,  always within
the assumption that the elements of the matrix $F$ scale as
$O(1/\sqrt{N})$, it is possible to rewrite  and close the BP
equations in terms of only $2N(P+M)$ messages.  In
statistical physics terms, the resulting equations correspond to the
Thouless-Anderson-Palmer equations (TAP) \cite{ThoulessAnderson77}
used in the study of
spin glasses. In
the thermodynamic limit, these are asymptotically equivalent to the BP equations.  Going from BP to TAP is, in the compressed
sensing literature, the step to go from the rBP
\cite{Rangan10} to the AMP \cite{DonohoMaleki09} algorithm. Let us now
show how to take this step for the present problem of matrix factorization.

In the thermodynamic limit, it is clear from (\ref{BP_a_closed}-\ref{BP_s_closed})
that the messages ${\hat{x}}_{il\to \mu l}$, $v_{il\to \mu l}$ and ${\hat{f}}_{\mu i
  \to \mu l}$, $s_{\mu i \to \mu l}$ are nearly independent of $\mu
l$. For instance in the equation giving ${\hat{x}}_{il\to \mu l}$, the only
dependence on $\mu$ is through the fact that the sum over $\nu$
avoids the value $\nu=\mu$. But this is one term in $M$, and therefore
one might expect that this term is negligible.
 However, one must be careful when these small terms are summed over
 and their sum might be of the leading order in $N$. Such terms are called
 in spin-glass theory the ``Onsager
reaction terms''. In the following we derive these Onsager terms. 

 Let us define the following variables all of order $O(1)$ on which we
will close the equations
\bea
      T^t_{il} &=& \frac{\sqrt{N}\sum_{\nu} B^t_{\nu l \to i l }}{\sum_{\nu}
        A^t_{\nu l \to i l }}  \, , \quad \quad 
     \Sigma^t_{il} =  \frac{N}{\sum_{\nu} A^t_{\nu l \to i l }}  \, , \label{eq:def_Sigma} \\
     W^t_{\mu i} &=& \frac{\sqrt{N}\sum_{n} R^t_{\mu n \to \mu i  }}{\sum_{n}
       S^t_{\mu n \to \mu i   }}\, , \quad \quad
     Z^t_{\mu i} = \frac{N}{\sum_{n} S^t_{\mu n \to \mu i  }}  \, , \label{eq:def_WZ} \\
    V^t_{\mu l} &=& \frac{1}{N}\sum_j [c_{jl\to \mu l}(t) s_{\mu j \to \mu l}(t)  +
    c_{jl\to \mu l}(t) {\hat f}^2_{\mu j \to \mu l}(t)   +  {\hat x}^2_{jl\to \mu l}(t)
    s_{\mu j \to \mu l}(t)  ] \, , \label{eq:def_V} \\
   \omega^t_{\mu l} &=& \frac{1}{\sqrt{N}}\sum_j {\hat{x}}_{j l\to \mu l}(t) {\hat{f}}_{\mu j \to \mu l}(t) \, . \label{eq:def_o} 
\eea
To keep track of all the Onsager terms that will influence the leading
order of the final equations we notice that
\bea
{\hat{x}}_{il\to \mu l}(t+1)&=& f_X\left(\frac{N}{\sum_{\nu}A^t_{\nu l\to i l}
    - A^t_{\mu l\to i l}},\frac{ \sqrt{N}\sum_{\nu}B^t_{\nu l\to i l} -\sqrt{N} B^t_{\mu l\to
      i l}}{ \sum_{\nu}A^t_{\nu l\to i l} -A^t_{\mu l\to i
      l} }\right)\, ,\nonumber \\
&=&
{\hat{x}}_{il}(t+1) -\frac{1}{\sqrt{N}} B^t_{\mu l\to i l}  \Sigma^t_{il} \frac{\partial
  f_X}{\partial T}\left(\Sigma^t_{il},T^t_{il}\right) +
O\left(1/N\right)\, , \nonumber \\
&=&{\hat{x}}_{il}(t+1) - \frac{1}{\sqrt{N}} g_{\rm out}(\omega^t_{\mu l}, y_{\mu l}, V^t_{\mu l})
{\hat{f}}_{\mu i}(t) c_{il}(t+1) +
O\left(1/N\right)\, . 
\eea
Similarly:
\bea
{\hat{f}}_{\mu i\to \mu l}(t+1)&=&
{\hat{f}}_{\mu i}(t+1) -\frac{1}{\sqrt{N}} g_{\rm out}(\omega^t_{\mu l}, y_{\mu l}, V^t_{\mu l})
{\hat{x}}_{i l}(t) s_{\mu i}(t+1) +
O\left(1/N\right)\,  , \\
     c_{il \to \mu l}(t+1) &=& c_{il}(t+1) +
     O\left(1/\sqrt{N}\right)\, , \quad \quad  s_{\mu i\to \mu
       l}(t+1)= s_{\mu i}(t+1)  + O\left(1/\sqrt{N} \right)\,
     , \\
   g_{\rm out}(\omega^t_{\mu i l},y_{\mu l}, V^t_{\mu i l}) &=&
   g_{\rm out}(\omega^t_{\mu l},y_{\mu l}, V^t_{\mu l}) - \frac{1}{\sqrt{N}} {\hat{f}}_{\mu i}(t)
   {\hat{x}}_{il}(t) \partial_\omega g_{\rm out}(\omega^t_{\mu l},y_{\mu l},
   V^t_{\mu l}) + O(1/N) \, , 
\eea
From these expansions we obtain the GAMP algorithm for matrix factorization 
\bea
  V^t_{\mu l} &=& \frac{1}{N}\sum_j [c_{jl}(t) s_{\mu j}(t)  +
    c_{jl}(t) {\hat f}^2_{\mu j}(t)   +  {\hat x}^2_{jl}(t)
    s_{\mu j}(t)  ] \, , \label{gamp_V}\\
   \omega^t_{\mu l} &=& \frac{1}{\sqrt{N}} \sum_j {\hat{x}}_{j l}(t) {\hat{f}}_{\mu j}(t) -  g_{\rm
     out}(\omega^{t-1}_{\mu l}, y_{\mu l}, V^{t-1}_{\mu l})  \frac{1}{N}\sum_j \left[
     {\hat{f}}_{\mu j}(t) {\hat{f}}_{\mu j}(t-1) c_{jl}(t)  + {\hat{x}}_{jl}(t)  {\hat{x}}_{jl}(t-1)
     s_{\mu j}(t)  \right]   \, ,  \label{gamp_omega}\\
   ( \Sigma^t_{il} )^{-1}&=& \frac{1}{N}\sum_{\mu} \left\{ - \partial_{\omega}  g_{\rm
             out} (\omega^t_{\mu l}, y_{\mu l} , V^t_{\mu l}) \,
           \left[  {\hat f}^2_{\mu i}(t) + s_{\mu i}(t)
           \right] - g^2_{\rm out} (\omega^t_{\mu l}, y_{\mu l} ,
           V^t_{\mu  l})  \,  s_{\mu i}(t) \right\}\, ,  \label{gamp_Sigma}\\
     T^t_{il} &=&  \Sigma^t_{il}  \Big\{ \frac{1}{\sqrt{N}}  \sum_\mu  g_{\rm out} (\omega^t_{\mu l}, y_{\mu
             l} , V^t_{\mu l}) \,  {\hat{f}}_{\mu i}(t) - {\hat{x}}_{il}(t) \frac{1}{N}\sum_\mu
           {\hat f}^2_{\mu i}(t)  \partial_\omega g_{\rm out}(\omega^t_{\mu l},y_{\mu l},
   V^t_{\mu l})  \nonumber \\ &&- {\hat{x}}_{il}(t-1) \frac{1}{N}\sum_\mu s_{\mu i}(t)  g_{\rm
             out} (\omega^t_{\mu l}, y_{\mu l} , V^t_{\mu l})  g_{\rm
             out} (\omega^{t-1}_{\mu l}, y_{\mu l} , V^{t-1}_{\mu l})
           \Big\} \, ,\label{gamp_T}\\ 
  ( Z^t_{\mu i} )^{-1}&=& \frac{1}{N}\sum_{l} \left\{ - \partial_{\omega}  g_{\rm
             out} (\omega^t_{\mu l}, y_{\mu l} , V^t_{\mu l}) \,
           \left[  {\hat x}^2_{i l}(t) + c_{i l}(t)
           \right] - g^2_{\rm out} (\omega^t_{\mu l}, y_{\mu l} ,
           V^t_{\mu  l})  \,  c_{il}(t) \right\}\, , \label{gamp_Z}\\
     W^t_{\mu i} &=& Z^t_{il}  \Big\{ \frac{1}{\sqrt{N}}  \sum_l  g_{\rm out} (\omega^t_{\mu l}, y_{\mu
             l} , V^t_{\mu l}) \,  {\hat{x}}_{il}(t) - {\hat{f}}_{\mu i}(t)
           \frac{1}{N} \sum_l
           {\hat x}^2_{il}(t)  \partial_\omega g_{\rm out}(\omega^t_{\mu l},y_{\mu l},
   V^t_{\mu l})  \nonumber \\ &&- {\hat{f}}_{\mu i}(t-1) \frac{1}{N}\sum_l c_{i l}(t)  g_{\rm
             out} (\omega^t_{\mu l}, y_{\mu l} , V^t_{\mu l})  g_{\rm
             out} (\omega^{t-1}_{\mu l}, y_{\mu l} , V^{t-1}_{\mu l})
           \Big\} \, ,\label{gamp_W}\\ 
    {\hat{x}}_{il}(t+1) &=& f_X(\Sigma^t_{il},T^t_{il})\, , \quad \quad
    c_{il}(t+1) = f_c(\Sigma^t_{il},T^t_{il})\, , \label{gamp_ac} \\
    {\hat{f}}_{\mu i}(t+1) &=& f_F(Z^t_{\mu i},W^t_{\mu i})\, , \quad \quad
    s_{\mu i}(t+1) = f_s(Z^t_{\mu i},W^t_{i\mu }) \, .\label{gamp_rs}
\eea

The initial condition for iterations are 
\bea
{\hat{x}}_{il}(t=0) &=& \int {\rm d}X \, X P^{il}_X(X) \, , \quad \quad
c_{il}(t=0) = \int {\rm d}X \, X^2 P^{il}_X(X) - {\hat x}^2_{il\to \mu l}(t=0)  \, , \\
{\hat{f}}_{\mu i}(t=0) &=& \sqrt{N} \int {\rm d}F \, F P^{\mu i}_F(F) \, ,
\quad \quad
s_{\mu i}(t=0) = N \int {\rm d}F \, F^2 P^{\mu i}_F(F) -
{\hat f}^2_{\mu i}(t=0)\, .
\eea
In order to compute $\omega^{t=0}$, $T^{t=0}$ and $W^{t=0}$ use the
above equations as if  ${\hat{f}}_{\mu i}(-1)=0$ and ${\hat{x}}_{il}(-1)=0$. 

The interpretation of the terms in the GAMP for matrix factorization
is the following: $\omega_{\mu l}^t$ is the mean of the current estimate of
$z_{\mu l} = \sum_i F_{\mu i} X_{il}$ and $V^t_{\mu l}$ is the
variance of that estimate; $T_{il}$ and $\Sigma_{il}$ is the mean and
variance of the current estimate of $X_{il}$ without taking into
account the prior information of $X_{il}$; the parameters ${\hat{x}}_{il}$ and
$c_{il}$ are then the mean and variance of the current estimate of
$X_{il}$ with the prior information taken into account. Analogously
for $W_{\mu i}$ and $Z_{\mu i}$ being the mean and variance of the
estimate for $F_{\mu i}$ before the prior is taken into account, and
${\hat{f}}_{\mu i}$ with $s_{\mu i}$ are the mean and variance once the prior
information was accounted for.

A reader familiar with the AMP and GAMP algorithm for compressed
sensing \cite{DonohoMaleki09,Rangan10b,KrzakalaMezard12} will recognize that the above equations indeed reduce to the
compressed sensing GAMP of \cite{Rangan10b} when one sets ${\hat{f}}_{\mu
  i}(t)=\sqrt{N} F_{\mu i}$ and $s_{\mu i}=0$. 

The above algorithm is closely related to the BiG-AMP of
\cite{parker2013bilinear}. 
There are however three differences between our algorithm and
BiG-AMP:
\begin{enumerate}
\item
We find a  $s_{\mu i}$-dependent term in the expression (\ref{gamp_Sigma}) for
$\Sigma_{il}$ which is not present in BiG-AMP.
\item
Similarly, we find a 
$c_{il}$-dependent term in the expression (\ref{gamp_Z}) for
$Z_{\mu i}$  which is not present in BiG-AMP. 
\item
The time indices are
slightly different.
\end{enumerate}

Considering the last point, the fact of having 
different time indices during the iterations does not influence the
fixed points, in which we are mainly interested. However, the use of correct
time indices is crucial for the assumptions leading to the density
evolution of this algorithm (that we derive in section \ref{Sec:SE})
to hold. 

As for the missing  terms in  the BiG-AMP expressions of
\cite{parker2013bilinear}  for $\Sigma_{il}$ and $Z_{\mu i}$, they have a
more serious effect as they can change
the fixed point. To the best of our understanding, these terms have
been neglected in \cite{parker2013bilinear}, while they should be kept. It seems to us that  some of the leading order
terms are missing in eqs. (15-16) from  \cite{parker2013bilinear}.

\subsection{Simplifications due to the Nishimori identities}
\label{Sec:Nish}

In the previous section we derived the GAMP algorithm for matrix
factorization, eqs. (\ref{gamp_V}-\ref{gamp_rs}). This algorithm can
in principle be used for any set of matrices $F$ and $X$. If iterated
in the form derived in Section \ref{Sec:GAMP} it often shows problems
of convergence. 
There are ways to slightly improve the convergence of the above algorithm in a
wide range of applications by a
number of empirical methods suggested in \cite{parker2014bilinear}. 

We will focus on the particular case when matrices $X$ and
$F$ were indeed generated from the separable probability distributions
$P_F(F)$ and $P_X(X)$ described in eqs. (\ref{PF}-\ref{PX}), and the output $y$ was
generated by the assumed model. 
\be
   P_X(X) = P_{X^0}(X)\, , \quad \quad   P_F(F) = P_{F^0}(F)\, ,
   \quad \quad    P^0_{\rm out}(y|z) = P_{\rm out}(y|z)\, . \label{eq:P=P0}
\ee
In this case the belief
propagation is a proxy for the optimal Bayes inference algorithms and
a number of properties described in section \ref{Sec:OB} hold. In
analogy with fundamental works on spin glasses \cite{NishimoriBook,Iba99} we called
these properties the Nishimori identities. The setting where
conditions (\ref{eq:P=P0}) hold will be called the {\it Bayes-optimal setting}. 

The Nishimori identities hold and the system is {\it on the Nishimori
  line} when one is using the correct priors
on $F$ and $X$ and the right output channel in the reconstruction
process, i.e. when conditions (\ref{eq:P=P0}) hold.
In the limit $N\to \infty$ and thanks to self-averaging we then have on the Nishimori line at every
iteration step $t$ 
\be
     \frac{1}{PM} \sum_{\mu, l}  \frac{ \int {\rm d}z P_{\rm
        out}(y_{\mu l}|z) (z-\omega_{\mu l}^t)^2  e^{-\frac{(z-\omega_{\mu l}^t)^2}{2V_{\mu l}^t}}  }{   \int {\rm d}z P_{\rm
        out}(y_{\mu l}|z) e^{-\frac{(z-\omega_{\mu l}^t)^2}{2V_{\mu
            l}^t}}  }  =  \frac{1}{M P } \sum_{\mu, l}   V_{\mu l}^t \, .   
\ee
The meaning of this identity is that the mean squared error of the
current estimate of $Z=FX$ computed from the current estimates of
variances $V_{\mu l}^t$ is equal to the mean squared difference between the true
$Z$ and its current estimate $\omega_{\mu l}^t$.   
Using the above expression and eq.~(\ref{eq:id_gout}) we obtain an identity 
\be
          -  \frac{1}{MP} \sum_{\mu,l}   \partial_\omega g_{\rm out}(\omega_{\mu l}^t, y_{\mu l},
          V^t_{\mu l}) =   \frac{1}{MP} \sum_{\mu,l}   g^2_{\rm out}(\omega_{\mu l}^t, y_{\mu l},
          V^t_{\mu l}) \, . \label{eq:gout_NL}
\ee
The above identity holds also if the sum is only over $\mu$ or only
over $l$. Finally using the conditional independence assumed in BP
between the incoming messages we get also 
\be
     -  \frac{1}{M} \sum_{\mu}   \partial_\omega g_{\rm out}(\omega_{\mu l}^t, y_{\mu l},
          V^t_{\mu l}) s_{\mu i}(t)=   \frac{1}{M} \sum_{\mu}   g^2_{\rm out}(\omega_{\mu l}^t, y_{\mu l},
          V^t_{\mu l}) s_{\mu i}(t) \, . 
\ee
Under this condition we can simplify considerably the expressions for
$\Sigma_{il}^t$ and $Z_{\mu i}^t$ and get
\bea
       ( \Sigma^t_{il} )^{-1}&=&  -\frac{1}{N}\sum_{\mu}  \partial_{\omega}  g_{\rm
             out} (\omega^t_{\mu l}, y_{\mu l} , V^t_{\mu l}) \,
             {\hat f}^2_{\mu i}(t) \, ,  \label{gamp_Sigma_NL}\\
 ( Z^t_{\mu i} )^{-1}&=& - \frac{1}{N}\sum_{l} \partial_{\omega}  g_{\rm
             out} (\omega^t_{\mu l}, y_{\mu l} , V^t_{\mu l}) \,
            {\hat x}^2_{i l}(t) \, . \label{gamp_Z_NL}
\eea 
Note that the r.h.s. of the two above equations is always strictly
positive, which is reassuring given these expressions play the role of
a variance of a probability distribution.  Note also that the BiG-AMP
algorithm of \cite{parker2013bilinear} uses expressions
(\ref{gamp_Sigma_NL},\ref{gamp_Z_NL}) instead of
(\ref{gamp_Sigma},\ref{gamp_Z}), however, without mentioning the
reason. 

\subsection{Simplification for matrices with random entries}
\label{Sec:full_TAP}

Relying on the definitions of order parameters (\ref{qq}-\ref{QD}) and
using part of the results on section \ref{Sec:SE} we can write a version of the GAMP for matrix
factorization that is in the leading order in $N$ equivalent to
(\ref{gamp_V}-\ref{gamp_rs}) for matrices $X$ and $F$ having iid
elements.

Let us define the analog of $\hat \chi^t$ and $\hat q^t$ (\ref{eq:def_chihat}-\ref{eq:def_qhat}) as
empirical means of the corresponding functions 
\bea
       \tilde \chi^t &\equiv&   - \frac{1}{MP} \sum_{\mu,l}  \partial_\omega
        g_{\rm out}(\omega_{\mu i l}^t,y_{\mu l},V^t_{\mu i l})\, ,  \label{tilde_chi} \\
     \tilde q^t &\equiv &\frac{1}{MP} \sum_{\mu,l}  
        g^2_{\rm out}(\omega_{\mu i l}^t,y_{\mu l},V^t_{\mu i l}) \, . \label{tilde_q} 
\eea

Anticipating the reasoning that we shall use later in section  \ref{Sec:SE}, we realize that
in the leading order quantities $V^t_{il}$, $\Sigma^t_{il}$ and $Z^t_{\mu
  i}$ do not depend on their indices $i,l,\mu$. We have 
\bea
       V^t &=& Q^t_F Q^t_X - q^t_F q^t_X \, ,\label{eq:V_fullTAP} \\
  (\Sigma^t)^{-1}  &=& \alpha Q_F^t \tilde \chi^t - \alpha (Q_F^t  - q_F^t ) \tilde
           q^t \, ,  \\
 (Z^t)^{-1}  &=& \pi  Q_X^t \tilde \chi^t - \pi  (Q_X^t  - q_X^t ) \tilde q^t \, .  
\eea 
where we define
\bea
 q^t_X &\equiv& \frac{1}{NP} \sum_{jl}   {\hat x}^2_{jl}(t) \, , \quad \quad
 q^t_F \equiv \frac{1}{NM} \sum_{\mu i}   {\hat f}^2_{\mu i}(t) \, , \label{qq-1} \\
 Q^t_X &\equiv& q^t_X + \frac{1}{NP} \sum_{jl}   c_{jl}(t)  \, ,\quad
 \quad  Q^t_F \equiv q^t_F + \frac{1}{NM} \sum_{\mu i}   s_{\mu i}(t) \, . \label{QD-1}
\eea

These three equations can hence replace (\ref{gamp_V}),
(\ref{gamp_Sigma}) and (\ref{gamp_Z}) in GAMP. Furthermore, if we focus
on the fixed point and hence disregard some of the time indices
eqs.~(\ref{gamp_omega}), (\ref{gamp_T}) and (\ref{gamp_W}) can be simplified as 
\bea
 \omega_{\mu l}^t &=& \frac{1}{\sqrt{N}}\sum_j {\hat{x}}_{jl}(t) {\hat{f}}_{\mu j}(t) - g_{\rm
   out}(\omega_{\mu l}^{t-1}, y_{\mu l}, V^{t-1}) ( Q_X^t q_F^t +
 q_F^t Q_X^t - 2 q_F^t q_X^t ) \, ,  \label{eq:omega_fullTAP}\\
   T_{il}^t &=& \Sigma^t \left\{  \frac{1}{\sqrt{N}}\sum_\mu  g_{\rm
   out}(\omega_{\mu l}^{t}, y_{\mu l}, V^{t}) {\hat{f}}_{\mu i}(t)  + \alpha {\hat{x}}_{il}(t)
 \tilde \chi^t  \, q_F^t - \alpha {\hat{x}}_{il}(t)  (Q_F^t-q_F^t ) \tilde q^t
\right\}  \, , \label {eq:T_} \\ 
 W_{\mu i}^t &=& Z^t \left\{  \frac{1}{\sqrt{N}}\sum_l  g_{\rm
   out}(\omega_{\mu l}^{t}, y_{\mu l}, V^{t}) {\hat{x}}_{ i l}(t)  + \pi
   {\hat{f}}_{\mu i}(t)
 \tilde \chi^t  \, q_X^t -\pi    {\hat{f}}_{\mu i}(t)  (Q_X^t-q_X^t ) \tilde q^t
\right\} \, . \label {eq:W_} 
\eea

Within the  Bayes-optimal setting of (\ref{eq:P=P0}), we can use the
Nishimori identity (\ref{eq:gout_NL}) to show
that $\tilde \chi^t = \tilde q^t$. Consequently we can use only one of
those parameters computed either from (\ref{tilde_chi}) or from
(\ref{tilde_q}). The equations then further simplify to strictly
positive expressions for the variance-parameters \be (\Sigma^t)^{-1} =
\alpha q_F^t \tilde q^t \, , \quad \quad (Z^t)^{-1} = \pi q_X^t \tilde
q^t \, .  \label{eq:SZ_NL} \ee The set of
eqs.~(\ref{gamp_ac}-\ref{gamp_rs}), (\ref{eq:V_fullTAP}),
(\ref{eq:omega_fullTAP}-\ref{eq:SZ_NL}) was presented for the simple
output channel with white noise in \cite{krzakala2013phase}.  We
detail this procedure for the generic output in
Alg. \ref{alg:grbmamp}.  \input{algo_tex}

We want to stress here that all these simplifications take place
for any output channel $P_{\rm out}(y|z)$.
In contrast with the ``uniform variance" approximation of
\cite{parker2013bilinear} the above result does not mean that the variances $c_{il}$ and
$s_{\mu i}$ are independent in the leading order on their indices. On
the contrary, these
variances depend on their indices even in the simplest case of GAMP
when the matrix $F_{\mu i}$ is known, i.e. for the compressed sensing problem. 

\section{The Bethe free entropy} 
\label{Sec:Bethe}



The fixed point of the belief propagation equations or its AMP version
can be used to estimate the posterior likelihood, i.e. the
normalization ${\cal Z}(Y)$ of the posterior probability
(\ref{eq:post}). The logarithm of this normalization is called the Bethe free entropy in statistical physics
\cite{YedidiaFreeman03}. Negative logarithm of the normalization is
called the free energy, in physics there is usually a temperature
associated to the free energy. Bethe free entropy is computed from the fixed point of the BP
equations (\ref{message_m}-\ref{message_tn}) as \cite{YedidiaFreeman03,MezardMontanari09}:

\be 
\Phi^{\rm Bethe} = \sum_{\mu l}
\log{{\cal Z}^{\mu l}} + \sum_{\mu i} \log{{\cal Z}^{\mu i}} + \sum_{i l} \log{{\cal Z}^{i
    l}} - \sum_{\mu i l} \log{{\cal Z}^{il,\mu l}} - \sum_{\mu i l}
\log{{\cal Z}^{\mu i,\mu l}} \, ,
\label{FBethe}
\ee
 where the five contributions are \bea {\cal Z}^{\mu
  i} &=& \int {\rm d}F_{\mu i} \, P_F(F_{\mu i})
\prod_{l=1}^P  \tilde n_{\mu l \to \mu i}(F_{\mu i}) \, , \\
{\cal Z}^{i l} &=& \int {\rm d}X_{ i l}\, P_X(X_{ i l})
\prod_{\mu=1}^M  \tilde m_{\mu l \to i l}(X_{ i l}) \, , \\
{\cal Z}^{\mu l} &=& \int \prod_{i=1}^N {\rm d}X_{ i l}\; \prod_{k=1}^N {\rm
  d}F_{\mu k} P_{\rm out}(y_{\mu l} | \sum_{k=1}^N F_{\mu k} X_{kl})
\prod_{i=1}^N m_{i l \to \mu l}(X_{ i l}) \prod_{k=1}^N
n_{\mu k \to \mu l}(X_{ i l})  \, , \\
{\cal Z}^{il,\mu l} &=& \int {\rm d}X_{ i l} \,
m_{i l \to \mu l}(X_{ i l}) \, \tilde m_{\mu l \to i l}(X_{ i l})  \, , \\
{\cal Z}^{\mu i,\mu l} &=& \int {\rm d}F_{\mu i} \, n_{\mu i \to \mu
  l}(F_{\mu i}) \, \tilde n_{\mu l \to \mu i }(F_{\mu i}) \, .  \eea
The derivatives of this expression for $\Phi^{\rm Bethe}$ with respect
to the messages give back the
full BP equations of (\ref{message_m}-\ref{message_tn}).
In this general form, the computation of $ \Phi^{\rm Bethe}$ for the
present problem is not of practical interest, and it is thus very useful to carry
out the same steps that we did in  Section \ref{sec:mes_pas} in order
to obtain a more mathematically tractable form of $ \Phi^{\rm Bethe}$  that is
asymptotically equivalent to (\ref{FBethe}) in the thermodynamic
limit, using the set of AMP message passing equations
(\ref{eq:def_Vmuil}-\ref{eq:def_omuil}), (\ref{eq:B}-\ref{eq:A}),
(\ref{eq:R}-\ref{eq:S}), and
(\ref{BP_a_closed}-\ref{BP_s_closed}). The result is:
\be 
\Phi^{\rm Bethe}  = \sum_{\mu l} \log{{\cal Z}^{\mu l}} +\sum_{\mu i} \log{ {\cal X}^{\mu i}} +
\sum_{ i l} \log{ {\cal X}^{i l}} + \sum_{\mu i l } \log{ \frac{ {\cal X}^{\mu i \to
      \mu l} }{{\cal X}^{\mu i} } } + \sum_{\mu i l } \log{ \frac{ {\cal X}^{ i l
      \to \mu l}}{{\cal X}^{il}} } \, ,\label{eq:Bethe_mp} 
\ee 
with 
\bea 
{\cal Z}^{\mu
  l} &=& \int {\rm d} z \frac{e^{-\frac{(\omega_{\mu l} -z)^2}{2V_{\mu
        l}}}}{\sqrt{2\pi V_{\mu l}}} P_{\rm
  out}(y_{\mu l}|z) \, ,  \label{eq:Zmul}\\
{\cal X}^{\mu i} &=& \int {\rm d}F_{\mu i} P_F(F_{\mu i}) e^{- \frac{N F^2_{\mu
      i}}{2
    {\cal Z}_{\mu i}} +\sqrt{N} F_{\mu i}  \frac{W_{\mu i}}{{\cal Z}_{\mu i}}  }\, , \\
{\cal X}^{ i l} &=& \int {\rm d}X_{ i l} P_X(X_{i l}) e^{- \frac{X^2_{ il}
  }{2
    \Sigma_{i l}} +  X_{i l}  \frac{T_{ i l}}{\Sigma_{ i l}} } \, ,\\
{\cal X}^{\mu i \to \mu l} &= &\int {\rm d}F_{\mu i} P_F(F_{\mu i})
e^{-\frac{F^2_{\mu i}}{2}\sum_{n \neq l } S_{\mu n \to \mu i} +
  F_{\mu i} \sum_{n\neq l} R_{\mu n \to \mu i} }  \, , \\
{\cal X}^{ i l \to \mu l} &= &\int {\rm d}X_{i l} P_X(X_{ i l})
e^{-\frac{X^2_{ i l}}{2N}\sum_{\nu \neq \mu } A_{\nu l \to i l} +
  \frac{ X_{i l}}{\sqrt{N}} \sum_{\nu\neq \mu} B_{\nu l \to i l} } \,
.  \eea 
Finally we might want to express the free entropy using the
fixed point of the GAMP eqs.~(\ref{gamp_V}-\ref{gamp_rs}). In order to
do this we need to rewrite the last two terms in (\ref{eq:Bethe_mp}).
Using an expansion in $1/N$ and keeping the leading order terms we get
 \bea
\sum_{l } \log{ \frac{ {\cal X}^{\mu i \to \mu l} }{{\cal X}^{\mu i} } } &=& -
\frac{W_{\mu i}}{ {\cal Z}_{\mu i}} {\hat{f}}_{\mu i} + \frac{1}{2 {\cal Z}_{\mu i}} (s_{\mu
  i} + {\hat f}^2_{\mu i}) + \frac{1}{2N} s_{\mu i} \sum_{l=1}^P g^2_{\rm
  out}(\omega_{\mu l}, y_{\mu
  l}, V_{\mu l}) {\hat x}^2_{il} \, , \\
\sum_{ \mu } \log{ \frac{ {\cal X}^{ i l \to \mu l}}{{\cal X}^{il}} } &=& - \frac{T_{i
    l}}{ \Sigma_{ i l}} {\hat{x}}_{ i l } + \frac{1}{2 \Sigma_{i l}} (c_{i l}
+ {\hat x}^2_{i l}) + \frac{1}{2N} c_{i l} \sum_{\mu=1}^M g^2_{\rm
  out}(\omega_{\mu l}, y_{\mu l}, V_{\mu l}) {\hat f}^2_{\mu i} \, .  
\eea 
We
remind that the above expressions give the posterior likelihood given
a {\it fixed point} on the GAMP equations.

To write the final formula in a more easily
interpretable form we use the probability distributions ${\cal M}_F$
and ${\cal M}_X$ defined in (\ref{eq:Mx}-\ref{eq:MF}) with normalizations  
\bea
     \hat  {\cal X}^{\mu i} &=&   \hat Z_F({\cal Z}_{\mu i},W_{\mu i}) \sqrt{2\pi
       Z_{\mu i}} = \int {\rm d}F_{\mu i} P_F(F_{\mu i})  e^{-
       \frac{(\sqrt{N} F_{\mu   i} -W_{\mu i} )^2}{2 Z_{\mu i}} }\, , \\
 \hat  {\cal X}^{i l} &=&  \hat Z_X(\Sigma_{i l},T_{i l}) \sqrt{2\pi
       \Sigma_{i l}}  = \int {\rm d}X_{ i l} P_X(X_{i l})  e^{- \frac{( X_{
       il} - T_{il} )^2 }{2 \Sigma_{i l}} } \, .
\eea
Putting all pieces together we find:
\bea
\Phi^{\rm Bethe}
&=& \sum_{il} \left[ \log \hat{{\cal X}}^{il} \left(T_{il},\Sigma_{il} \right) +
\frac{c_{il}+({\hat{x}}_{il}-T_{il})^2}{2\Sigma_{il}} \right] 
+ \sum_{\mu i} \left[ \log
\hat{{\cal X}}^{\mu i} \left(   W_{\mu i},Z_{\mu    i}\right)+ \frac{s_{\mu i}+({\hat{f}}_{\mu i}-W_{\mu i})^2}{2Z_{\mu i}}\right]
\nonumber \\
&+&\sum_{\mu l}\left[ \log Z^{\mu l}\left(\omega_{\mu l},V_{\mu l}\right) + 
\frac 1{2N} \sum_{\mu i l} g^2_{\rm out} \left(\omega_{\mu l},V_{\mu
    l} \right) \((s_{\mu i} {\hat x}^2_{il} + {\hat f}^2_{\mu i} c_{il}\)) \right]
\, .\label{FREE_ENERGY_BETHE}
\eea
The above expression evaluated at the fixed point of the AMP algorithm
hence gives the Bethe approximation to the log-likelihood. It is
mainly use to decide which fixed point of AMP is better. Indeed, there
are cases where there exist more than one AMP fixed point and it is
the one with the largest Bethe entropy that corresponds asymptotically
to the optimal Bayesian inference. 

\subsection{Fixed-point generating Bethe free entropy} 
Since the free entropy has a meaning only at the fixed point, we can
transform it by using any of the fixed point identities verified by
the BP messages. A well known property of Bethe free entropy and
belief propagation \cite{YedidiaFreeman03} is that the BP fixed points
are stationary points of the Bethe free entropy. 
In this section we show that also for the AMP for matrix factorization
the Bethe free entropy can be written in a form that
allows to generate the fixed-point BP equations as a stationary point.
It can be achieved by writing the Bethe free entropy
eq.~(\ref{FREE_ENERGY_BETHE}) as
\bea 
&&\Phi^{\rm Bethe}_{\rm
  AMP}\left(\{T_{il}\},\{\Sigma_{il}\},\{W_{\mu i}\},\{Z_{\mu
    i}\},\{\omega_{\mu l}\}, \{{\hat{x}}_{il}\},\{c_{il}\},\{{\hat{f}}_{\mu i}\},\{s_{\mu
    i}\}\right) = \sum_{il} \left[\log \hat{{\cal X}}^{il} \left(T_{il},\Sigma_{il} \right) +
\frac{c_{il}+({\hat{x}}_{il}-T_{il})^2}{2\Sigma_{il}}\right]\nonumber\\
&&+ \sum_{\mu i} \left[\log \hat{{\cal X}}^{\mu i} \left( W_{\mu i},Z_{\mu
    i}\right)+ \frac{s_{\mu i}+({\hat{f}}_{\mu i}-W_{\mu i})^2}{2Z_{\mu i}}\right]
+\sum_{\mu l}\left[ \log Z^{\mu l}\left(\omega_{\mu l},V_{\mu l}\right) +
\frac 12 \sum_{\mu l} \frac{(\omega_{\mu l}-\sum_i {\hat{f}}_{\mu i}
  {\hat{x}}_{il}/\sqrt{N})^2}{V_{\mu l}-\sum_i s_{\mu
    i}c_{i l}/{N}} \right ] \label{FREE_ENERGY_BETHE_GENERATING}\\
&& \text{with~} \quad V_{\mu l} =  \frac{1}{N}\sum_j [c_{jl} s_{\mu j}  +
    c_{jl} {\hat f}^2_{\mu j}   +  {\hat x}^2_{jl}    s_{\mu j}  ] \, .\nonumber
\eea
In order to derive   (\ref{FREE_ENERGY_BETHE_GENERATING}) from
(\ref{FREE_ENERGY_BETHE}), we have substituted
$g^2_{\rm out}$ by its fixed point expression, and imposed
the values of the variance $V$. Under the present form, the Bethe free
entropy satisfies the following theorem:

\begin{theorem}
  \emph{(Bethe/AMP correspondance)} 
The fixed point of the AMP equations eqs. (\ref{gamp_V}-\ref{gamp_rs}) are the
  stationary points
  of the cost function $\Phi^{\rm Bethe}_{\rm AMP}$
  eq.~(\ref{FREE_ENERGY_BETHE_GENERATING}).
\end{theorem}

\begin{proof}
  This can be checked explicitly by setting to
  zero the derivatives of $\Phi^{\rm Bethe}_{\rm AMP}$. Indeed, the derivatives with respect to $T,\Sigma,W$ and $Z$
  yield \bea {\hat{x}}_{il} &=& T_{il} + \Sigma_{il} \frac
  {\partial}{\partial T} \log \hat{{\cal X}}^{il}
  = f_X(\Sigma_{il},T_{il})\, ,\label{astar} \\
  c_{il} &=& \Sigma^2_{il} \frac {\partial}{\partial \Sigma_{il}} \log
  \hat{{\cal X}}^{il} -({\hat{x}}_{il}-T_{il})^2
  = f_c(\Sigma_{il},T_{il}) \, ,\label{cstar} \\
  {\hat{f}}_{\mu i} &=& W_{\mu i} + Z_{\mu i} \frac {\partial}{\partial W}
  \log \hat{{\cal X}}^{\mu i}
  = f_F(Z_{\mu i},W_{\mu i}) \, ,\label{rstar}\\
  s_{\mu i} &=& Z^2_{\mu i} \frac {\partial}{\partial Z_{\mu i}} \log
  \hat{{\cal X}}^{\mu i} -({\hat{f}}_{\mu i}-W_{\mu i})^2
  = f_s(Z_{\mu i},W_{\mu i}) \, .\label{sstar}
  \eea 
%

Then, the stationarity with respect to $\omega$ can be expressed
easily by noting that $\frac{\partial}{\partial \omega_{\mu z}}\log
Z^{\mu l}=g_{\rm out}$ (a consequence of 
  Eq.~(\ref{eq:Zmul}):
\be
  g_{\rm out} (V_{\mu l})+ \frac{(\omega_{\mu l}-\sum_i {\hat{f}}_{\mu i}
    {\hat{x}}_{il} /\sqrt{N})}{V_{\mu
      l}-\sum_i s_{\mu i}c_{i l}/{N}} =0 \, ,\label{eq:om} \ee 
  which is nothing but the fixed point equation for $\omega$. 

It is convenient to compute the derivative with respect to $V$ (even though
  this quantity is eventually a function of $r$,$s$,$a$ and $c$) using
  $\frac{\partial}{\partial V_{\mu z}}\log Z^{\mu l}=\frac 12\((g_{\rm
    out}^2 + \partial_{\omega} g_{\rm out}\)) $ so that at the fixed
  point, when eq.~(\ref{eq:om}) is satisfied, we have
\be
\frac{\partial \Phi_{\rm AMP}^{\rm Bethe}}{\partial V_{\mu i}} = \frac 12
  \((g_{\rm out}^2 + \partial_{\omega} g_{\rm out}\)) - \frac 12
  \frac{(\omega_{\mu l}-\sum_i {\hat{f}}_{\mu i} {\hat{x}}_{il}/\sqrt{N})^2}{(V_{\mu l}-\sum_i
    s_{\mu i}c_{i l}/{N})^2} = \frac 12 \partial_{\omega} g_{\rm
    out}\, .
\ee
Using this equation, one can finally check explicitly that deriving with respect to
$a,c,r$ and $s$ yields the remaning AMP equations for $T,\Sigma,W$ and
$Z$. This concludes the proof.
\end{proof}

\subsection{The variational Bethe free entropy}
\label{Sec:Bethe_var}

We have shown that the fixed points of the approximate message passing equations are extrema of
$\Phi^{\rm Bethe}_{\rm AMP}$. However they are in general saddle points
of this function, and it is very useful to derive an alternative
``variational'' free
entropy, the {\it maxima} of which are the fixed points. In particular, this will allow us to find these
fixed points by alternative methods which do not rely on iterating the
equations as was done for compressed sensing in \cite{KrzakalaManoel14}.
This variational free entropy can also be used not only at the
maximum, but for each possible values of the parameters, as the
current estimate of the quality of reconstruction. Such a property has
been used to implement a so-called {\it adaptive damping} in
compressed sensing \cite{VilaSchniterUs} and it can hence be anticipated that
similar implementation trick will be useful for matrix factorization
as well. 

\subsubsection{Generic output channel}
In order to derive the variational Bethe free entropy, we impose the
fixed point conditions, and express the free
entropy only as a function of the parameters of our trial
distributions for the two matrices. Then, we simply have
\be
\Phi^{\rm Bethe}_{\rm var} \left(\{T_{il}\},\{\Sigma_{il}\},\{W_{\mu    i}\},\{Z_{\mu
    i}\} \right) =\Phi^{\rm Bethe}_{\rm AMP}\left(\{T_{il}\},\{\Sigma_{il}\},\{W_{\mu i}\},\{Z_{\mu
    i}\},\{\omega_{\mu l}^*\},\{ a^*_{il}\},\{c^*_{il}\},\{r^*_{\mu i}\},\{s^*_{\mu i}\}\right)
\ee
where $a^*,c^*,r^*,s^*$ are given in terms of the Eqs.
(\ref{astar}-\ref{sstar}) by: $a^*=f_X(\Sigma_{il},T_{il})$,
 $c^*=f_c(\Sigma_{il},T_{il})$, $r^*=f_F(Z_{\mu i},W_{\mu
  i})$, $s^*=f_s(Z_{\mu i},W_{\mu i})$, and
$\omega^*$ is the solution of (\ref{gamp_omega}).

In order to write this variational expression in a nicer form, let us
notice that the Kullback-Leibler divergences between ${\cal M}_X$, ${\cal M}_F$ (\ref{eq:Mx}-\ref{eq:MF}) and
the prior distribution are
\bea - D_{\rm KL}( {\cal M}_F || P_F ) &=& \log{ \hat
  {\cal X}^{\mu i} } + \frac{f_s(Z_{\mu i} ,W_{\mu i})+ (f_F(Z_{\mu i} ,W_{\mu i})- W_{\mu i})^2}{2  Z_{\mu i}} \, , \\
- D_{\rm KL}( {\cal M}_X|| P_X ) &=& \log{ \hat {\cal X}^{i l} } + \frac{f_c(\Sigma_{i l},T_{i l}) + (f_X(\Sigma_{i l},T_{i l})- T_{i l})^2}{2 \Sigma_{i l}} \, .  
\eea 
Let us define an additional distribution
\bea
   {\cal M}_{\rm out}(\omega_{\mu l},V_{\mu l},z) =  \frac{1}{{\cal Z}^{\mu l}}
         P_{\rm out}(y_{\mu l}|z) \frac{1}{\sqrt{2\pi V_{\mu l}}}
         e^{-\frac{(z-\omega_{\mu l})^2}{2V_{\mu l}}} \, , \label{eq:Mz}\\
\eea
where ${\cal Z}^{\mu l}$ is given by (\ref{eq:Zmul}). Then one has
\be
- D_{\rm KL}( {\cal M}_{\rm out}|| P_{\rm out} ) = \log {\cal Z}^{\mu l} +
\frac 12 \log 2\pi  V_{\mu l} + \frac 12\((1+ V_{\mu l}  \partial_{\omega} g_{\rm out} +
 V_{\mu l}  g_{\rm out} ^2\))\, .
\ee
Starting from (\ref{FREE_ENERGY_BETHE}), we find:
\bea 
 &&\Phi^{\rm Bethe}_{\rm var} \left(\{T_{il}\},\{\Sigma_{il}\},\{W_{\mu
    i}\},\{Z_{\mu i}\} \right) = - \sum_{\mu i} D_{\rm KL}( {\cal M}_F(
Z_{\mu i},W_{\mu i} ) || P_F ) - \sum_{il} D_{\rm KL}( {\cal M}_X(
\Sigma_{i l},T_{i l} ) || P_X ) \nonumber \\
&&+\sum_{\mu l} \left[\log Z^{\mu l}\left(\omega^*_{\mu l},V^*_{\mu
    l}\right) + \frac {g_{\rm out}^2}2 (V^*_{\mu l} - \sum_j s^*_{\mu j}
c^*_{jl} /N)\right] \, ,\nonumber\\
&&= - \sum_{\mu i} D_{\rm KL}( {\cal M}_F( Z_{\mu i},W_{\mu i} ) ||
P_F ) - \sum_{il} D_{\rm KL}( {\cal M}_X( \Sigma_{i l},T_{i l} ) ||
P_X )  \label{eq:Bethe}  \\ &&- \sum_{\mu l} \left[ D_{\rm KL}( {\cal M}_{\rm
  out}(\omega^*_{\mu l},V^*_{\mu l})|| P_{\rm out} )
-\frac 12\left( \log{2\pi V_{\mu l}^*} +1+V_{\mu l}^*\partial_{\omega} g_{\rm out} +
\frac{g_{\rm out}^2}{N} \sum_j s^*_{\mu j}
c^*_{jl}  \right)  \right]
, \nonumber 
\eea
 with $V^*$ and $\omega^*$ satisfying
eqs.~(\ref{eq:V_fullTAP}) and (\ref{eq:omega_fullTAP}).  Note that this
expression has the same form as the one used in
\cite{rangan2013fixed} for the simpler case of GAMP for compressed
sensing and for the generalized linear problem. Our expression thus
generalizes the formula of \cite{rangan2013fixed} to the bi-linear case.


\subsubsection{The AWGN output channel} 

In the case of the additive white Gaussian noise output channel
(\ref{eq:DL_Pout}) the function $g_{\rm out}$ takes the simple form:
\be
 g_{\rm out}(\omega_{\mu l},y_{\mu l},V_{\mu l}) =
\frac{y_{\mu l} - \omega_{\mu l}}{\Delta + V_{\mu l}}\ ,
\ee
 hence
$\partial_\omega g_{\rm out}$ does not depend on the variable
$\omega_{\mu l}$. The only explicit dependence on $\omega_{\mu l}$ in
the free entropy is through eq. (\ref{eq:Zmul}) which becomes for the
AWGN output channel 
\be
{\cal Z}^{\mu l}= \frac{1}{\sqrt{2\pi(\Delta + V_{\mu
      l})}} e^{-\frac{(y_{\mu l} - \omega_{\mu l})^2}{2(\Delta+ V_{\mu
      l})}} \, .\label{eq:zml_GN} \ee The free entropy is defined only at
the fixed point of the GAMP equations. Given a fixed point we can
express from (\ref{gamp_omega}) for the AWGN channel \be \frac{y_{\mu
    l} - \omega_{\mu l}}{\Delta + V_{\mu l}} = \frac{y_{\mu l} -
  \frac{1}{\sqrt{N}}\sum_{j} {\hat{x}}_{jl} {\hat{f}}_{\mu j} }{ \Delta + \frac{1}{N}
  \sum_j c_{jl} s_{\mu j} }\, . \ee
 We plug this last expression into
(\ref{eq:zml_GN}) to obtain 
\be
{\cal Z}^{\mu l}= \frac{1}{\sqrt{2\pi(\Delta
    + V_{\mu l})}} e^{-\frac{(y_{\mu l} - \frac{1}{\sqrt{N}}\sum_{j}
    {\hat{x}}_{jl} {\hat{f}}_{\mu j} )^2}{2(\Delta+ \frac{1}{N} \sum_j c_{jl} s_{\mu
      j} )^2} (\Delta + V_{\mu l})} \, .
\ee
 Simplifying the last two terms
of eq. (\ref{eq:Bethe}) we obtain for the AWGN channel:
\bea
 \Phi^{\rm Bethe}
_{\rm AWGN}(\{T_{il}\},\{\Sigma_{il}\},\{W_{\mu i}\},\{ Z_{\mu i}\}) &=& - \sum_{\mu
  i} D_{\rm KL}( {\cal M}_F( Z_{\mu i},W_{\mu i} ) || P_F ) -
\sum_{il} D_{\rm KL}( {\cal M}_X( \Sigma_{i l},T_{i l} ) || P_X )
\nonumber \\ && - \sum_{\mu l} \frac{(y_{\mu l} -
  \frac{1}{\sqrt{N}}\sum_{j} a^*_{jl} r^*_{\mu j} )^2}{2(\Delta+
  \frac{1}{N} \sum_j c^*_{jl} s^*_{\mu j}) } - \frac{1}{2} \sum_{\mu l}
\log{\left[2\pi (\Delta + V^*_{\mu l})\right]} \,
. \label{eq:Bethe_AWGN} 
\eea 
The first three terms of this free
entropy are clearly negative and the last term cannot be larger than
$-MP \log{(2\pi \Delta)}/2$. Hence the free entropy
(\ref{eq:Bethe_AWGN}) is bounded from above. This is consistent with its
interpretation as a variational expression. The stationary points of
(\ref{eq:Bethe_AWGN}) are the fixed points of the GAMP algorithm and
hence the fixed points corresponding to the maximum likelihood could
also be found by direct maximization of the expression
(\ref{eq:Bethe_AWGN}). This offers an interesting algorithmic
alternative to the iterative AMP algorithm that was explored for the
compressed sensing problem in \cite{KrzakalaManoel14}. 

Another use of the expression  (\ref{eq:Bethe_AWGN}) is that during the
iteration of the GAMP algorithm its value should be increasing, hence
we can adaptively choose the step-size of the iterations to ensure
this increase. Such an adaptive dumping was implemented in
\cite{parker2014bilinear} using a different form of the free entropy that
does not correspond to the Bethe free entropy but to the variational
mean field (VMF) free entropy which reads 
\bea
   \Phi^{\rm VMF}_{\rm AWGN}(\{T_{il}\},\{\Sigma_{il}\},\{W_{\mu i}\},\{ Z_{\mu i}\}) &=& - \sum_{\mu i} D_{\rm KL}( {\cal
    M}_F( Z_{\mu i},W_{\mu i} ) || P_F )  - \sum_{il} D_{\rm KL}(
  {\cal M}_X( \Sigma_{i l},T_{i l} ) || P_X ) \nonumber \\ && - \frac{1}{2\Delta} \sum_{\mu l}
 \left[ (y_{\mu l} - \frac{1}{\sqrt{N}}\sum_{i}  {\hat{f}}_{\mu i} {\hat{x}}_{il}
    )^2 + V_{\mu l} \right]-  \frac{MP}2 \log{ 2\pi \Delta} \, . \label{eq:MTF_AWGN}
\eea

It is easy to check that $\Phi^{\rm VMF}_{\rm
  AWGN}(T_{il},\Sigma_{il},W_{\mu i}, Z_{\mu i})<\Phi^{\rm Bethe}_{\rm
  AWGN}(T_{il},\Sigma_{il},W_{\mu i}, Z_{\mu i})$, which could be
expected since the Bethe expression, which is asymptotically exact,
should be a better approximation than the mean field approximation.

\section{Asymptotic analysis}
\label{Sec:AA}

\subsection{State evolution}
\label{Sec:SE}


In this section we derive the asymptotic ($N\to \infty$) evolution of
the GAMP iterations for matrix
factorization. This asymptotic analysis holds as long as all the elements
of the true matrix $F$  are iid random variables generated from a
distribution $P_{F^0}$, and all elements of the true matrix $X$ are iid random variables generated from a
distribution $P_{X^0}$. In general we
will not assume $P_{F^0}=P_{F}$ and $P_{X^0}=P_{X}$: this special case
of Bayes-optimal analysis will be treated in the next section. 

In the present section we will
also distinguish between the true output channel characterized by the
conditional probability distribution $P^0_{\rm out}(y_{\mu l}|z^0_{\mu l})$
and the output channel that is being used in the GAMP algorithm
$P_{\rm out}(y_{\mu l}|z_{\mu l})$. We remind that $z^0_{\mu
  l}=\sum_{i=1}^N F^0_{\mu i} X^0_{il}$, where $F^0_{\mu i}$ and $X^0_{il}$ are the elements of the
actual matrices that we do not know and aim to recover,  and $z_{\mu
  l}=\sum_{i=1}^N F_{\mu i} X_{il}$. Again the special case of
$P^0_{\rm out}=P_{\rm out}$ will be treated the next section. 

We will assume that at least one of the probability distributions
$P_{F^0}$ and $P_{X^0}$ (and also at least one of $P_{F}$ and $P_{X}$)
has zero mean, otherwise there would be additional
terms in this asymptotic analysis, as e.g. in \cite{CaltagironeKrzakala14}. 

Let us first recall the definition of the order parameters; all of them are finite, of order
$O(1)$, in the thermodynamic limit:
\bea
 m^t_X &\equiv& \frac{1}{NP} \sum_{jl}   {\hat{x}}_{jl}(t) X^0_{jl}\,
 ,  \label{mx}\quad \quad  m^t_F \equiv \frac{1}{M\sqrt{N}} \sum_{\mu
   i}   {\hat{f}}_{\mu i}(t) F^0_{\mu i} \, , \\
 q^t_X &\equiv& \frac{1}{NP} \sum_{jl}   {\hat x}^2_{jl}(t) \, , \quad \quad
 q^t_F \equiv \frac{1}{NM} \sum_{\mu i}   {\hat f}^2_{\mu i}(t) \, , \label{qq} \\
 Q^t_X &\equiv& q^t_X + \frac{1}{NP} \sum_{jl}   c_{jl}(t)  \, ,\quad
 \quad  Q^t_F \equiv q^t_F + \frac{1}{NM} \sum_{\mu i}   s_{\mu i}(t) \, . \label{QD}
\eea
Note also that the above sums over a pair of indices could also be
sums over only one index (and adjusted normalization) and the order
parameters would not change in the leading order: for instance, we
expect that in the thermodynamic limit, $\frac{1}{N} \sum_{j}
{\hat{x}}_{jl}(t) X^0_{jl}$ will go to the same limit as $m^t_X$ defined in (\ref{mx}).

First let us compute the average over realizations of $X^0$, $F^0$ and
$w^0$ of
the quantity $V_{\mu l}^t$ defined by eq.~(\ref{eq:def_V}). By the
assumptions of the belief propagation equations
(\ref{message_tm}-\ref{message_tn}), the terms in the product in (\ref{eq:def_V}) are statistically
independent and we can hence write for the average, to the leading
order 
\be
         V^t = Q^t_F Q^t_X - q^t_F q^t_X \, .
\ee
Further, we realize that the variance of this quantity (again over the
realizations of $X^0$, $F^0$ and $w^0$) is 
\be
      \mathbb{E}[ (  V_{\mu l}^t -  V^t )^2] =  \mathbb{E}\left[ \left\{  \frac{1}{N}\sum_i
      [c_{il \to \mu l}(t) -\frac{1}{N} \sum_k c_{kl}(t) ] s_{\mu i \to \mu l}()  + ...
      + ... \right\}^2\right] = O(1/N)  \, . 
\ee
In order to derive this result, we expand the square and obtain a double sum over $i$
and $j$. Because of the conditional independence between incoming
messages assumed in belief propagation, the terms with $i \neq j$
average exactly to zero. As for the
terms with $i=j$, they add up to a contribution of order $O(1/N)$. From this
we can conclude that, to leading order in the thermodynamic limit, the quantity $V_{\mu
  l}^t = V^t$ does not depend on its indices. 

Further we are interested in the average $\Sigma^t$ of the quantity $\Sigma_{il}^t$
over the realization of $F^0$, $X^0$, and $w^0$. Using the definition of $\Sigma_{il}^t$
eq.~(\ref{eq:def_Sigma}) and the expression for $A_{\mu l\to il}^t$ eq.~(\ref{eq:A}) we obtain 
\be
      (\Sigma_{il}^t)^{-1} = -\frac{1}{N}\sum_{\mu} \left\{ \partial_\omega
        g_{\rm out}(\omega_{\mu i l}^t,y_{\mu l},V^t_{\mu i l}) [
        {\hat{f}}_{\mu i\to \mu l}^2(t) +  s_{\mu i\to \mu l}(t)   ]  +
        g^2_{\rm out}(\omega_{\mu i l}^t,y_{\mu l},V^t_{\mu i l})
        s_{\mu i\to \mu l}(t)  \right\}\, .
\ee
We proceed analogously for $Z^{t}_{\mu i}$. 
Using again the conditional independence between incoming messages
assumed in BP equations we obtain 
\bea
           (\Sigma^t)^{-1}  &=& \alpha Q_F^t \hat \chi^t - \alpha (Q_F^t  - q_F^t ) \hat
           q^t \, ,  \label{eq:DE_S}\\
 (Z^t)^{-1}  &=& \pi  Q_X^t \hat \chi^t - \pi  (Q_X^t  - q_X^t ) \hat q^t \, ,  \label{eq:DE_Z}
\eea
where we introduced new parameters 
\bea
      \hat \chi^t &=& - \frac{1}{M} \mathbb{E}_{F^0,X^0,w^0}\left[ \sum_\mu  \partial_\omega
        g_{\rm out}(\omega_{\mu i l}^t,y_{\mu l},V^t_{\mu i l}) 
      \right] \, ,\label{eq:def_chihat}\\
     \hat q^t &=& \frac{1}{M} \mathbb{E}_{F^0,X^0,w^0}\left[  \sum_\mu  
        g^2_{\rm out}(\omega_{\mu i l}^t,y_{\mu l},V^t_{\mu i l})
      \right] \, .\label{eq:def_qhat}
\eea
We use  $y_{\mu l} = h(z^0_{\mu l},w^0)$, and we remind that, to leading order, $V^t_{\mu i l}=V^t$. 
The function $g_{\rm out}$ above hence depends on two correlated fluctuating
variables $\omega_{\mu i l}^t$ and $z^0_{\mu l}$, and on $w^0$. Both
the variables $\omega_{\mu i l}^t$ and $z^0_{\mu l}$
are sums over many independent (for $z$ this is by construction, for
$\omega$ by the BP assumptions) terms. Hence, according to the
central limit theorem, they are Gaussian random variables. Their mean is zero when at least one of
the distributions $P_X$ and $P_F$ (and one of the $P_{X^0}$ and
$P_{F^0}$) have zero mean (which we assume in this section). The
covariance matrix between the variables  $\omega_{\mu i l}^t$ and
$z^0_{\mu l}$ is 
\bea
        \mathbb{E}( \omega^2_{\mu i l} ) &=&\frac{1}{N} \sum_{j (\neq i)}
        \mathbb{E}(   {\hat f}^2_{\mu j \to \mu l}   {\hat x}^2_{j l  \to \mu l} ) + 
        \frac{1}{N} \sum_{ j(\neq i),k(\neq i,j)}    \mathbb{E}(  {\hat{f}}_{\mu j \to \mu l}
        {\hat{f}}_{\mu k \to \mu l}   {\hat{x}}_{j l  \to \mu l} {\hat{x}}_{k l  \to \mu l} )=
        q_F q_X \, ,\label{eq:qq} \\
 \mathbb{E}( \omega_{\mu i l} z^0_{\mu l} ) &=&\frac{1}{N} \sum_{j(\neq i)}
 \mathbb{E}(  F_{\mu j}^0 {\hat{f}}_{\mu j \to \mu l}  X^0_{jl}  {\hat{x}}_{jl \to \mu
 l} )  +   \frac{1}{N} \sum_{k, j(\neq i,k)} \mathbb{E}(  F_{\mu k}^0 {\hat{f}}_{\mu j \to \mu l}  X^0_{kl}  {\hat{x}}_{jl \to \mu
 l} ) = m_F m_X\, , \label{eq:mm}  
\eea 
where we again used the BP assumption of independence between the incoming messages, but
also between $F^0_{\mu i}$ and the message ${\hat{x}}_{il \to \mu l}$, and
between $X_{il}^0$ and ${\hat{f}}_{\mu i \to \mu l}$. As for the variance of
$z^0$, we denote it by:
\be
   \mathbb{E}[ (z^0_{\mu l})^2 ] = N  \, \mathbb{E}[ (F^0_{\mu
          i} )^2 ] \,  \mathbb{E}[ (X^0_{i l})^2 ] =  \langle (z^0)^2
        \rangle \, .\label{eq:zz} 
\ee

Altogether this gives for $\hat \chi$ and $\hat q$
\bea
     \hat \chi^t &=& -  \int {\rm d}w\,  P_0(w) \int {\rm d}p \,
     {\rm d}z \, {\cal N}[p,z; q^t_F q^t_X,\langle (z^0)^2 \rangle, m^t_F m^t_X
     ] \, \partial_p g_{\rm out}[p,h(z,w),Q_F^t Q_X^t - q_F^t q_X^t] \, , \label{eq:chihat}\\
 \hat q^t &=& \int {\rm d}w\,  P_0(w) \int {\rm d}p \,
     {\rm d}z \, {\cal N}[p,z; q^t_F q^t_X,\langle (z^0)^2 \rangle, m^t_F m^t_X
     ] \, g^2_{\rm out}[p,h(z,w),Q_F^t Q_X^t - q_F^t q_X^t] \, . \label{eq:qhat} 
\eea
where  ${\cal N}[p,z; \sigma^2_p, \sigma^2_z,  \mathbb{E}(p z) ]$ is a joint Gaussian distribution of variables $p$ and
$z$ with zero means, variances and correlation given in the argument. 
From the above analysis it also follows that in the leading order the
quantities $\Sigma^t_{il}$ and $Z_{\mu i}^t$ do not depend on their
indices $il$ and $\mu l$. 

We now study the asymptotic behavior of $T^t_{il}$ defined
by eq.~(\ref{eq:def_Sigma}) 
\bea
    T^t_{il} /\Sigma^t_{il} = \frac{1}{\sqrt{N}}\sum_{\mu } B^t_{\mu l \to \i l} =
     \frac{1}{\sqrt{N}}\sum_\mu {\hat{f}}_{\mu i \to \mu l}(t)  g_{\rm out}(\omega^t_{\mu i l},
    y_{\mu l}, V^t_{\mu i l}) \, ,
\eea 
where we used definition of $B^t_{\mu l \to \i l} $ in eq.~(\ref{eq:B}).
The message ${\hat{f}}_{\mu i \to \mu l}(t)$ are uncorrelated with all the
other incoming messages and also with all the $F_{\mu j}^0$ for $j\neq
i$. It is, however, correlated with $F_{\mu i}^0$ and the dependence
on $F_{\mu i}^0$ has to be hence treated separately. After an
expansion in the leading order we obtain
\bea
 T^t_{il} /\Sigma^t_{il} &=& 
    X^0_{il}  \frac{1}{\sqrt{N}} \sum_{\mu} {\hat{f}}_{\mu i \to \mu l}(t)  \, F^0_{\mu i}\,  \partial_z g_{\rm out}(\omega^t_{\mu i l},
    h( z_{\mu l} ,w_{\mu l}), V^t_{\mu i l}) +  \frac{1}{\sqrt{N}}\sum_\mu {\hat{f}}_{\mu i \to \mu l}(t) \,  g_{\rm out}(\omega^t_{\mu i l},
    h( \sum_{j(\neq i)} F^0_{\mu j} X^0_{jl}  ,w), V^t_{\mu i l})
    \nonumber \\
   &\approx & \alpha X^0_{il}  m^t_F \hat m^t    +  {\cal N}(0,1)
   \sqrt{ \alpha q^t_F \hat q^t} \, ,\label{eq:DE_TS}
\eea
where in the first term we defined the new parameter $\hat m$
as 
\be
     \hat m^t =  \frac{1}{M} \mathbb{E}_{F^0,X^0,w^0}\left[ \sum_\mu  \partial_z
        g_{\rm out}(\omega_{\mu i l}^t,h( z_{\mu l} ,w_{\mu
          l}),V^t_{\mu i l})
      \right]\, . \label{eq:def_mhat}
\ee
Using the same kind of analysis as we did for $\hat q$ and $\hat
\chi$, we find that
\be
 \hat m^t =  \int {\rm d}w\,  P_0(w) \int {\rm d}p \,
     {\rm d}z \, {\cal N}[p,z; q^t_F q^t_X,\langle (z^0)^2 \rangle, m^t_F m^t_X
     ] \, \partial_z g_{\rm out}[p,h(z,w),Q_F^t Q_X^t - q_F^t q_X^t] \, , \label{eq:mhat}
\ee
 
The second term of (\ref{eq:DE_TS}) when averaged over realization of
$F^0$, $X^0$ and $w^0$ behaves as a Gaussian random variable. In
(\ref{eq:DE_TS}) we moreover assumed
that  
\be
   \mathbb{E}_{F^0,X^0,w^0} \left[ \frac{1}{M} \sum_{\mu}  g_{\rm out}(\omega^t_{\mu i l},
    h( \sum_{j(\neq i)} F^0_{\mu j} X^0_{jl}  ,w), V^t_{\mu i l})  \right] =  0 \, ,   \label{eq:zero_m}
\ee
which is true in all the special cases analysed in this paper, and is also
true in general under the Bayes-optimal inference as detailed in the
next section. 
If the zero-mean assumption  (\ref{eq:zero_m}) did not hold the density evolution
equations would contain additional terms (similarly as if both $F$ and
$X$ had non-zero means), see e.g. the state evolution in compressed
sensing for non-zero mean matrices \cite{CaltagironeKrzakala14}. Under the zero-mean assumption, the variance
of the Gaussian variable is $\alpha q^t_F \hat q^t$ with $\hat q^t$
given by (\ref{eq:qhat}). 

Analogously we have 
\be
   W^t_{\mu i} /Z^t_{\mu i}  \approx \pi \sqrt{N} F^0_{\mu i}  m^t_X \hat m^t
   +  {\cal N}(0,1)
   \sqrt{ \pi q^t_X \hat q^t} \ . \label{eq:DE_WZ}
\ee
With the use of (\ref{eq:DE_S}-\ref{eq:DE_S}),
(\ref{eq:DE_TS}), (\ref{eq:DE_TS}), and the expressions for
messages (\ref{gamp_ac}-\ref{gamp_rs}) we obtain 
\bea
         Q^{t+1}_X - q_X^{t+1} &=& \int {\rm d}X^0 P_{X^0}(X^0)  \int
         {\cal D}\xi \,  \,  f_c\left[  \frac{1}{\alpha Q_F^t \hat \chi^t -
             \alpha(Q_F^t -q_F^t) \hat q^t}  , \frac{\alpha m_F^t
             \hat m^t \,  X^0+ \xi \sqrt{\alpha q_F^t \hat q^t}}{\alpha Q_F^t \hat \chi^t -
             \alpha(Q_F^t -q_F^t) \hat q^t}   \right] \, ,\label{eq:DE_Qx} \\
         q_X^{t+1} &=& \int {\rm d}X^0 P_{X^0}(X^0)  \int
         {\cal D}\xi \, \,  f_X^2\left[  \frac{1}{\alpha Q_F^t \hat \chi^t -
             \alpha(Q_F^t -q_F^t) \hat q^t}  , \frac{\alpha m_F^t
             \hat m^t\,  X^0 + \xi \sqrt{\alpha q_F^t \hat q^t}}{\alpha Q_F^t \hat \chi^t -
             \alpha(Q_F^t -q_F^t) \hat q^t}   \right] \, ,\label{eq:DE_qx} \\
          m_X^{t+1} &=& \int {\rm d}X^0 P_{X^0}(X^0)  \int
         {\cal D}\xi \, \,   X^0 f_X\left[  \frac{1}{\alpha Q_F^t \hat \chi^t -
             \alpha(Q_F^t -q_F^t) \hat q^t}  , \frac{\alpha m_F^t
             \hat m^t \,  X^0 + \xi \sqrt{\alpha q_F^t \hat q^t}}{\alpha Q_F^t \hat \chi^t -
             \alpha(Q_F^t -q_F^t) \hat q^t}   \right] \, ,\label{eq:DE_mx} 
\eea
where ${\cal D}\xi= {\rm d}\xi e^{-\xi^2/2}/{\sqrt{2\pi}}$ is a
Gaussian integration measure. Analogously we have for the $F$-related order parameters 
\bea
         Q^{t+1}_F - q_F^{t+1} &=& \int {\rm d}F^0 P_{F^0}(F^0)  \int
         {\cal D}\xi \,  \,  f_s\left[  \frac{1}{\pi  Q_X^t \hat \chi^t -
             \pi (Q_X^t -q_X^t) \hat q^t}  , \frac{\pi m_X^t
             \hat m^t \,  \sqrt{N}  F^0+ \xi \sqrt{\pi  q_X^t \hat q^t}}{\pi  Q_X^t \hat \chi^t -
             \pi (Q_X^t -q_X^t) \hat q^t}   \right] \, ,\\
         q_F^{t+1} &=& \int {\rm d}F^0 P_{F^0}(F^0)  \int
         {\cal D}\xi \,  \,  f_F^2\left[  \frac{1}{\pi  Q_X^t \hat \chi^t -
             \pi (Q_X^t -q_X^t) \hat q^t}  , \frac{\pi m_X^t
             \hat m^t\,  \sqrt{N} F^0 + \xi \sqrt{\pi  q_X^t \hat q^t}}{\pi  Q_X^t \hat \chi^t -
             \pi (Q_X^t -q_X^t) \hat q^t}   \right] \, ,\label{eq:DE_qF}  \\
          m_F^{t+1} &=& \int {\rm d}F^0 P_{F^0}(F^0)  \int
         {\cal D}\xi \, \,  \sqrt{N} F^0 f_F\left[  \frac{1}{\pi  Q_X^t \hat \chi^t -
             \pi (Q_X^t -q_X^t) \hat q^t}  , \frac{\pi m_X^t
             \hat m^t \,  \sqrt{N} F^0 + \xi \sqrt{\pi  q_X^t \hat q^t}}{\pi  Q_X^t \hat \chi^t -
             \pi (Q_X^t -q_X^t) \hat q^t}   \right] \, .\label{eq:DE_mF} 
\eea
The six equations (\ref{eq:DE_Qx}-\ref{eq:DE_mx}) together with (\ref{eq:chihat}-\ref{eq:mhat}) are
the general form of density evolution for GAMP in the general case of matrix
factorization. We remind that these equations describe the asymptotic 
evolution of the algorithm in the ``thermodynamic'' limit of large sizes, as long as the matrices $X$, $X^0$, $F$,
$F^0$ were generated with iid elements, at least one of the random
variables $X$ and
$F$, and at least one of the $X^0$ and $F^0$ has zero mean, and the
output channel satisfies the condition (\ref{eq:zero_m}). The first of
these condition is absolutely essential for our approach;  the
restriction to zero means is here for convenience, the non-zero means and
generic form of output function can be treated with the same formalism
that we have used here, with
additional terms.

\subsubsection{State evolution of the Bayes-optimal inference}
\label{sec:AA_NL}

To satisfy the Nishimori identities we have to suppose that all the prior
distributions are matching the true distributions from which the
signal and noise were generated, i.e. conditions~(\ref{eq:P=P0}) hold.
The condition $P^0_{\rm out}(y|z) = P_{\rm out}(y|z)$ is equivalent to $P_0(w)=P(w)$ and $h(z,w)=h^0(z,w)$.
In this Bayes-optimal setting, (\ref{eq:P=P0}),  the asymptotic analysis simplifies considerably
since we have 
\bea
       q^t_X &=& m^t_X   \, , \quad \quad      q^t_F = m^t_F \, ,   \label{eq:NL_qm}  \\
     Q^t_X&=& \mathbb{E}[ (X^0)^2]      \, , \quad \quad             Q^t_F =
     N  \mathbb{E}[ (F^0)^2]  \, ,
     \label{eq:NL_Q}  \\
     \hat \chi^t &=& \hat q^t = \hat m^t \, . \label{eq:NL_chi} 
\eea
To justify the above statement we need to prove that if (\ref{eq:P=P0}) is
satisfied and (\ref{eq:NL_qm}-\ref{eq:NL_chi}) hold up to iteration $t$ then (\ref{eq:NL_qm}-\ref{eq:NL_chi}) hold also in
iteration $t+1$. This is done in the next two subsections. 

In the Bayes-optimal setting, the state evolution simplifies into
\bea
           m_X^{t+1} &=& \int {\rm d}X P_{X}(X)  \int
         {\cal D}\xi \, \,  f_X^2\left[  \frac{1}{
             \alpha m_F^t \hat m^t}  , \frac{\alpha m_F^t
             \hat m^t\,  X + \xi \sqrt{\alpha m_F^t \hat m^t}}{
             \alpha m_F^t \hat m^t}   \right] \, ,\label{eq:mx_NL} \\
           m_F^{t+1} &=& \int {\rm d}F P_{F}(F)  \int
         {\cal D}\xi \,  \,  f_F^2\left[  \frac{1}{
             \pi m_X^t \hat m^t}  , \frac{\pi m_X^t
             \hat m^t\,  \sqrt{N} F + \xi \sqrt{\pi  m_X^t \hat m^t}}{
             \pi m_X^t \hat m^t}   \right] \, ,\label{eq:mF_NL} \\ 
             \hat m^t &=& - \int {\rm d}w \, P(w)  \int
         {\rm d} p\,  {\rm
               d}z   \frac{  e^{-\frac{p^2}{2m_F^t m_X^t}}  e^{-\frac{(z-p)^2}{2[ \langle(z^0)^2 
             \rangle - m_F^t m_X^t]}} }{2\pi \sqrt{ m_F^t m_X^t (   \langle(z^0)^2 
             \rangle - m_F^t m_X^t  ) }} \,    \partial_p g_{\rm out}(p,h(z,w),
             \langle(z^0)^2 
             \rangle - m_F^t m_X^t ) \, , \label{eq:mhat_NL} 
\eea
where $ {\cal D}\xi $ is a Gaussian integral ${\rm d} \xi
e^{-\xi^2/2}/\sqrt{2\pi}$. 
Here we chose to
use the expression coming from eq. (\ref{eq:DE_qx}), (\ref{eq:DE_qF}) and (\ref{eq:chihat}),
but we
could have used any of the other expressions that are equivalent on
the Nishimori line. Where $m_F$ and $m_X$ are initialized as
squares of the means of the corresponding prior distributions
\be
     m_F^{t=0} = N \left[\int {\rm d}F \, F P(F)\right]^2\, ,  \quad \quad   m_X^{t=0} = \left[ \int
     {\rm d}X \, X P(X) \right]^2\, .  \label{eq:DE_NL_init}
\ee
In case the prior distribution depends on another random variables,
e.g. in case of matrix calibration, we take additional average with
respect to that variable. 
If the above initialization gives $m_F^{t=0}=0$ and $m_X^{t=0}=0$ then
this is a fixed point of the state evolution. This is due to the
permutational symmetry between the columns of matrix $F$ and rows of matrix $X$. To obtain a nontrivial
fixed point we initialize at $m_F^{t=0}=\eta$ for some very small
$\eta$, corresponding to an infinitesimal prior information about the
matrix elements of the matrix $F$. Note that this is needed only in
the state evolution, the algorithm breaks the permutational symmetry
spontaneously. The same situation appears in an Ising ferromagnet at
low temperature where zero magnetization is a fixed point of the
equilibrium equations, but the physically correct solution to which
dynamical procedures converge had large magnetization in absolute
value.

Our general strategy in the asymptotic analysis of optimal Bayesian
inference and related phase transition is that the corresponding fixed
points must satisfy the Nishimori identities, hence we will restrict
our search for fixed points to parameters lying on the Nishimori line,
i.e. satisfying the identities (\ref{eq:NL_qm}-\ref{eq:NL_chi}). When
these identities are not imposed the iterations of the state evolution
equations are not always converging to fixed points on the Nishimori
line. This is also reflected in problems with convergence in the GAMP
algorithm for matrix factorization. In the algorithm the Nishimori
identities are unfortunately not straightforward to impose.

\subsubsection{The input Nishimori identities}
\label{app:Nish}

Assume that in the state evolution the Nishimori identities (\ref{eq:NL_qm}-\ref{eq:NL_chi})
hold for all iteration times smaller or equal to $t$. Out aim is to
show that then $m_X^{t+1}$ and $q_X^{t+1}$ computed from (\ref{eq:DE_mx}) and
(\ref{eq:DE_qx}) are equal. 
Recall that
\bea
m_X^{t+1} &=& \int {\rm d}X_0 P_{X}(X_0)  \int
{\cal D}\xi \, \,  X_0 f_X\left[  \frac{1}{
    \alpha m_F^t \hat m^t}  , \frac{\alpha m_F^t
    \hat m^t\,  X_0 + \xi \sqrt{\alpha m_F^t \hat m^t}}{
    \alpha m_F^t \hat m^t}   \right] \, ,\\
q_X^{t+1} &=& \int {\rm d}X_0 P_{X}(X_0)  \int
{\cal D}\xi \, \,  f_X^2\left[  \frac{1}{
    \alpha m_F^t \hat m^t}  , \frac{\alpha m_F^t
    \hat m^t\,  X_0 + \xi \sqrt{\alpha m_F^t \hat m^t}}{
    \alpha m_F^t \hat m^t}   \right] \, ,\\
\eea
and the definition of 
\be
f_X(\Sigma^2,R) =\frac{\int {\rm d}X \, e^{-\frac{(X-R)^2}{2\Sigma^2}}
  X\, P_{X}(X)}{\int {\rm d}X\,  e^{-\frac{(X-R)^2}{2\Sigma^2}} P_{X}(X)}
=\frac{\int {\rm d}X \, e^{-\frac{X^2}{2\Sigma^2}+\frac{R}{\Sigma^2}X}
  X \, P_{X}(X)}{\int {\rm d}X \,
  e^{-\frac{X^2}{2\Sigma^2}+\frac{R}{\Sigma^2}X}  P_{X}(X)}\, .
\ee
Denoting $\tilde q=\alpha m_F^t \hat m^t$ we have:
\bea
m_X^{t+1} &=& 
\int {\rm d}X_0 \, P_{X}(X_0)  \int{\cal D}\xi \, \,  X_0
\frac{\int {\rm d}X\,  e^{-\frac{\tilde q X^2}{2}+ ( \tilde q X_0 +\xi
    \sqrt{\tilde q}  )X }
  X\, P_{X}(X)}{\int {\rm d}X\,  e^{-\frac{\tilde q X^2}{2}+( \tilde q X_0 +\xi
    \sqrt{\tilde q} )X }  P_{X}(X)}\, ,
\\
q_X^{t+1} &=& 
\int {\rm d}X_0 P_{X}(X_0) \int{\cal
 D}\xi \, \,  \left[ \frac{\int {\rm d}X \, e^{-\frac{\tilde q X^2}{2}+ (  \tilde q X_0 +\xi
    \sqrt{\tilde q}  )X }
  X\, P_{X}(X)}{\int {\rm d}X \,  e^{-\frac{\tilde q X^2}{2}+( \tilde q  X_0 +\xi
    \sqrt{\tilde q}  )X}  P_{X}(X)} \right]^2.
\eea
Performing the change of variables $\xi\sqrt{\tilde q}  + \tilde q X_0
\to \xi \sqrt{\tilde q} $ the
Gaussian measure become ${\cal D}\xi e^{-\frac {\tilde q X_0^2}2  +\xi
  X_0\sqrt{\tilde q}}$ so that
\be
m_X^{t+1} = \int{\cal D} \xi \int {\rm d}X_0 P_{X}(X_0)  e^{-\frac {\tilde q X_0^2}2
  +\xi X_0 \sqrt{\tilde q}}\,  X_0\, 
\frac{\int {\rm d}X\,  e^{-\frac{\tilde q X^2}{2}+ \xi
    \sqrt{\tilde q}  X }
  X P_{X}(X)}{\int {\rm d}X \,  e^{-\frac{\tilde q X^2}{2}+\xi
    \sqrt{\tilde q} X }  P_{X}(X)}
= \int{\cal D}\xi \frac{\left[\int {\rm d}X \,  e^{-\frac{\tilde q X^2}{2}+ \xi
    \sqrt{\tilde q}  X }  X P_{X}(X)\right]^2}{\int {\rm d}X\,  e^{-\frac{\tilde q X^2}{2}+\xi
    \sqrt{\tilde q} X }  P_{X}(X)}\, .
\ee
Analogously we obtain
\be
q_X^{t+1} = \int{\cal D}\xi \int {\rm d}X_0 P_{X}(X_0) e^{-\frac {\tilde q X_0^2}2
  +\xi\sqrt{\tilde q}X_0} \left[\frac{\int {\rm d}X\,  e^{-\frac{\tilde q X^2}{2}+ \xi
    \sqrt{\tilde q}  X }
 X\,  P_{X}(X)}{\int {\rm d}X \, e^{-\frac{\tilde q X^2}{2}+\xi
    \sqrt{\tilde q} X }  P_{X}(X)}\right]^2 = 
m_X^{t+1}\, .
\ee
The proof of $m_F^{t+1}= q_F^{t+1}$ is exactly the same. 

The next identity we want to prove is $Q_X^{t+1}= \mathbb{E}[ (X^0 )^2]$. From the general state evolution equation (\ref{eq:DE_Qx})
we get under conditions (\ref{eq:NL_qm}-\ref{eq:NL_chi}) that
\be
Q_X^{t+1} -  q_X^{t+1} = \int {\rm d}X_0 P_{X}(X_0)  \int
{\cal D}\xi \, \,  f_c\left[  \frac{1}{
    \alpha m_F^t \hat m^t}  , \frac{\alpha m_F^t
    \hat m^t\,  X_0 + \xi \sqrt{\alpha m_F^t \hat m^t}}{
    \alpha m_F^t \hat m^t}   \right]\, .
\ee
Using the definition of the function $f_c(\Sigma,R)$ (\ref{f_ac_gen}), the same change of variables, and resulting cancelations as above we get 
\be
Q_X^{t+1}  =  \int{\cal D}\xi \int {\rm d}X_0 \,  X_0^2\, P_{X}(X_0) e^{-\frac {\tilde q X_0^2}2
  +\xi\sqrt{\tilde q}X_0} =  \mathbb{E}[ (X^0 )^2 ]\, .
\ee
And analogously for $Q_F^{t+1}=N  \mathbb{E}[ (F^0 )^2 ]$.

\subsubsection{The output Nishimori identities}
\label{app:Nish_out}

Let us now assume that the input Nishimori identities
(\ref{eq:NL_qm}-\ref{eq:NL_Q}) are satisfied and we want to show that
(\ref{eq:NL_chi}) and (\ref{eq:zero_m}) hold. 

We depart from the general expressions
(\ref{eq:chihat}-\ref{eq:mhat}). We notice that for $q_F^t=m_F^t$ and
$q_X^t=m_X^t$ the joint Gaussian measure for variables $p$ and $z$ in
(\ref{eq:qq}-\ref{eq:zz}) can be written as a product of two Gaussian
measures. We have in that case $ \mathbb{E} [\omega_{\mu l } (\omega_{\mu l} - z^0_{\mu l})
] = 0$, hence one of the Gaussian has zero mean and variance $m_F^t
m_X^t$ and the other one mean $p$ and variance $V^t=\langle (z^0)^2 \rangle-m_F^t
m_X^t$. Furthermore, performing the integration over variable
$z$ by parts in eq. (\ref{eq:mhat}) and then using the relation between
$h(z,w)$ and $P_{\rm out}(y|z)$ from eq. (\ref{eq:hPout}) and the definition of
$g_{\rm out}$ from eq. (\ref{eq:def_gout}) we get 
\be
       \hat m^t = \int {\rm d}y \, {\rm d}p\,  {\rm d}z\,  P_{\rm out}(y|z)
     \,   {\cal N}(p,z)  \frac{z-p}{V^t}   \frac{ \int {\rm
           d}z'  P_{\rm out}(y|z') (z'-p)  e^{-\frac{(z'-p)^2}{2V^t}}   }{ V^t \int {\rm
           d}z''  P_{\rm out}(y|z'')  e^{-\frac{(z''-p)^2}{2V^t}}   }
       \, . 
\ee
Doing analogous manipulations of expliciting the Gaussian measure and
using eq. (\ref{eq:hPout}) and the definition of $g_{\rm out}$ in equation (\ref{eq:qhat}) we obtain 
\be
        \hat q^t = \hat m^t \, . 
\ee
For $\hat \chi^t$ we do integration with respect to $p$ by parts and
using steps as in the above we obtain 
\be
     \hat \chi^t = - \frac{1}{m_F^t m_X^t} \int {\rm d}y \, {\rm d}p\,
     {\rm d}z \, P_{\rm out}(y|z)\, 
       {\cal N}(p,z)  \, p \,   \frac{ \int {\rm
           d}z'  P_{\rm out}(y|z') (z'-p)  e^{-\frac{(z'-p)^2}{2V^t}}   }{ V^t \int {\rm
           d}z''  P_{\rm out}(y|z'')  e^{-\frac{(z''-p)^2}{2V^t}}   }
       + \hat q^t = \hat q^t\, .
\ee
Thanks to a cancelation between the integrals over variables $z$
and $z''$ we can perform explicitly the integral over $y$ (keeping in
mind that $P_{\rm out}(y|z')$ is a normalized probability
distribution). The remaining Gaussian integral is then zero.  

In a analogous manner we prove eq.~(\ref{eq:zero_m}) by noticing that
the expectation with respect to $F^0$, $X^0$ and $w^0$ is exactly the
integral $\int {\rm d}y \, {\rm d}p\,  {\rm d}z\,  P_{\rm out}(y|z) \,  {\cal N}(p,z)$.

To conclude, identities (\ref{eq:NL_chi}) and (\ref{eq:zero_m}) hold
in the limit $N\to \infty$. However, as common in statistical physics
we can recall the self-averaging property under which quantities on
almost every large ($N\to \infty$) instance are equal to their averages aver
randomness (disorder) of $F^0$, $X^0$ and $w^0$. This self-averaging
then for instance justifies the use of eq.~(\ref{eq:gout_NL}) on large single instances of
the matrix factorization problem.

\subsection{Replica method}
\label{Sec:repl}

The replica method is known as a non-rigorous approach to evaluate the typical performance
of various Bayesian inference problems.  We here show how this is employed for 
the matrix factorization problem. We show, as expected from other problems, that the 
result of the replica analysis is fully equivalent to the result of
the state evolution. 

\subsubsection{Moment assessment for $n \in \mathbb{N}$}
The expression of the partition function
\begin{eqnarray}
Z(Y)=\int {\rm d}F {\rm d}X \prod_{\mu, i} P_F(F_{\mu i}) \prod_{i,l} P_X(X_{i l}) \prod_{\mu, l}P_{\rm out}\((y_{\mu l}|\sum_{i}F_{\mu i}X_{il}\))
\label{partition_replica}
\end{eqnarray}
constitutes the basis of our analysis. In  statistical physics, 
one can generally examine properties of systems via evaluation of the free entropy
$\log Z(Y)$, which statistically fluctuates depending on the realization of $Y$ in the current case. 
However, as $N,M,P \to \infty$, one can expect that the 
self-averaging property holds and, therefore,
the free entropy density $N^{-2} \log Z(Y)$ converges to its typical value $\phi \equiv N^{-2} \left [ \log Z(Y) \right ]_Y$
with probability of unity. 
This is also expected to hold for other macroscopic quantities relevant to the performance
of the matrix factorization. Therefore, assessment of $\phi$ is the central issue in our analysis. 

This can be systematically carried out by the replica method. For this, we first evaluate the $n$-th moment of $Z(Y)$, 
$\left [ Z^n(Y) \right ]_Y=\int {\rm d}Y P_0(Y) Z^n(Y)$,  
for $n \in \mathbb{N}$ utilizing an identity
\begin{eqnarray}
Z^n(Y)=\int \prod_{a=1}^n\{{\rm d}F^a {\rm d}X^a \prod_{\mu, i} P_F(F_{\mu i}^a) \prod_{i,l} P_X(X_{i l}^a) \}
\times \{\prod_{\mu, l} \prod_{a=1}^n P_{\rm out}(y_{\mu l}|\sum_{i}F_{\mu i}^aX_{il}^a)\}, 
\label{replicated}
\end{eqnarray}
with respect to the generative distribution of $Y$
\begin{eqnarray}
P_0(Y)=\int {\rm d}F^0 {\rm d}X^0 
\prod_{\mu, i} P_{F^0}(F_{\mu i}^0) 
\prod_{i,l }P_{X^0}(X_{i l}^0) 
\prod_{\mu , l} P_{\rm out}^0(y_{\mu l}|\sum_{i}F_{\mu i}^0X_{il}^0), 
\label{generative_dist}
\end{eqnarray}
where we assumed that the functional forms of $P_{F^0}(F_{\mu i})$, $P_{X^0}(X_{il})$ and 
$P_{\rm out}^0(y_{\mu l}| \sum_{i}F_{\mu i}X_{il})$ may be different from 
those of the assumed model $P_F(F_{\mu i})$, $P_{X}(X_{il})$ and $P_{\rm out}(y_{\mu l}|\sum_{i}F_{\mu i}X_{il})$
for generality. When they are equal, which correspond to the
Bayes-optimal setting, $P_0(Y)=Z(Y)$ holds. 

In performing the integrals of $2(n+1)$ matrices $(F^0,\{F^a\}_{a=1}^n )$, $(X^0,\{X^a\}_{a=1}^n )$ and $Y$ that come out 
in evaluating $\left [ Z^n(Y) \right ]_Y$, we insert trivial identities with respect to all combinations of 
replica indices $a \le b=0,1,2,\ldots,n$
\begin{eqnarray}
1= M \int {\rm d}q_F^{ab} \delta \left ( F^{a} \cdot F^{b} - M q_F^{ab} \right )
\label{order_parameter_F}
\end{eqnarray}
and 
\begin{eqnarray}
1= NP \int {\rm d}q_X^{ab} \delta \left ( X^{a} \cdot X^{b} - NP q_X^{ab} \right )
\label{order_parameter_}
\end{eqnarray}
to the integrand, where $F^a \cdot F^b \equiv \sum_{ \mu , i} F_{\mu i}^a F_{\mu i}^b$ and similarly for $X^a \cdot X^b$. 
Let us denote ${\cal Q}_F\equiv (q_F^{ab})$ and ${\cal Q}_X \equiv (q_X^{ab})$, and introduce two joint distributions
\begin{eqnarray}
P_F(\{F^a\} ; {\cal Q}_F) =\frac{1}{V_F({\cal Q}_F)} \prod_{\mu, i}
\left (P_{F^0}(F_{\mu i}^0)  \prod_{a=1}^n P_F(F_{\mu i}^a) \right ) \prod_{a \le b} \delta \left ( F^{a} \cdot F^{b} - M q_F^{ab} \right )
\label{subshell_F}
\end{eqnarray}
and 
\begin{eqnarray}
P_X(\{X^a\} ; {\cal Q}_X) =\frac{1}{V_X({\cal Q}_X)} \prod_{ i,l}
\left (P_{X^0}(X_{il}^0)  \prod_{a=1}^n P_X(X_{il}^a) \right ) \prod_{a \le b} \delta \left ( X^{a} \cdot X^{b} - NP q_X^{ab} \right ), 
\label{subshell_X}
\end{eqnarray}
where $V_F({\cal Q}_F)$ and $V_X({\cal Q}_X)$ are the normalization constants. These yield an expression of $\left [ Z^n(Y) \right ]_Y$ as
\begin{eqnarray}
&&
\left [ Z^n(Y) \right ]_Y =\int {\rm d}Y {\rm d}(M {\cal Q}_F) {\rm d}(NP {\cal Q}_X) 
\left \{V_F({\cal Q}_F)V_X({\cal Q}_X) \right . \cr
&& \left .  \hspace*{5cm} \times 
\left [ \prod_{\mu , l} 
\left (P_{\rm out}^0(y_{\mu l}|\sum_{i}F_{\mu i}^0X_{il}^0) 
\prod_{a=1}^n P_{\rm out}(y_{\mu l}|\sum_{i}F_{\mu i}^aX_{il}^a) \right )\right ]_{{\cal Q}_F,{\cal Q}_X} \right \}, 
\label{new_moment}
\end{eqnarray}
where $\left [ \cdots \right ]_{{\cal Q}_F,{\cal Q}_X}$ denotes the average with respect to (\ref{subshell_F}) and (\ref{subshell_X}). 
In computing $\left [ \cdots \right ]_{{\cal Q}_F,{\cal Q}_X}$, it is noteworthy that 
$\{F^a\}$ and $\{X^a\}$ follow statistically independent distributions, and 
either of them has zero mean and both of them have finite variances from our assumption.   
These allow us to handle $z_{\mu l}^a \equiv \sum_{i}F_{\mu i}^a X_{\mu i}^a$ $(a=0,1,2,\ldots,n; 
\mu =1,2,\ldots, M; l=1,2,\ldots, P)$ as
multivariate Gaussian random variables whose distribution is given by 
\begin{eqnarray}
P_Z(\{z_{\mu l}^a \}|{\cal Q}_F,{\cal Q}_X)=\prod_{\mu, l} \frac{1}{\sqrt{(2\pi)^{n+1} \det {\cal T} }}
\exp \left (-\frac{1}{2} \sum_{a,b} z_{\mu l}^a \left ({\cal T}^{-1} \right )z_{\mu l}^b \right ), 
\end{eqnarray}
where ${\cal T}=(q_F^{ab} q_X^{ab}) \in \mathbb{R}^{(n+1) \times (n+1)}$. 
Employing this and evaluating the integrals of ${\cal Q}_F$ and ${\cal Q}_X$ by means of 
the saddle point method yield an expression
\begin{eqnarray}
\frac{1}{N^2}
\log \left [ Z^n(Y) \right ]_Y=
\mathop{\rm extr}_{{\cal Q}_F,{\cal Q}_X} 
\left \{\alpha \pi {\cal I}_{FX} ({\cal Q}_F,{\cal Q}_X) + \alpha {\cal I}_F({\cal Q}_F)+\pi {\cal I}_X({\cal Q}_X) \right \}, 
\label{logZn}
\end{eqnarray}
where 
\begin{eqnarray}
{\cal I}_{FX} ({\cal Q}_F,{\cal Q}_X) =\log \left (\int  \left (\int  
P_Z(\{z^a \}|{\cal Q}_F,{\cal Q}_X)
\left (P_{\rm out}^0 (y|z^0) \prod_{a=1}^n P_{\rm out} (y|z^a) \right )\prod_{a=0}^n 
{\rm d} z^a\right ){\rm d} y\right ), 
\end{eqnarray}
\begin{eqnarray}
{\cal I}_F({\cal Q}_F) &=& \frac{1}{NM}\log V_F({\cal Q}_F) \cr
&=&\mathop{\rm extr}_{\hat{\cal Q}_F} \left \{\frac{1}{2} {\rm Tr} \hat{\cal Q}_F {\cal Q}_F + \log \left (
\int  
P_{ F^0}(F^0)\prod_{a=1}^n P_F(F^a) \exp \left (-\frac{N}{2} \bF^{\rm T} \hat{\cal Q}_F \bF \right ) \prod_{a=0}^n {\rm d}F^a\right ) \right \}, 
\end{eqnarray}
\begin{eqnarray}
{\cal I}_X({\cal Q}_X) &=& \frac{1}{NP}\log V_X({\cal Q}_X) \cr
&=&\mathop{\rm extr}_{\hat{\cal Q}_X}\left \{ \frac{1}{2} {\rm Tr} \hat{\cal Q}_X {\cal Q}_X + \log \left (
\int  
P_{ X^0}(X^0)\prod_{a=1}^n P_X(X^a) \exp \left (-\frac{1}{2} \bX^{\rm T} \hat{\cal Q}_X \bX \right ) \prod_{a=0}^n {\rm d}X^a\right ) \right \}, 
\end{eqnarray}
and $\mathop{\rm extr}_{\Theta} \left \{ \cdots \right \}$ denotes the operation of extremization with respect to $\Theta$.  
$\hat{\cal Q}_F=(\hat{q}_F^{ab}) \in \mathbb{R}^{(n+1) \times (n+1)}$ are introduced for the saddle point evaluation of $V_{F}({\cal Q}_F)$
and $\bF=(F^a) \in \mathbb{R}^{n+1}$, and similarly for $\hat{\cal Q}_X\in \mathbb{R}^{(n+1)\times (n+1)}$ and $\bX =(X^a)\in \mathbb{R}^{n+1}$. 

\subsubsection{Replica symmetric free entropy}
Literally evaluating (\ref{logZn}) yields the correct leading order estimate of $N^{-2} \log [Z^n(Y) ]_{Y}$ for 
each $n \in \mathbb{N}$. However, we here restrict the candidate of the dominant saddle point to 
\begin{eqnarray}
(q_F^{ab}, q_X^{ab},\hat{q}_F^{ab}, \hat{q}_X^{ab}) = \left \{
\begin{array}{ll}
(Q_F^0, Q_X^0, \hat{Q}_F^0, \hat{Q}_X^0), & a=b=0, \cr
(Q_F,Q_X,\hat{Q}_F,\hat{Q}_X), & a=b \ (a,b \ne 0), \cr
(q_F,q_X,-\hat{q}_F,-\hat{q}_X), & a\ne b \ (a, b \ne 0), \cr
(m_F,m_X,-\hat{m}_F,-\hat{m}_X), & a=0, b\ne 0, 
\end{array}
\right .
\label{RSorderparameter} 
\end{eqnarray}
in order to obtain an analytic expression  with respect to $n \in
\mathbb{R}$. This restricted form is called the replica symmetric
ansatz, because the different non-zero replica-indices $a$ and $b$ play the
same role. It turns out that this ansatz is equivalent to the
assumptions made in the cavity method \cite{MezardParisi87b}. 
After some algebra utilizing 
a formula $\exp ( A \sum_{a < b} u^a u^b ) =\exp \left (-\frac{A}{2}\sum_{a} (u^a)^2\right )\int {\cal D} \xi \exp \left (\sqrt{A} \xi\sum_{a} u^a \right )$
for $A \ge 0$, this provides
\begin{eqnarray}
&&{\cal I}_{FX} ({\cal Q}_F,{\cal Q}_X) =\log \left (\int  {\rm d} y {\cal D}\xi \left (
\int {\cal D}u^0 P_{\rm out}^0 \left (y| \sqrt{Q_F^0 Q_X^0-\frac{m_F^2 m_X^2}{q_F q_X}} u^0 +
\frac{m_F m_X}{\sqrt{q_F q_X}} \xi \right ) \right . \right . \cr
&& \hspace*{5cm} \left . \left .  \times 
\left (\int {\cal D} u P_{\rm out} \left (y| \sqrt{Q_F Q_X-q_F q_X} u +\sqrt{q_F q_X} \xi \right ) \right )^n
\right )\right ), 
\end{eqnarray}
\begin{eqnarray}
&&{\cal I}_{F} ({\cal Q}_F )=\mathop{\rm extr}_{\hat{Q}_F^0, \hat{Q}_F,\hat{q}_F,\hat{m}_F} \left \{
\frac{\hat{Q}_F^0 Q_F^0}{2}+\frac{n\hat{Q}_FQ_F}{2}-\frac{n(n-1)\hat{q}_Fq_F}{2}-n\hat{m}_F m_F \right . \cr
&& \hspace*{2cm} \left . +\log \left (\int {\cal D}\xi  {\rm d}F^0e^{-\frac{N\hat{Q}_F^0}{2}(F^0)^2} P_{F^0}(F^0) 
\left (\int {\rm d}F e^{-\frac{N(\hat{Q}_F +\hat{q}_F)}{2} F^2 +(\sqrt{N\hat{q}_F} \xi +N\hat{m}_F F^0) F} P_F(F)\right )^n \right )  \right \}, 
\end{eqnarray}
and 
\begin{eqnarray}
&&{\cal I}_{X} ({\cal Q}_X )=\mathop{\rm extr}_{\hat{Q}_X^0, \hat{Q}_X,\hat{q}_X,\hat{m}_X} \left \{
\frac{\hat{Q}_X^0 Q_X^0}{2}+\frac{n\hat{Q}_XQ_X}{2}-\frac{n(n-1)\hat{q}_Xq_X}{2}-n\hat{m}_X m_X \right . \cr
&& \hspace*{2cm} \left . +\log \left (\int {\cal D}\xi  {\rm d}X^0e^{-\frac{\hat{Q}_X^0}{2}(X^0)^2} P_{X^0}(X^0) 
\left (\int {\rm d}X e^{-\frac{\hat{Q}_X +\hat{q}_X}{2} X^2 +(\sqrt{\hat{q}_X} \xi +\hat{m}_X X^0) X} P_X(X)\right )^n \right )  \right \}, 
\end{eqnarray}
all of which are analytic with respect to $n \in \mathbb{R}$. 
Substituting these into an identity $N^{-2} \left [ \log Z(Y) \right ]_Y=\lim_{n\to 0} $ $\frac{\partial }{\partial n} N^{-2} 
\log \left [Z^n (Y) \right ]_Y$ leads to the general expression of the free entropy of the matrix factorization problems. 
In this, $Q_F^0=N\int dF^0 (F^0)^2 P_{F^0}(F^0)=N \mathbb{E}[ (F^0)^2 ]$, 
$Q_X^0=\int {\rm d}X^0 (X^0)^2 P_{X^0}(X^0)= \mathbb{E}[ (X^0)^2 ]$, $\hat{Q}_F^0=0$, and $\hat{Q}_X^0=0$ are enforced 
for the consistency of $\lim_{n\to 0} \left [Z^n (Y) \right ]_Y =1$. After taking these into account, the expression becomes as
\begin{align}
&\phi=\frac{1}{N^2} \left [ \log Z(Y) \right ]_Y \cr
&= 
{\rm extr}\left \{ \alpha \pi \int {\rm d}y \,  {\cal D} \xi \, 
{\cal D}u^0 P_{\rm out}^0 \left (y| \sqrt{Q_F^0 Q_X^0-\frac{m_F^2 m_X^2}{q_F q_X}} u^0 +
\frac{m_F m_X}{\sqrt{q_F q_X}} \xi \right )
\log \left (\int {\cal D} u \, P_{\rm out} \left (y| \sqrt{Q_F Q_X-q_F q_X} u +\sqrt{q_F q_X} \xi \right ) \right ) \right . \cr
&+ \alpha \left (\frac{\hat{Q}_F Q_F}{2}+ \frac{\hat{q}_F q_F}{2}-\hat{m}_F m_F +
\int {\cal D}\xi \,  {\rm d}F^0
P_{F^0}(F^0)
\log\left (\int {\rm d}F e^{-\frac{N(\hat{Q}_F+\hat{q}_F)}{2}F^2+(\sqrt{N\hat{q}_F} \xi+N\hat{m}_F F^0 )F} P_F(F)\right )  \right ) \cr
&\left . + \pi \left (\frac{\hat{Q}_X Q_X}{2}+ \frac{\hat{q}_X q_X}{2}-\hat{m}_X m_X +
\int {\cal D}\xi \,  {\rm d}X^0
P_{X^0}(X^0)
\log\left (\int {\rm d}X e^{-\frac{\hat{Q}_X+\hat{q}_X}{2}X^2+(\sqrt{\hat{q}_X} \xi+\hat{m}_X X^0 )X} P_X(X) \right )  \right ) \right \}, 
\label{generalRSfreeenergy}
\end{align}
where the extremization with respect to
$Q_F,Q_X,q_F,q_X,m_F,m_X$, and their conjugate variables leads
to the same equations we obtained in the state evolution. Particularly
extremization w.r.t. the six conjugate variables gives
eqs.~(\ref{eq:DE_Qx}-\ref{eq:DE_mF}). Extremization with respect to
$Q_F,Q_X,q_F,q_X,m_F,m_X$ gives $\hat Q_F = \pi Q_X (\hat \chi-\hat q)$, $\hat Q_X
= \alpha Q_F (\hat \chi -\hat q)$, $\hat q_F = \pi q_X \hat q$, $\hat q_X = \alpha
q_F \hat q$, $\hat m_F = \pi m_X \hat m$, $\hat m_X = \alpha m_F \hat
m$ where $\hat \chi$, $\hat q$, and $\hat m$ are given byt
eqs.~(\ref{eq:chihat}-\ref{eq:qhat}) and (\ref{eq:mhat}). 

\subsubsection{Simplification in the Bayes-optimal setting}
One can generally evaluate thermodynamically dominant values of $Q_F,Q_X,q_F,q_X,m_F,m_X$ by solving the extremization 
problem of (\ref{generalRSfreeenergy}), which is involved with twelve variables including the conjugate variables
and therefore is rather complicated to handle. However, the problem is
significantly simplified in the Bayes-optimal setting. 

This is because the Nishimori identities 
allows us to handle $F^0=(F_{\mu i}^0)$ and $X^0=(X_{il}^0)$ as if 
they were the $n+1$-st replica variables added to the 
$n$-replicated system composed of $\{F^a\}_{a=1}^n=\{(F_{\mu i}^a)\}_{a=1}^n$ and $\{X^a\}_{a=1}^n=\{(X_{il}^a)\}_{a=1}^n$ 
in the computation of the moment 
 $\left [Z^n(Y) \right ]_Y=\int dY P_0(Y) 
Z^n(Y) =\int dY P_0^{n+1}(Y)$.  
The replica symmetry among $a=0,1,2,\ldots,n$ ensures the following properties:
\begin{itemize}
\item $m_F=q_F$, $m_X=q_X$, $\hat{m}_F=\hat{q}_F$, and $\hat{m}_X=\hat{q}_X$ are satisfied. 
\item $Q_F=Q_F^0$, $Q_X=Q_X^0$, $\hat{Q}_F=0$, and $\hat{Q}_X=0$ hold for $n\to 0$. 
\end{itemize}
Substituting these into (\ref{generalRSfreeenergy}) yields a simplified expression of the free entropy, which is involved with {\em only}
four macroscopic variables  
$m_F$, $m_X$, $\hat{m}_F$, and $\hat{m}_X$,  as
\begin{eqnarray}
\phi&=&\frac{1}{N^2} \left [ \log Z(Y) \right ]_Y =\frac{1}{N^2}\int dY P_0(Y) \log P_0(Y) \cr
&=& {\rm extr}\left \{ \alpha \pi \int {\rm d}y\,  {\cal D} \xi \, 
{\cal D}u^0 P_{\rm out} \left (y| \sqrt{Q_F^0 Q_X^0-m_F m_X} u^0 +{\sqrt{m_F m_X}} \xi \right )
\log \left (\int {\cal D} u \, P_{\rm out} \left (y| \sqrt{Q_F^0 Q_X^0-m_F m_X} u +\sqrt{m_F m_X} \xi \right ) \right ) \right . \cr
&&\hspace*{1cm}+ \alpha \left (- \frac{\hat{m}_F m_F}{2}+
\int {\cal D}\xi\,  {\rm d}F^0e^{-\frac{N\hat{m}_F}{2}(F^0)^2+\sqrt{N\hat{m}_F}\xi F^0} P_{F}(F^0)
\log\left (\int {\rm d}F e^{-\frac{N\hat{m}_F}{2}F^2+\sqrt{N\hat{m}_F} \xi F} P_F(F)\right )  \right ) \cr
&&\hspace*{1cm}\left . + \pi \left (- \frac{\hat{m}_X m_X}{2}+
\int {\cal D}\xi \, {\rm d}X^0
e^{-\frac{\hat{m}_X}{2}(X^0)^2+\sqrt{\hat{m}_X}\xi X^0}P_{X}(X^0)
\log\left (\int {\rm d}X e^{-\frac{\hat{m}_X}{2}X^2+\sqrt{\hat{m}_X} \xi X} P_X(X) \right )  \right ) \right \}, 
\label{NishimoriRSfreeenergy}
\end{eqnarray}
where we changed integration variables as 
$\sqrt{N\hat{m}_F}\xi +N\hat{m}_F F^0 \to \sqrt{N\hat{m}_F}\xi$ 
and 
$\sqrt{\hat{m}_X}\xi +\hat{m}_X X^0 \to \sqrt{\hat{m}_X}\xi$ 
together with 
${\cal D}\xi \to {\cal D}\xi e^{-\frac{N\hat{m}_F}{2} (F^0)^2 +\sqrt{N\hat{m}_F} \xi F^0}$
and 
${\cal D}\xi \to {\cal D}\xi e^{-\frac{\hat{m}_X}{2} (X^0)^2 +\sqrt{\hat{m}_X} \xi X^0}$, 
respectively. 
The saddle point conditions can be summarized via equations that we
obtained in the state evolution, notably eqs.~(\ref{eq:mx_NL}-\ref{eq:mhat_NL}) with $\hat m_F = \pi m_X \hat m$, $\hat m_X = \alpha m_F \hat
m$. 

The free entropy (\ref{NishimoriRSfreeenergy}) would be also obtained
from an asymptotic limit of the Bethe free entropy of section
\ref{Sec:Bethe}. Overall, as usual in statistical physics, the cavity method and
the replica method yield equivalent results for the matrix
factorization problem. 

\section{Examples of asymptotic phase diagrams}
\label{sec:phase}

In this section we use the state evolution derived in section
\ref{sec:AA_NL} to analyze the asymptotic MMSE of the Bayes-optimal inference in matrix factorization
for applications listed in section \ref{sec:examples}. We restrict our  analysis to  the Bayes-optimal inference,
i.e. the case where we generate the data as specified in
Sec.~\ref{sec:examples} and assume that we know the corresponding distributions. In terms of the AMP
algorithm and the state evolution this means we can use all the
simplifications that arise under the Nishimori identities. The AMP algorithms for matrix factorization and the asymptotic
analysis derived in sections~\ref{sec:AMP} and \ref{Sec:AA} apply to
all the examples, the only elements that are application-dependent
are the ``input" functions $f_X$ and $f_F$, and the output
function $g_{\rm out}$.

\subsection{Dictionary learning, blind matrix calibration, sparse PCA
  and blind source separation}
\label{sec:DL_phase}

In terms of our asymptotic analysis the equations for dictionary
learning, sparse PCA, and blind source separation are very close, see definitions in section \ref{sec:examples}, these problems basically
differ by the region of parameters $\alpha, \pi$ that is of interest. 
Moreover the dictionary learning can be seen as the $\eta \to \infty$
limit of the blind calibration problem (which is trivially taken in
the equations). We hence group the discussion of these problems in
the present section, they all present the first order phase transition
as low measurement noise. 

\subsubsection{Input and output functions}

The matrices $F$ and $X$ in our setting of the dictionary learning and
sparse PCA problems are generated according to
eqs.~(\ref{eq:DL_PF}-\ref{eq:DL_PX}). Using the definitions of the input
function $f_X$ in eq.~(\ref{f_ac_gen}), and $f_F$ in
eq.~(\ref{f_rs_gen}) we obtain explicitly: \be f_F(Z,W) =
\frac{W}{1+Z}\, , \quad \quad f_X(\Sigma,T) =\frac{ \rho\,
  e^{-\frac{(T-\overline X)^2}{2(\Sigma+\sigma)}}
  \frac{\sqrt{\Sigma}}{(\Sigma+\sigma)^{\frac{3}{2}}} (\overline X
  \Sigma + T \sigma) }{ (1-\rho) e^{-\frac{T^2}{2\Sigma}} + \rho
  \frac{\sqrt{\Sigma}}{\sqrt{\Sigma+\sigma}} e^{-\frac{(T-\overline
      X)^2}{2(\Sigma+\sigma)}} } \, , \label{eq:f_ar} \ee The
functions $f_c(\Sigma,T)$ and $f_s(Z,W)$ are then obtained from
eqs.~(\ref{eq:deriv_fa}-\ref{eq:deriv_fr}).

In case of matrix calibration we have some prior knowledge on the
matrix $F$ given by eq.~(\ref{eq:PF_cal}). This leads to a function $f_F$ of
the form 
\be f_F(Z,W) =
\frac{W + \sqrt{N} F_{\mu i}'
  \frac{\sqrt{1+\eta}}{\eta}}{Z +1+ \frac{1}{\eta}} \, . 
\ee
Indeed as the uncertainty in the matrix $\eta \to \infty$ this goes to
the $f_F$ for dictionary learning or sparse PCA (\ref{eq:f_ar}). 

The output function $g_{\rm out}$ defined in eq.~(\ref{eq:DL_Pout}) is,
for the output channel (\ref{eq:def_gout}) with 
additive white Gaussian noise of variance~$\Delta$:
\be
     g_{\rm out}(\omega,y,V)  = \frac{y - \omega}{ \Delta + V }\, . \label{eq:out_GN}
\ee 
For such a simple output function the
eqs.~(\ref{eq:chihat}-\ref{eq:mhat}) in the density evolution simplify
greatly into 
\be
\hat \chi^t = \hat m^t = \frac{1}{\Delta + Q_F^t Q_X^t -q_F^t q_X^t}
\, , \quad \quad  
\hat q^t = \frac{\Delta_0 +\rho_0 (\overline X_0^2 + \sigma_0)+ q_F
  q_X -2 m_F m_X} {(\Delta + Q_F^t Q_X^t -q_F^t q_X^t)^2}\, ,
\ee 
which under the simplification of the Nishimori line gives
\be
   \hat \chi^t =\hat q^t =\hat m^t = \frac{1}{\Delta + \rho (\overline
     X^2 + \sigma) -m_F^t m_X^t} \, . \label{eq:mhat_DL}
\ee
This is the only equation in the state evolution that explicitly
depends on the variance of the measurement noise $\Delta$. 
Also eqs.~(\ref{eq:DE_Qx}-\ref{eq:DE_mF}) simplify for distributions
(\ref{eq:DL_PF}) and (\ref{eq:PF_cal}) and using the Nishimori identities
they reduce to a
pair of equations 
\bea
m^{t+1}_F &=& \frac{\frac{1}{\eta}+ \pi m_X^t \hat m^t}{
  (1+\frac{1}{\eta}) +  \pi m_X^t \hat m^t} \, , \label{eq:mF}\\  
     m^{t+1}_X &=& (1-\rho) \int {\cal D}z  \,  f_X^2\left(\frac{1 }{\alpha
       m_F^t \hat m^t}, z
     \frac{1}{\sqrt{\alpha
       m_F^t \hat m^t } }   
\right)  + \rho \int {\cal D} z  \, f_X^2 \left(
   \frac{1 }{\alpha m_F^t \hat m^t}, z
     \frac{\sqrt{\alpha
       m_F^t \hat m^t+ 1  }}{\sqrt{\alpha
       m_F^t \hat m^t }}  \right)\, . \label{eq:mx}
\eea
Note that indeed only eq.~(\ref{eq:mF}) depends on the matrix
uncertainty parameter $\eta$, and the dictionary learning limit
$\eta\to \infty$ is straightforward. The MMSE of the matrix $F$, defined in (\ref{MSE_F_def}), predicted by
this state evolution is then $E_F$. The MMSE of the matrix
$X$, defined in (\ref{MSE_X_def}), is found equal to $E_X$ 
\be
E_F= 1 - m_F\, ,  \quad \quad E_X=  \rho( \overline X^2 + \sigma ) - m_X\, .
\ee
The two sets of initial conditions that we will
analyze to investigate the MMSE and the associated phase transitions are
\begin{itemize}
   \item Random (uninformative) initialization: $m^{t=0}_X=0$, and
     $m^{t=0}_F=1/(1+\eta)$. Note that in the limit of dictionary
     learning $\eta \to \infty$ this initialization corresponds to a
     fixed point of the state evolution equations. 
     This fixed point reflects the $N\!$ permutational symmetry in the
     dictionary learning problem, and its instability corresponds to a
     spontaneous breaking of this symmetry. In the limit of dictionary
     learning we will hence initialize the state evolution with
     $m^{t=0}_F$ being a very small positive constant, and we will see
     the behavior will not depend on its precise value. 
   \item Planted (informative) initialization: $m^{t=0}_X=\rho(
     \overline X^2 + \sigma ) - \delta_X$, and
     $m^{t=0}_F=1-\delta_F$, where $\delta_X$ and $\delta_F$ are small positive
     constants to test the ``stability" of the zero MMSE point. 
\end{itemize}

The free entropy density from which we compute the limiting performance
of the Bayes-optimal inference in case of a first order phase
transition is expressed from (\ref{NishimoriRSfreeenergy}) as follows.
For simplicity from now on (till the end of this section \ref{sec:DL_phase}) we
analyze only the case where the mean of the elements of $X^0$ was
zero, $\overline X=0$, and the variance of the nonzero ones was one, $\sigma=1$) 
 \bea
  && \phi(m_X,m_F) = - \frac{\alpha \pi}{2} + \frac{\alpha \pi}{2}
     \log{(\Delta+\rho - m_X m_F)}   + \alpha  \pi  \frac{\Delta+\rho}{\Delta+\rho - m_X m_F} 
 \nonumber\\ &-& \pi (1-\rho) \int
    {\cal D}z \, \log{\left[  1 -\rho + \frac{\rho}{\sqrt{\alpha m_F \hat m +1}}
        e^{\frac{z^2 \alpha m_F \hat m}{2(\alpha m_F \hat m +1)}} \right]} -  \pi \rho \int
    {\cal D}z \, \log{\left[  1 -\rho + \frac{\rho}{\sqrt{\alpha m_F \hat m +1}}
        e^{\frac{z^2 \alpha m_F \hat m}{2}} \right]}
      \nonumber\\ &-& \frac{\alpha}{2} \left[ \pi m_X \hat m -
        \log{\left(1+\frac{\pi m_X \hat m }{ 1+ \frac{1}{\eta}
            }\right)}   \right] \, ,\label{eq:Bethe_DL}
\eea
where $\hat m$ is given by eq.~(\ref{eq:mhat_DL}). The dependence on
the matrix uncertainty is only in the last term and the limit of the
completely unknown matrix $F$ is easily taken by $\eta\to \infty$.

The simple Eqs.~(\ref{eq:mhat_DL})-(\ref{eq:Bethe_DL}) is all we
need at the end to analyze the MMSEs $E_X$ and $E_F$ of dictionary learning, blind matrix
calibration, sparse PCA and blind source separation problems when the signal
$X^0$ and $F^0$ are generated according to eqs.~(\ref{eq:DL_PF}-\ref{eq:DL_PX}) with
$\overline X = 0$ and $\sigma = 1$, $\rho$ is the fraction of nonzero
elements in $X^0$, and $\eta$ is the matrix uncertainty from
(\ref{eq:eta_BMC}). The parameter $\alpha=M/N$ is the ratio between
number of lines and the number of columns of $F^0$, $\pi = P/N$ is the ration
between number of columns and the number of lines of $X^0$. The output channel has additive
white Gaussian noise of variance $\Delta$, and this information about
distributions and their parameters is used in the posterior likelihood
in the optimal Bayes inference. 

\subsubsection{Phase diagram for blind matrix calibration and
  dictionary learning} 

\paragraph{Identifiability threshold in zero measurement noise:}
For the
     noiseless case $\Delta=0$, and all positive $\eta>0$, the linear stability analysis of
     eqs.~(\ref{eq:mhat_DL}-\ref{eq:mx}) around the informative fixed point
     $m_X=\rho$, $m_F=1$ (i.e. the MMSEs $E_F=E_X=0$) leads to an update for the perturbations
     $\delta_F^{t+1}=(\delta_F^t \rho +\delta_X)/(\alpha \pi)$ and
     $\delta_X^{t+1}= \rho(\delta_F^t \rho +\delta_X)/\alpha$. By
     computations of the largest eigenvalue of the corresponding $2
     \times 2$ matrix we obtain that this informative fixed point is stable if and only if
     \be \pi>\pi^*(\alpha,\rho) \equiv \alpha/(\alpha-\rho) \, . \label{eq:pistar}\ee In other words the zero MSE fixed
     point is locally stable above the counting
     lower bound (\ref{eq:bound_DL}). 

Further we notice that in the low noise limit $\Delta \to 0$, for
all positive and finite $\eta$, and for
$m_X=\rho - \delta_X$, $m_F=1-\delta_F$ with $\delta_X$ and $\delta_F$
being small positive constants of the same order as $\Delta$ the free
entropy (\ref{eq:Bethe_DL}) becomes in the leading order $(\pi \alpha - \pi \rho
-\alpha)\log(\Delta+\delta_X +\rho \delta_F)/2$. For $\pi >
\pi^*(\alpha,\rho)$ this is a large positive value, and a large
negative value for $\pi < \pi^*(\alpha,\rho)$. 

Hence for the noiseless measurements $\Delta=0$ the asymptotic Bayes-optimal MMSE is $E_X=0$ and $E_F=0$
for $\pi > \pi^*(\alpha,\rho)$ for all $\eta>0$. This is a
remarkable result as it implies that in the Bayes-optimal setting the
dictionary is identifiable (or
an exact calibration of the matrix possible) as
soon as the number of samples $P$ per signal element is larger than
the value $\pi^*(\alpha,\rho)$ given by the trivial counting bound
(\ref{eq:bound_DL}). 

\paragraph{Identifiability versus achievability gap:}
The next question is whether this MMSE is achievable in a
computationally tractable
way. To answer this we study the state evolution starting in the
uninformative initialization. First we analyze the behavior of the
state evolution when $m_X=\delta_X$, and $m_F=\delta_F$ where both
$\delta_X$, and $\delta_F$ are positive and small, while we also consider
$\eta$ being very large. The linear expansion of state evolution
update then leads to $\delta_X^{t+1} = \rho^2
\alpha \delta_F / (\Delta + \rho)$ and $\delta_F = 1/\eta + \pi
\delta_X/(\Delta + \rho)$. Hence for $\eta\to \infty$ the uninformative initialization is in fact a
stable fixed point of the state evolution equations  as long as
\be 
\pi \le \pi_F \equiv  (\Delta + \rho)^2/(\alpha \rho^2)\, . 
\ee 
This means that
for $\pi^*(\rho,\alpha) < \pi < \pi_F(\Delta,\rho,\alpha)$ the MMSE is not
achievable in the dictionary learning (e.g. when $\alpha<1$ 
$\pi^*(0,\alpha) < \pi_F (0,0,\alpha)$) with the approximate message passing presented in this
paper.  This simple analysis leads us to the conclusion that a first
order phase transition is in play in the dictionary learning problem,
and as we will see also in the blind calibration ($\eta < \infty$) and
sparse PCA ($\alpha>1$).

As a side remark let us remind that the limit $\eta\to 0$ should lead to results known
from Bayesian compressed sensing. In particular in compressed sensing for low noise the
matrix $X$ is identifiable if and only if $\alpha>\rho$. To reconcile
this with the previous results notice that indeed for $\eta = c
\Delta \to 0$ with $c=O(1)$ the leading term of the free entropy
becomes $\pi(\alpha - \rho) \log(\Delta)$. Hence compressed sensing
result is recovered.  Whereas for $1 \gg \eta \gg
\Delta$ it is the dictionary learning static phase transition $\pi^* = 
\alpha / (\alpha-\rho)$ that is the relevant one. 


\paragraph{Phase diagrams for blind calibration and dictionary learning}

\begin{figure}[!ht]
\centering
\includegraphics[width=3.2in]{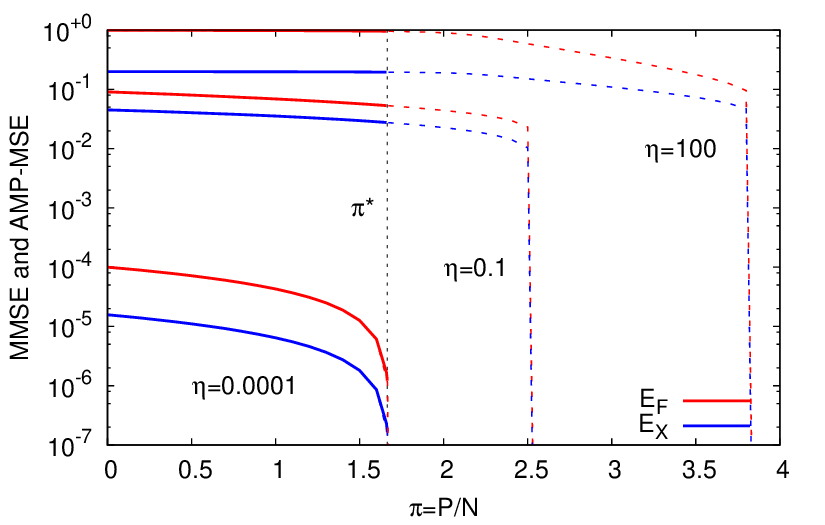}
\includegraphics[width=3.2in]{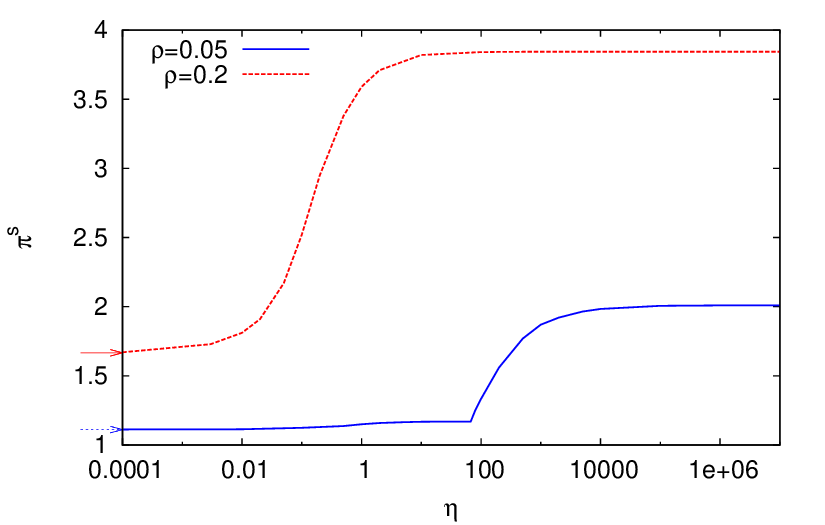}
\caption{Left: Blind matrix calibration: The predicted MSE $E_F$ (for
  the matrix estimation) and $E_X$ (for the signal estimation)
  corresponding to $\rho=0.2$, $\alpha=0.5$, $\Delta=0$, for three
  values of $\eta$. MMSE is in full lines, AMP-MSE is in dashed
  lines. The MMSE jumps abruptly from a finite value to zero at the
  phase transition point $\pi^*$ (\ref{eq:pistar}).  However, the
  AMP-MSE matches the MMSE only when the sample complexity is larger
  than the spinodal transition, that takes place at a larger value
  $\pi^s(\eta)>\pi^*$. The AMP-MSE is zero for $\pi>\pi^s(\eta)$.  Right: We
  plot the value of the spinodal transition at which the AMP-MSE has a
  discontinuity as function of $\eta$ for $\alpha=0.5$, $\Delta=0$ and
  two different values of the sparsity $\rho$. Arrows on the left
  mark the static transition $\pi^*$ and we see that $\lim_{\eta \to
    0}\pi^s(\eta) = \pi^*$. In the limit of dictionary learning,
  $\eta\to \infty$, the spinodal transition converges to a finite
  value. Interestingly, for small values of the density,
  e.g. $\rho=0.05$, we see a sharp phase transition in the threshold
  $\pi^s(\eta)$, in this case at $\eta \approx68$.  }
\label{fig_MMSE_DL}
\end{figure}

\begin{figure}[!ht]
\centering
\includegraphics[width=3.4in]{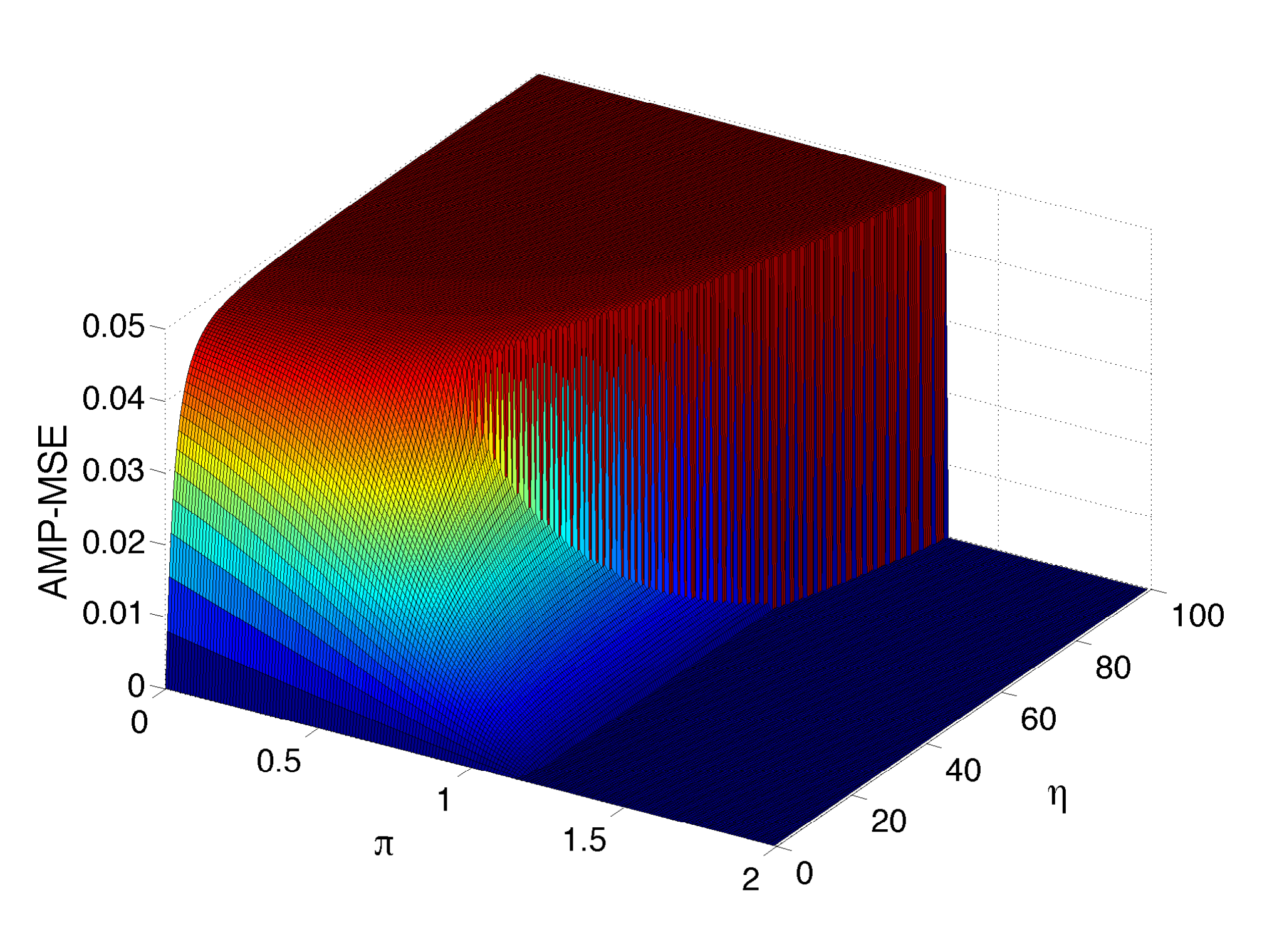}
\includegraphics[width=3.2in]{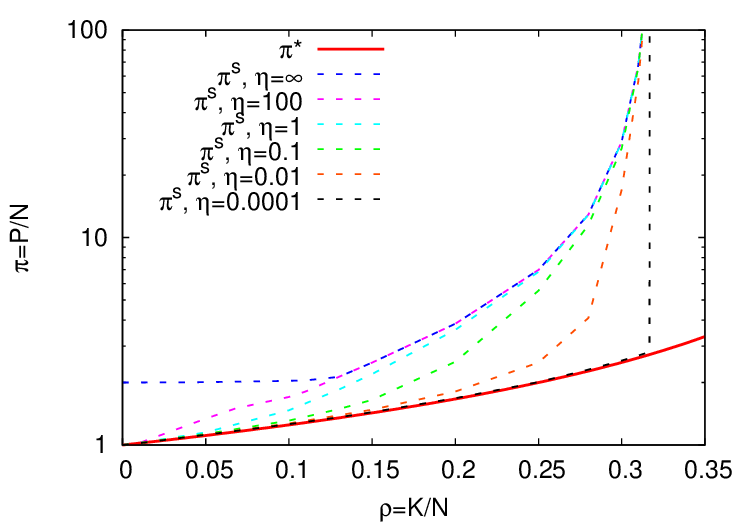}
\caption{Left: The AMP-MSE of the signal matrix $X$ is plotted against $\pi$ and $\eta$ for
  the blind matrix calibration at
  $\rho=0.05$, $\alpha=0.5$ and $\Delta=0$.  The transition as a
  function of the number of samples $\pi$ for fixed matrix
  uncertainty $\eta$ is always discontinuous in this figure, for
  larger values of $\eta$ this discontinuity is, however, much more
  pronounced. 
  Right: The phase diagram of dictionary learning and blind matrix calibration for
  $\alpha\!=\!0.5$ and $\Delta\!=\!0$. In both cases, the matrix $F$
  is identifiable by the Bayes-optimal inference above the full red line
  $\pi^*=\frac{\alpha}{\alpha-\rho}$. However, the system undergoes
  the spinodal transition
  $\pi^s(\eta)$ shown here with a dashed line for dictionary learning
  ($\eta\!=\!\infty$) and for blind calibration ($\eta$ finite). Below
  the spinodal transition, the MMSE is not achieved with AMP. 
  Notice that for $\eta\to 0$ the
  spinodal line converges to $\pi^*$ for $\rho < \rho^{\rm CS}_{\rm
    BP}$. For all values of $\eta$ the spinodal line diverges as $\rho \to \rho^{\rm CS}_{\rm
    BP}$. The threshold value $\rho^{\rm CS}_{\rm
    BP}=0.317$ for $\alpha=0.5$ is the (spinodal) phase transition of pure compressed
  sensing (when the matrix $F$ is fully known), from
  \cite{KrzakalaMezard12}.
Notice also that in the limit of learning a dictionary with a very
small sparsity, $\eta\to \infty$ and $\rho\to 0$, the  spinodal line
converges to $\pi^s=2$ whereas the information theoretic transition is
at $\pi=1$. Such a gap in the low $\rho$ behavior is striking. }
\label{fig_MMSE_cal}
\end{figure}

\begin{figure}[!ht]
\centering
\includegraphics[width=3.2in]{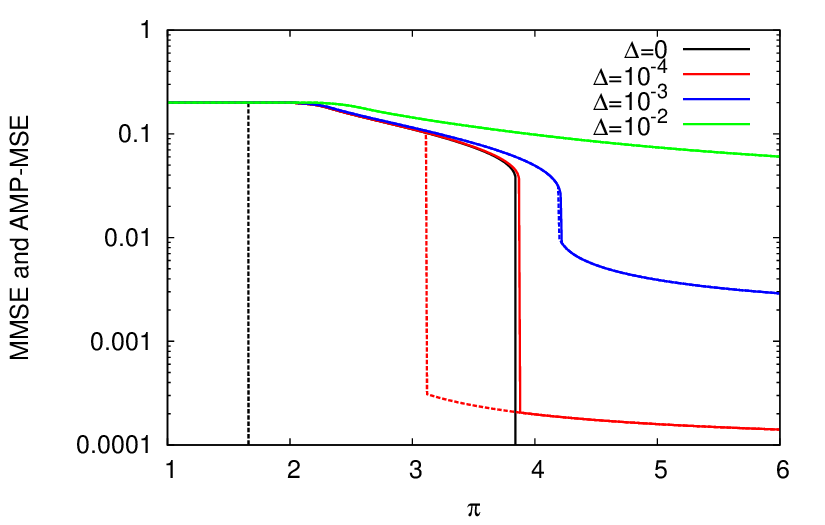}
\includegraphics[width=3.2in]{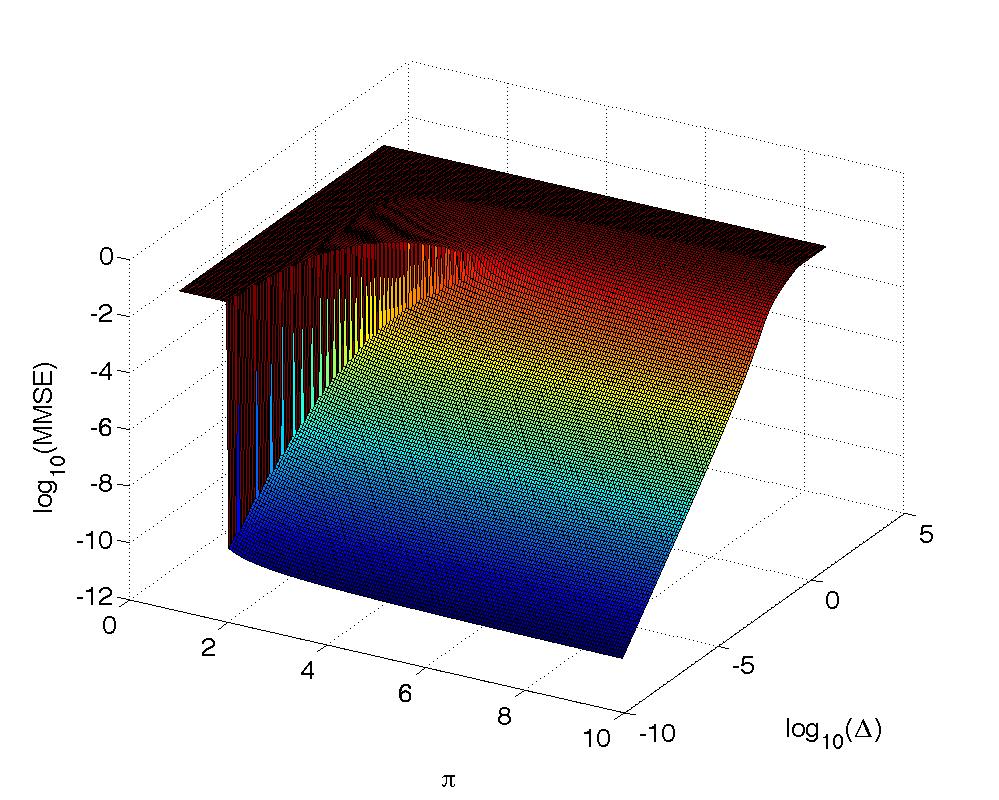}
\caption{ Dictionary learning with noisy measurements. Left: The MMSE (dashed lines) and the AMP-MSE (full lines) for the signal matrix $X$ as a function of the number of samples $P=\pi N$
  for $\alpha=0.5$, $\rho=0.2$ for various values of the measurement
  noise $\Delta$.  The behavior is qualitatively the same as for the
  noiseless case, the only difference is that the MMSE at
  $\pi>\pi^*(\Delta)$, and the AMP-MSE at $\pi>\pi^s(\Delta)$ are no
  longer strictly zero but rather $O(\Delta)$. If the additive noise
  $\Delta$ is large enough, however, the phase transition disappears
  and the MMSE$=$AMP-MSE is continuous (this is the case in this plot
  for $\Delta=10^{-2}$). Right: A surface plot of the MMSE of the signal $E_X$ in the
  $\Delta$, $\pi$-plane in the case
  $\alpha=0.5$ and $\rho=0.2$. Notice how the sharp transition disappears at large
  noise where it is replaced by a smooth evolution of the MMSE.}
\label{fig_MMSE_DL_noisy}
\end{figure}

\begin{figure}[!ht]
\centering
\includegraphics[width=3.2in]{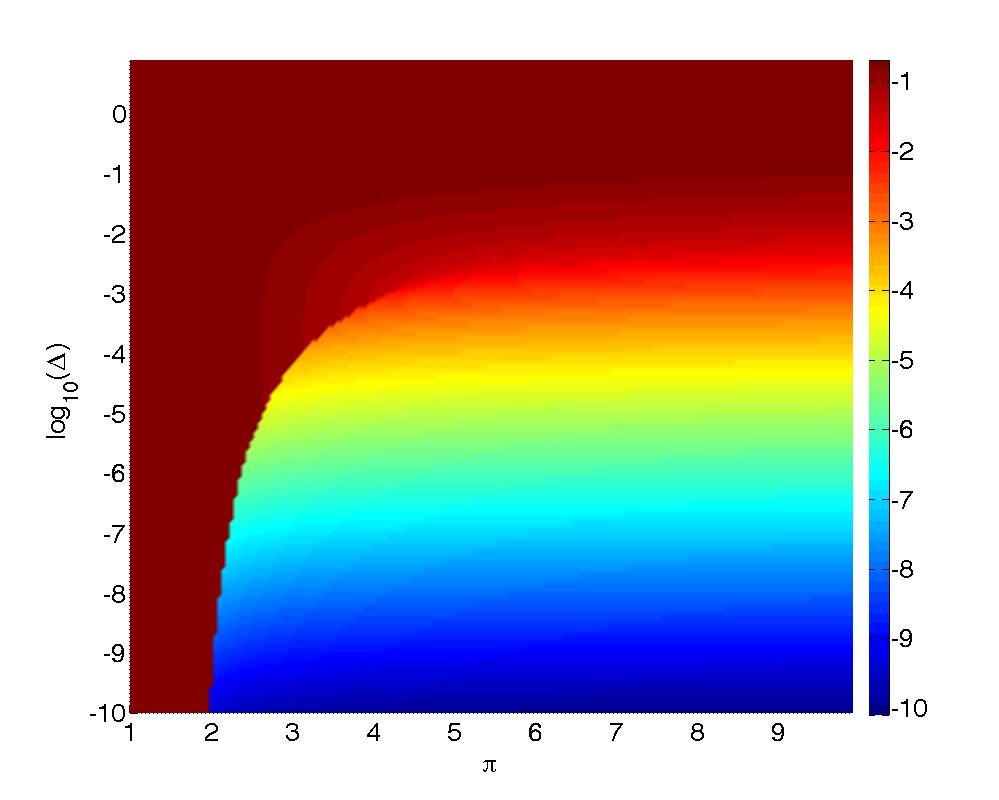}
\includegraphics[width=3.2in]{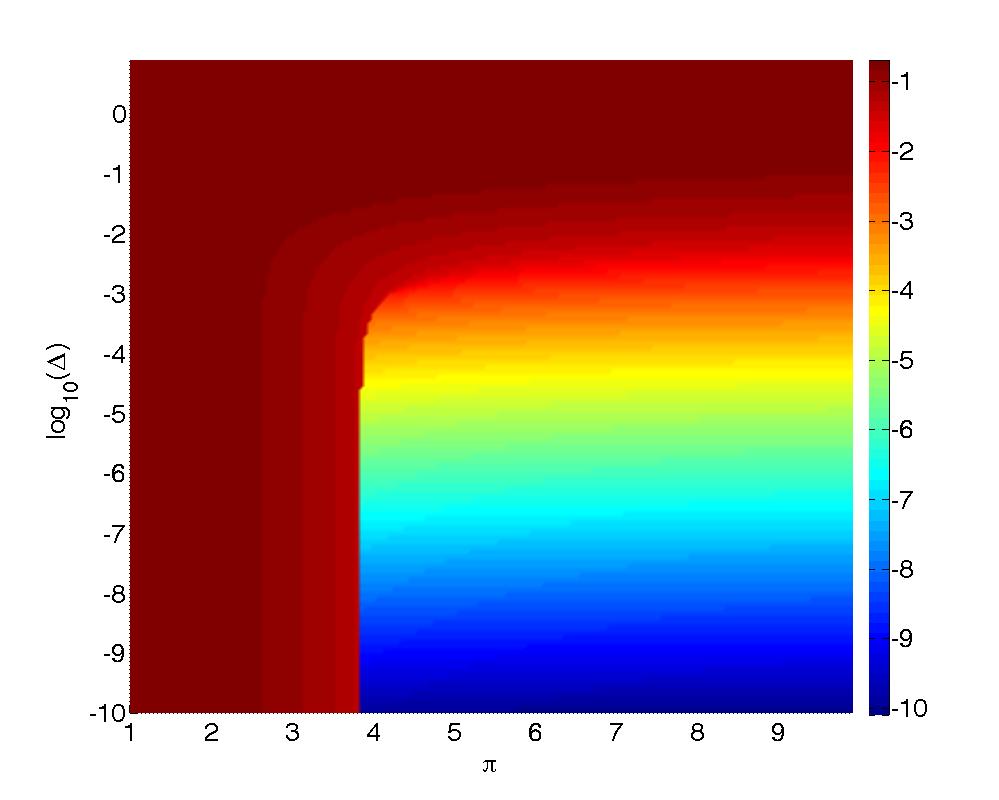}
\caption{Comparison of the MMSE (left) to the AMP-MSE (right) for noisy dictionary
  learning with  $\alpha=0.5$ and $\rho=0.2$, the MSE of the signal
  matrix $X$ is plotted on the $\Delta$, $\pi$-plane. The color-scale is in
  decadic logarithm of the MSEs. We clearly see
  the region where MMSE$<$AMP-MSE and the one where the two are equal.}
\label{fig_2d}
\end{figure}

Due to the close link between the two
problems, we shall described the results for dictionary learning together with the case of
blind matrix calibration. In both these cases we are typically trying to
learn (calibrate) an overcomplete dictionary $\alpha<1$ and a sparse
signal $X$ from as few samples $P$ as possible. We hence first plot in
Fig.~\ref{fig_MMSE_DL} (left)
the MSE for $F$ (in red) and for $X$ (in blue) as a function of $\pi
= P/N$ and fixed (representative) value of undersampling ratio
$\alpha=0.5$, and density $\rho=0.2$ in zero (or negligible)
measurement noise, $\Delta=0$. We consider several values of matrix uncertainty
$\eta$ (the larger the value the less we know about the matrix). The
AMP-MSE achieved from the uninformative initialization is depicted in
dashed lines, the MMSE achieved in this zero noise case from the
planted (informative) initialization is in full
lines.

Fig.~\ref{fig_MMSE_DL} (right) shows the value of the spinodal
transition $\pi^s$ as a function of the matrix uncertainty $\eta$. We
see that, as expected, $\lim_{\eta\to 0} \pi^s(\eta) = \pi^*$. A
result that is less intuitive is that the large $\eta$ limit is also
well defined and finite $\lim_{\eta \to \infty} \pi^s(\eta)=
\pi^s_{\rm DL}<0$. This
means that even in the dictionary learning where no prior information
about the matrix elements is available the dictionary is identifiable
with AMP for large system sizes above the spinodal transition $\pi^s_{\rm
  DL}$.

In Fig.~\ref{fig_MMSE_DL} (right) we also see an interesting behavior in the function
$\pi^s(\eta)$ for low values of $\rho$ - there is a sharp phase
transition from low $\eta$ regime, where the AMP-MSE at the transition
has a weak discontinuity towards a relatively low value of MSE, and a
high $\eta$ regime where the discontinuity is very abrupt towards a
value of MSE that is close to the completely uninformative value. This
behavior is also illustrated in Fig.~\ref{fig_MMSE_cal} (left) where we plot the
AMP-MSE as a function of $\pi$ and $\eta$ (for fixed $\rho=0.05$,
$\alpha=0.5$, $\Delta=0$).

Fig.~\ref{fig_MMSE_cal} (right) depicts the phase diagram of
dictionary learning ($\eta\to \infty$) and blind matrix calibration
(finite $\eta$) we plot the static threshold $\pi^*$
($\eta$-independent) above which the matrix $F$ is identifiable in
full red, and the spinodal threshold $\pi^s$ (for various values of
$\eta$) above which the AMP identifies asymptotically the original
matrix $F$ is dashed line. Notice that $\pi^s(\rho)$ diverges as $\rho
\to \rho^{\rm CS}_{\rm BP}$, where $\rho^{\rm CS}_{\rm BP}$ is the AMP
phase transition in compressed sensing, see
e.g. \cite{KrzakalaMezard12}. This is expected as for $\rho >
\rho^{\rm CS}_{\rm BP}$ the sparse signal cannot be recovered with AMP even if
the matrix $F$ is fully known.

From now on (also in following subsections) we will discuss only the
case $\eta\to \infty$ when no prior information about the dictionary
$F$ is available.  In Fig.~\ref{fig_MMSE_DL_noisy} (left) we illustrate the
results for dictionary learning with non-zero measurement noise. The
situation is qualitatively similar to what happens in noisy compressed
sensing \cite{KrzakalaMezard12}. The first order phase transition is
becoming weaker as the noise grows, until some value $\Delta^*$ above
which there is no phase transition anymore and the AMP-MSE$=$MMSE for
the whole range of $\pi$.  In Fig.~\ref{fig_MMSE_DL_noisy} (right) we then plot the
AMP-MSE and the MMSE of the signal $X$ as a function of the noise
variance $\Delta$ and $\pi$ for $\alpha=0.5$ and $\rho=0.2$. This
surface plot demonstrates how the sharp transition disappears and is
replaced by a continuous evolution of the MSE when the noise is
large. This MMSE is compared to the AMP-MSE for the same case in
Fig.~\ref{fig_2d}. 


\paragraph{Phase diagram for sparse PCA}

\begin{figure*}[!ht]
 \begin{center}
 \includegraphics[width=0.45\textwidth]{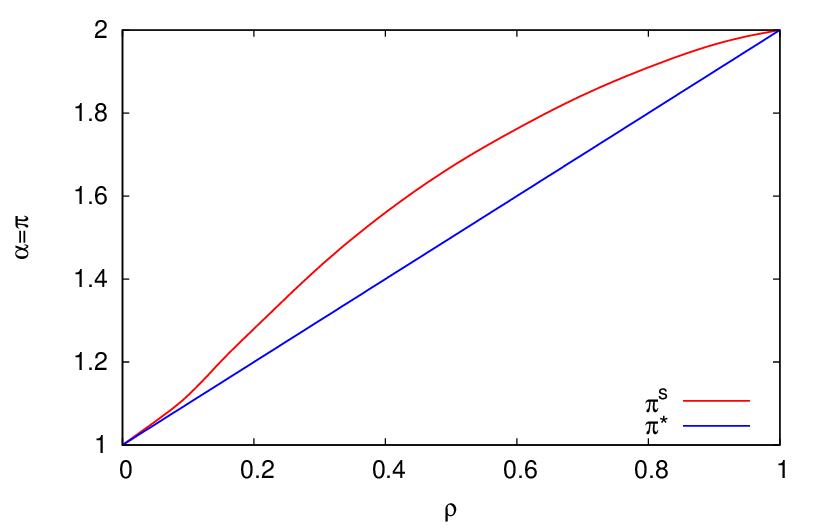}
\includegraphics[width=0.45\textwidth]{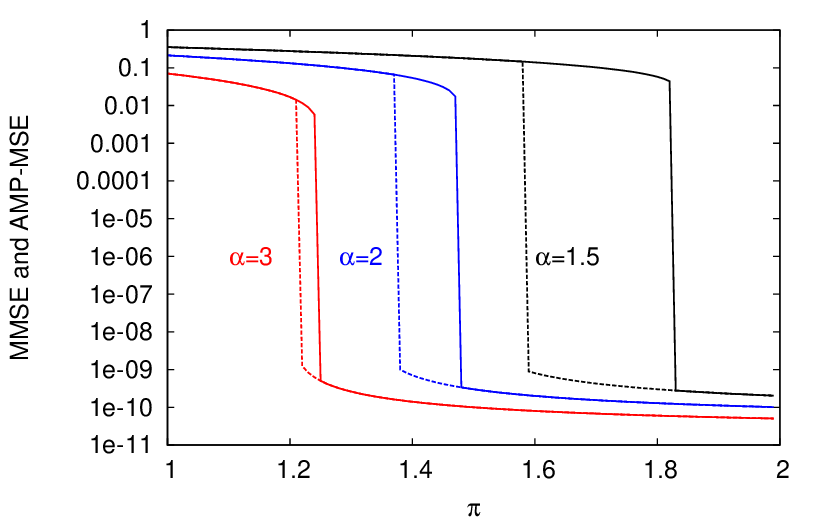}
 \caption{Phase diagram for sparse PCA. Left: Spinodal transition $\pi^s$ (upper red, where AMP-MSE
   goes to zero) and the
   optimal Bayes inference transition (lower
   blue, $\pi^*=\alpha^*=1+\rho$, where MMSE goes to zero) lines in zero
   measurement noise and when $\alpha=\rho$. Right: AMP-MSE and MMSE
   of the sparse matrix $X$ in sparse PCA  for $\rho=0.5$,
   $\Delta=10^{-10}$ and several values of ratio $\alpha$ as a
   function of $\pi$.}
\label{fig_SparsePCA_equal}
\end{center}
\end{figure*}

Sparse PCA as we set it in Section \ref{sec:examples} is closely
related to dictionary learning. Except that from the three sizes of
matrices $M$, $N$ and $P$ the smallest one is the $N$  --
corresponding to the matrix $Y$ to be of relatively low rank. Hence
for sparse PCA we should only really consider $\alpha>1$, $\pi>1$. 

Behavior of the state evolution is for this range of parameters
qualitatively very similar to the one we just observed in the previous
section. In Fig.~\ref{fig_SparsePCA_equal} left we treat the case of $M=P$,
i.e. $\pi = \alpha$, with zero measurement noise, and compute the smallest value for which the
matrices $F$ and $\rho$-sparse $X$ are recoverable for a given
$\rho$. Just as in the dictionary learning we obtain that the Bayes
optimal MMSE is zero everywhere above the counting bound $\pi = \alpha
> 1+\rho$,
blue line in Fig.~\ref{fig_SparsePCA_equal} left.  The AMP-MSE is, however,
zero only above the spinodal line, depicted in red in the figure. The
gap between the two lines is not very large in this case. 

The right part of Fig.~\ref{fig_SparsePCA_equal} shows the AMP-MSE and
the MMSE of the signal matrix $X$ for
measurement noise variance $\Delta=10^{-10}$ and density $\rho=0.5$.


\paragraph{Phase diagram for blind source separation}

\begin{figure*}[!ht]
 \begin{center}
 \includegraphics[width=0.45\textwidth]{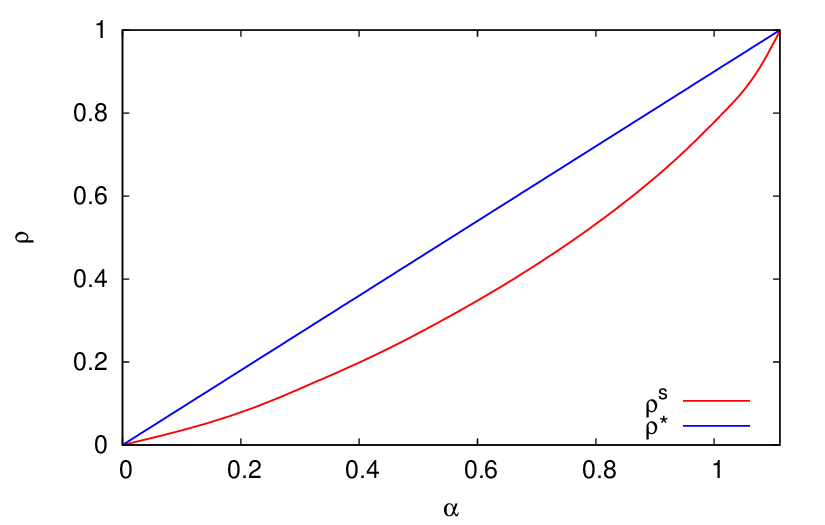}
 \caption{Blind source separation. At the measurement noise $\Delta=0$, $\pi = 10$ we plot the spinodal $\rho^s$ as a
   function of $\alpha $. Above this line AMP-MSE$>0$, whereas below
   this line
   AMP-MSE$=0$.  The static  transition line is $\rho^* = 0.9 \alpha$.}
\label{fig_BlindSource_equal}
\end{center}
\end{figure*}

In blind source separation, $P$ corresponds to the length of the
signal, $N$ is the number of sources, and $M$ the number of
sensors. Typically the signal is very long, i.e. one has $P \gg N$ and
$P \gg M$. The particularly interesting case is when there is more
sources than sensors $\alpha<1$ in that case the signal can be
reconstructed only if the density of the signal $\rho$ is smaller than
a certain value. The counting bound gives us $\rho < \alpha (\pi
-1)/\pi$ and this also corresponds to the value under which the MMSE
drops to zero under zero measurement noise. As in the previous case
also here we observe a first order phase transition and with AMP we
can reach zero error (in the noiseless case) only below $\rho^s$ that
we depict for $\pi=10$ as a function of $\alpha$ in
Fig.~\ref{fig_BlindSource_equal}.

\subsection{Low rank matrix completion}

In the remaining examples we treat cases in which neither $F$ nor $X$ are sparse,
$\rho=1$, we start with low rank matrix completion. 

In matrix completion the output function $g_{\rm out}$ is
eq. (\ref{eq:out_GN}) for the known matrix elements $\mu l$ (there is
$\epsilon M P$ of them), and
$g_{\rm out}(\omega,y,V)=0$ for the unknown elements $\mu l$ (there is
$(1-\epsilon) MP$ of them). The
input function $f_X$ (\ref{eq:f_ar}) for non-sparse matrix $X$, $\rho=1$, becomes  
\be
   f_X(\Sigma,T) = \frac{\overline X \Sigma +  T \sigma}{\Sigma +
     \sigma} \, .\label{eq:in_MC}
\ee

In the state evolution, using the Nishimori identities, we then have 
\be
        \hat m^t  =  \frac{\epsilon}{ \Delta  + \overline X^2 +
          \sigma -m_F^t m_X^t} \, ,
\ee
where $\epsilon$ is the fraction of known elements of $Y$. Moreover
for the input function  (\ref{eq:in_MC}) the state evolution equation
for $m_X$ becomes
\be
     m_X^{t+1} = \frac{ \overline X^2 + (\overline X^2 + \sigma)
       \sigma \alpha m_F^t \hat m^t   }{  1+ \sigma \alpha m_F^t \hat
       m^t  } \, . \label{eq:mx_1}
\ee
The equation for $m_F$ is the one of (\ref{eq:mF}).


It is instrumental to analyze the local stability of the informative
and the uninformative initialization for low rank matrix completion. For the informative initialization we consider $m_X^t = 1-
\delta_X^t$, and $m_F^t = 1- \delta_F^t$, where $\delta_X^t$, and
$\delta_F^t$ are small positive numbers. The state evolution update
equations at zero noise $\Delta=0$ lead to $\delta_X^{t+1} =
(\delta_X^t + \delta_F^t)/(\epsilon \alpha)$, and $\delta_F^{t+1} =
(\delta_X^t + \delta_F^t)/(\epsilon \pi)$. The largest eigenvalue of
this $2 \times 2$ system is $(\alpha + \pi)/(\epsilon \alpha \pi)$ and
hence the informative fixed point is stable for $\epsilon > \epsilon^*
=(\alpha +
\pi)/(\alpha \pi)$, which coincides with the counting bound
eq. (\ref{eq:bound_MC}). This means that  for $\epsilon > \epsilon^*$
the noiseless matrix completion, the matrices $F$ and $X$ can be
recovered without error (asymptotically). 

For the uninformative initialization we consider $m_X^t = \delta_X^t$, and $m_F^t =\delta_F^t$, where $\delta_X^t$, and
$\delta_F^t$ are again small positive numbers. This time the state
evolution equations give $\delta_F^{t+1} = \pi \epsilon \delta_X^t /
(1+\Delta)$ and $\delta_X^{t+1} = \alpha \epsilon \delta_F^t /
(1+\Delta)$. Hence the uninformative fixed point is stable for
$\epsilon < (\Delta + 1) / \sqrt{\pi \alpha}$. This can indeed be
verified in Fig.~\ref{LRMC}. 

\begin{figure}[!ht]
\centering
\includegraphics[width=3.2in]{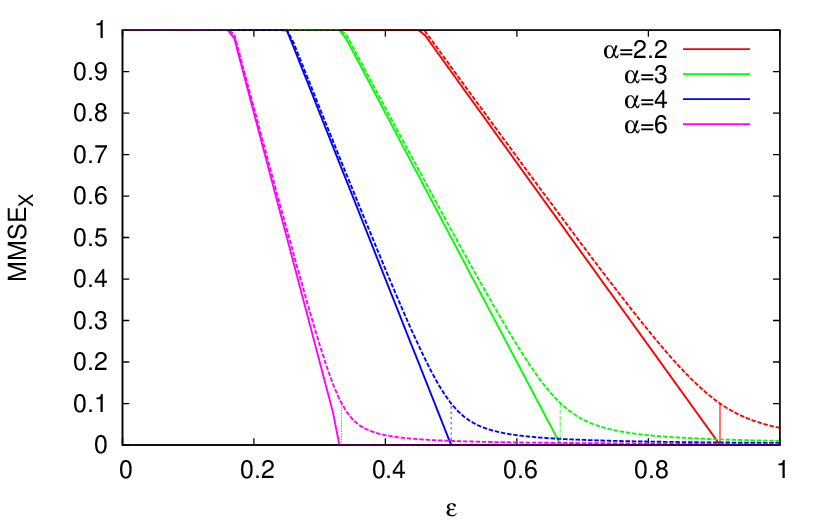}
\caption{Low rank matrix completion. The MMSE of the signal matrix $X$
  is plotted for $\alpha=\pi$ (for four
  different values of this parameter) in the noiseless
(full line) and noisy (dashed lines), with $\Delta=10^{-2}$,
cases. In zero noise there is a second order phase
transition which coincides with the counting bound, eq.~(\ref{eq:bound_MC}), marked by short
vertical lines. With non-zero noise there is no phase transition.}
\label{LRMC}
\end{figure}


In matrix completion we treat matrices $Y$ of low rank, hence $N$ is much smaller than
both $P$ and $M$. The main questions concerns the fraction $\epsilon$
of elements that need to be known in order for the recovery of $X$ and
$F$ to be possible. In this case we did not identify first order phase
transition, as a result we have MMSE$=$AMP-MSE. 
At zero measurement noise we observed a phase transition from a phase
of perfect recovery to a phase with positive MMSE, its
position coincides with the counting bound (\ref{eq:bound_MC}),
$\epsilon^* = (\alpha+\pi)/{\alpha \pi}$. With non-zero
measurement noise, in the scaling of noise and rank we consider in
this paper, the behavior of MMSE as a function of the other
parameters is smooth and derivable (no phase transition). In
Fig.~\ref{LRMC} we plot an example of the
MMSE as a function of the fraction of known elements $\epsilon$ for
squared matrix $Y$, i.e $\alpha=\pi$. We generated the signal elements
with
zero mean and unit variance, $\overline X=0$, $\sigma=1$.

Our analysis suggest that compared to the cases with non-zero sparsity
the low-rank matrix completion is a much easier problem, at least in
the random setting considered in the present paper. The fact that the
``counting'' threshold can be saturated or close to saturated in the
noiseless case by
several algorithms can be seen e.g. in the data presented in
\cite{parker2014bilinear}.

\subsection{Robust PCA}

The input functions are the eq. (\ref{eq:f_ar}) for matrix $F$ and
eq.~(\ref{eq:in_MC}) for matrix $X$. In robust PCA as defined by (\ref{eq:RPCA}) we get for the output function
\be
    g_{\rm out}(\omega,y,V) =   \frac{y-\omega}{V} \left[ 1 - \epsilon
    \frac{\Delta_s}{\Delta_s + V} -(1-\epsilon)
    \frac{\Delta_l}{\Delta_l + V} \right]\, .
\ee
The state evolution in the Bayes-optimal setting, i.e. using the Nishimori identities, becomes
\be
         \hat m^t =  \frac{1}{ \overline X^2 + \sigma -m_F^t m_X^t } \left[ 1 - \epsilon
    \frac{\Delta_s}{\Delta_s + \overline X^2 + \sigma -m_F^t m_X^t}
    -(1-\epsilon) \frac{\Delta_l}{\Delta_l + \overline X^2 + \sigma
      -m_F^t m_X^t} \right]\, .
\ee
Equation for $m_F^t$ is (\ref{eq:mF}), and for $m_X^t$ is (\ref{eq:mx_1}).


In robust PCA the informative initialization is again $m_X^t = 1-
\delta_X^t$, and $m_F^t = 1- \delta_F^t$, where $\delta_X^t$, and
$\delta_F^t$ are small positive numbers. For small noise $\Delta_s$
and $\Delta_l = O(1)$  the corresponding fixed point
is stable under the same conditions as for low rank matrix
completion, i.e. for $\epsilon > \epsilon^* =
(\alpha+\pi)/(\alpha \pi)$. 

The uninformative fixed point is $m_X^t = \delta_X^t$, and $m_F^t =\delta_F^t$, where $\delta_X^t$, and
$\delta_F^t$ are again small positive numbers. This evolves as
$\delta_X^{t+1}= \alpha \delta_F^t \hat m$, $\delta_F^{t+1}= \pi
\delta_X^t \hat m$, with $\hat m = (1+\delta_s = \epsilon \Delta_l -
\epsilon \Delta_s) /[ (1+\Delta_s) (1+\Delta_l)]$ in this limit. Hence
for instance for $\Delta_s \to 0$, $\Delta_l =1$ and $\pi =\alpha$ we
have that the uninformative fixed point is stable for $\epsilon <
2/\alpha -1$, which again corresponds to what we observe in
Fig.~\ref{RobustPCA}.

\begin{figure}[!ht]
\centering
\includegraphics[width=3.2in]{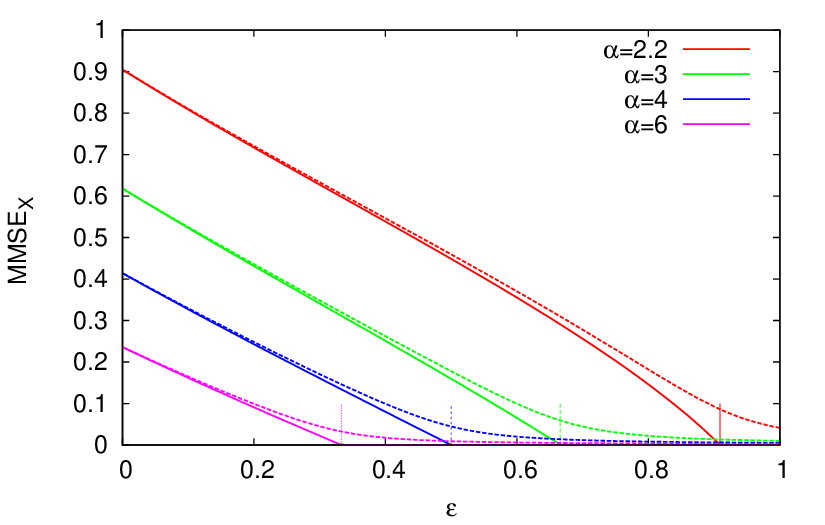}
\caption{MMSE of the matrix $X$ in robust PCA with $\alpha=\pi$ in the noiseless
(full line, $\Delta_s=0$, $\Delta_l=1$), and the noisy (dashed lines,
with $\Delta_s=10^{-2}$, $\Delta_l=1$) cases. The situation is not very
different from low rank matrix completion. Indeed in the noiseless case there
is a second order phase transition and the MMSE is zero beyond the
counting bound, marked by short vertical lines. In presence of noise there is no phase transition. 
}
\label{RobustPCA}
\end{figure}


In the example of Fig.~\ref{RobustPCA} we plot the MMSE as a function
of the fraction of undistorted elements $\epsilon$ in the case of
squared matrix $Y$, $\alpha=\pi$, the variance of the large
distortions $\Delta_l=1$ and two different values of the small
measurement noise $\Delta_s$. We see a second order phase transition
at the counting bound for $\Delta_s=0$ and a smooth decay of the MMSE
for $\Delta_X >0$.

It is interesting to compare how well robust PCA can be solved with
respect to the matrix completion. In both cases $\epsilon$ is the
fraction of known elements. The difference is that in matrix
completion their position is known, whereas in robust PCA it is
not. Intuitively the R-PCA should thus be a much harder problem. This
is not confirmed in our analysis that instead suggest that robust PCA
is as easy as matrix completion, since the zero noise phase
transitions in the two coincide. Moreover, whereas at $\epsilon \to 0$
there is no information left in matrix completion (that is why the
MMSE$=1$), in robust PCA the largely distorted elements can still be
explored and the MMSE$<1$. Note, however, that algorithmically it seems less easy to saturate this
theoretical asymptotic performance in R-PCA, see e.g. Figure 7 in \cite{parker2014bilinear}.

\subsection{Factor analysis}



In factor analysis, 
the input functions $f_F$ and $f_X$
are the same as the dictionary learning (\ref{eq:f_ar})
at $\rho=1$.
The output function (\ref{eq:def_gout}) for factor analysis (\ref{FAout})
is given by
\be
g_{\rm out}(\omega_{\mu l},y_{\mu l},V_{\mu})=\frac{y_{\mu
    l}-\omega_{\mu l}}{{\psi}_\mu+V_{\mu l}},
\ee
where $\psi_\mu$ is 
the variance of the $\mu$-th component of the unique factor. 
The 
variance of unique factor $\psi_\mu$ depends here on the index $\mu$ and does not on the
index $l$, which leads to a slight modification in the derivation of
the state evolution from section \ref{Sec:SE}. 
For simplicity, we assume that $\psi_\mu$'s are {\em known}; 
in practice, these should be estimated by the expectation and maximization scheme in conjunction with GAMP.  
Then, we obtain, on the Nishimori line 
\begin{align}
\hat{m}_\mu^t&=\frac{1}{\psi_\mu+ Q^0_{F,\mu} (\overline{X}^2+\sigma)-m^t_{F,\mu}m^t_X}. \\
m_{F,\mu}^{t+1}&=\frac{\pi Q^0_{F,\mu} 
m_X^t\hat{m}^t_\mu}{1 +\pi  Q^0_{F,\mu}   m_X^t\hat{m}_\mu^t}, \\
m_X^{t+1}&=\frac{\alpha\sigma^2\langle{m_{F,\mu}^t\hat{m}_\mu^t}\rangle_{\mu}}{1+\alpha\sigma\langle{m_{F,\mu}^t\hat{m}_\mu^t}\rangle_\mu},
\label{eq:FA_m}
\end{align}
where $\langle\cdot\rangle_{\mu}$ means the average over $\psi_\mu$, the variance of the elements of the matrix $F$ is denoted $Q_{F,\mu}^0$.  

To analytically solve (\ref{eq:FA_m}), one has to specify the distributions of $\psi_\mu$ and $Q_{F,\mu}^0$.
We set $Q_{F,\mu}^0=Q_F^0 =1$ for all $\mu$,
and suppose a two-peak distribution for $\psi$ as
\be
P_\psi(\psi)=\epsilon\delta(\psi-\psi_1)+(1-\epsilon)\delta(\psi-\psi_2).
\ee


Let us also assume $\overline X=0$ and $\sigma=1$. 
In this case the state evolution can be summarized as 
\begin{align}
m_{F1}&=\frac{\pi m_X\hat{m}_1^t}{1+\pi m_X\hat{m}_1^t},~~~
m_{F2}=\frac{\pi m_X\hat{m}_2^t}{1+\pi m_X\hat{m}_2^t}\\
m_X&=\frac{\alpha
  \{\epsilon m_{F1}\hat{m}_1^t+(1-\epsilon)m_{F2}\hat{m}_2^t\}}{1+\alpha 
 \{\epsilon m_{F1}\hat{m}_1^t+(1-\epsilon)m_{F2}\hat{m}_2^t\}},
\end{align}
where
\be
\hat{m}_1^t=\frac{1}{\psi_1+1-m_{F1}m_X},~~~\hat{m}_2^t=\frac{1}{\psi_2+1-m_{F2}m_X}.
\ee
The total MMSE is given by $E_F = 1-(\epsilon m_{F1}+(1-\epsilon)m_{F2})$ and $E_X=1-m_X$.
Fig.~\ref{fig:FA_DE_alpha2_05-2} (left) shows the $\pi$-dependence of the MMSE
at $\alpha=2$, $\psi_1=0.5$, and $\psi_2=2$ 
for $\epsilon=0$, 0.5, 0.9 and 1.

We analyze again the stability of the uninformative fixed point, $(m_X,m_{F_1},m_{F_2})=(0,0,0)$, of the state evolution. 
Small positive numbers $\delta_X$, $\delta_{F_1}$, and $\delta_{F_2}$
that give the uninformative initialization $m_X=\delta_X$, $m_{F_1}=\delta_{F_1}$, 
and $m_{F_2}=\delta_{F_2}$
evolve under the state evolution as 
\begin{align}
\delta_{F_1}^{t+1}&=\frac{\pi\delta_X^t}{\psi_1+1},~~~\delta_{F_2}^{t+1}=\frac{\pi\delta_X^t}{\psi_2+1},\\
\delta_X^{t+1}&=\alpha\pi\Big[\frac{\epsilon}{(\psi_1+1)^2}+\frac{1-\epsilon}{(\psi_2+1)^2}\Big]\delta_X^t.
\end{align}
These expressions indicate that the uninformative fixed point becomes 
unstable when $\alpha\pi \Big[\frac{\epsilon}{(\psi_1+1)^2}+\frac{1-\epsilon}{(\psi_2+1)^2}\Big]>1$.
The critical values of $\pi$ given by this condition coincides with the 
transition point where the MMSE departs from 1 shown in
Fig.~\ref{fig:FA_DE_alpha2_05-2} (left). As an example of $\epsilon$-dependence, we show 
the MMSE of $X$ at $\psi_1=10^{-2}$, $\psi_1=10^{-5}$
and $\psi_2=1$ for four different values of $\alpha=\pi$
in Fig.~\ref{fig:FA_DE_alpha2_05-2} (right). Consistently with our analysis, in these cases the
uninformative initialization is always unstable. 
%
%

\begin{figure}
\begin{center}
\includegraphics[width=2.6in]{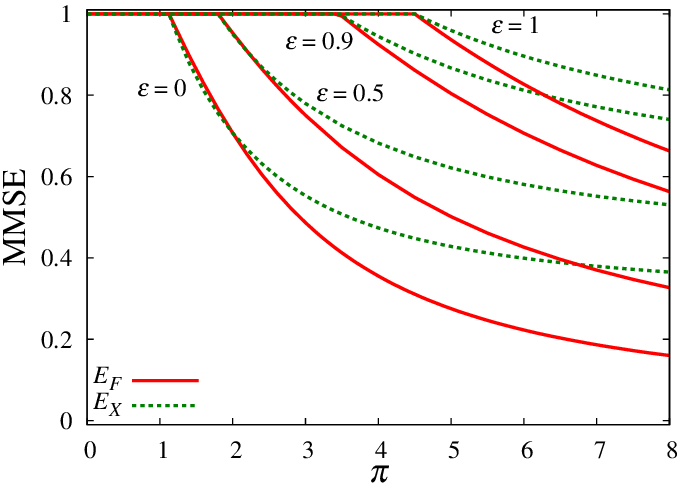}
\hspace{10mm}
\includegraphics[width=2.6in]{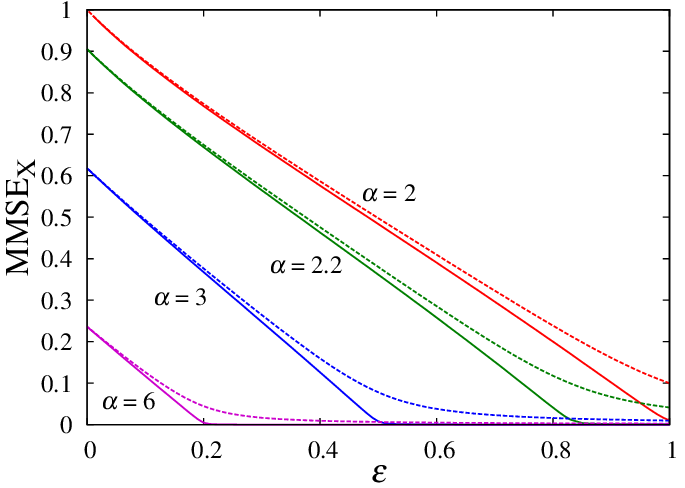}
\end{center}
\caption{Left: MMSE in factor analysis with $\psi_1=0.5$ and
  $\psi_2=2$, $\alpha=2$. Right:  MMSE of the signal matrix $X$ in factor analysis at four different values of $\alpha=\pi$
at $\psi_1=10^{-2}$ and $\psi_2=1$ (dashed lines),
and $\psi_1=10^{-5}$ and $\psi_2=1$ (solid lines).
\label{fig:FA_DE_alpha2_05-2} }
\end{figure}


The transition associated with the stability of the informative fixed point
occurs only when at least one of $\psi$s tends to be zero.
For instance when $\psi_1=0$, 
the informative initialization corresponds to $\delta_{F_1}=1-m_{F_1}$
and $\delta_X=1-m_{X}$ that are given by
\begin{align}
\delta_{F_1}^{t+1}=\frac{\delta_X^t+ \delta_{F_1}^t}{\pi},~~~\delta_X^{t+1}=\frac{\delta_{F_1}^t+\delta_X^t}{\alpha\epsilon}
\end{align}
without depending on $\psi_2$.
These expressions mean that when $\epsilon>\epsilon^*=\pi\slash [\alpha(\pi-1)]$,
the informative fixed point is stable, consistently with Fig.~\ref{fig:FA_DE_alpha2_05-2} (right).

The state evolution of factor analysis for the two-peak case is
qualitatively similar to that of robust PCA and low rank matrix
completion, but the values of the phase transition points differ.

\section{Discussion and Conclusions}
\label{sec:conclusion}

In this paper we have analyzed various examples of the matrix
factorization problem. We obtain a matrix $Y$ that is an element-wise noisy
measurement of an unknown matrix $Z=FX$, where both $Y$, and $Z$ are $M \times P$ matrices, $F$
is a $M\times N$ matrix, and $X$ is a $N \times P$ matrix. We have
considered the computational tractability of this problem in the large size limit
$N\to \infty$ while $\pi = P/N = O(1)$, and $\alpha = M/N = O(1)$. Our
analysis concerns the teacher-student scenario where $X$ and $F$ are
generated with random independent elements of some known probability
distributions and we employ the Bayes-optimal inference scheme to
recover $F$ and $X$ from $Y$.

Let us summarize our contribution: We derived the approximate message
passing algorithm for matrix factorization.  One version of the algorithm
---for calibration and dictionary learning--- was reported in
\cite{krzakala2013phase}, and a very related algorithm called Big-AMP
was discussed by \cite{parker2013bilinear,parker2014bilinear}.  This algorithm is
derived from belief propagation.  We have presented the AMP for matrix
factorization in several forms in Sections
\ref{sec:mes_pas}, and \ref{Sec:GAMP}. We
have also discussed simplifications that arise in the Bayes-optimal
setting when we can use the Nishimori
identities (Section \ref{Sec:Nish}), or when the matrix is large and
one uses self-averaging of some of the quantities appearing in GAMP
(Section \ref{Sec:full_TAP}). We focused on the theoretical properties
of the AMP algorithm, for a robust practical implementation we refer
the reader to the works \cite{parker2013bilinear,parker2014bilinear}
that include a very complete report on its performance on a range of
benchmarks. 

Next to the AMP algorithm we have also derived the corresponding Bethe
free entropy in Section \ref{Sec:Bethe}. The Bethe free entropy
evaluated at a fixed point of the GAMP equations approximates the
log-likelihood of the corresponding problem. We mainly use it in
situations when we have more than one fixed point of GAMP, it is then
the one with the largest values of the Bethe entropy that
asymptotically given the MMSE of the Bayes-optimal inference. We also
derived a variational Bethe free entropy in Section
\ref{Sec:Bethe_var}. This is a useful quantity that can serve in
controlling the convergence of the AMP approach. Alternatively, a direct maximization of this
expression is a promising algorithm
itself (see \cite{KrzakalaManoel14} for an investigation of this idea
for compressed sensing). 

The AMP algorithm for matrix factorization is amenable to asymptotic
analysis via the state evolution technique that was carried out
rigorously for approximate message passing in compressed sensing
\cite{BayatiMontanari10}. We derive the state evolution analysis for
matrix factorization using tools of statistical mechanics. In
particular we use two approaches leading to equivalent results the
cavity method (Section \ref{Sec:SE}) and the replica method (Section
\ref{Sec:repl}). Our derivation of the state evolution is not
rigorous, but we conjecture that it is nevertheless asymptotically
exact as is the case in many other systems of this type including the
compressed sensing. The main result of this state evolution are simple
iterative equations that provide a way to compute the MMSE of the
Bayes-optimal inference as well as the MSE reached theoretically in
the large size limit by the AMP algorithm. The rigorous proof of the
formulas derived in this paper is obviously an important topic for
future work. 

The main results of this paper concern analysis of MMSE and AMP-MSE
for various interesting examples of the matrix factorization
problem. We analyze the asymptotic phase diagrams for dictionary
learning, blind matrix calibration, sparse PCA, blind source
separation, matrix completion, robust PCA and factor analysis. Earlier
results on this analysis appeared in
\cite{krzakala2013phase,SakataKabashimaNew}. We find that when one of
the matrices $F$ or $X$ is sparse the problems undergo a first order
phase transition which is related to an interesting algorithmic
barrier known for instance from compressed sensing
\cite{KrzakalaMezard12}.

It is a generic observation that for most of the problems we analyzed
the theoretically achievable performance is much better than the one
achievable by existing algorithms. The AMP algorithm should be able to
match this performance for very large systems which is the most
exciting perspective for further development of this work. If
successful it could lead to an algorithmic revolution in various
application of the matrix factorization.

In this paper we concentrate on the theoretical analysis and not on
the performance of the algorithm itself. Some studies of the
performance of some versions of the algorithm can be find in
\cite{krzakala2013phase,parker2014bilinear}. We, however, observed that
the performance depends strongly on the implementation details and we
did not yet found a way to match the theoretically predicted
performance for systems of treatable (practical) size in all cases. 

It is worth discussing some of these algorithmic issues in this
conclusion.  One of the main problems it that the GAMP algorithm with
parallel update presents instabilities that drive its evolution away
from the so-called Nishimori line; see a recent study of this issue in
the compressed sensing problem \cite{CaltagironeKrzakala14}. This can
be seen even in the state evolution when we do not assume explicitly
that the result of the Bayes-optimal inference corresponds to a fixed
point that belongs to the Nishimori line. There are ways how to avoid
these issues, e.g. we observed that
the difficulties basically disappear when the sequential update of the
message passing algorithm from Section \ref{sec:mes_pas} is used
instead of the parallel one. This, however, does not scale very well with the systems size
and our results were hence spoiled by very strong finite size
effects. When learning the $M\times N$ matrix $F$, and the $N
\times P$ matrix $X$ we also often observed that a (very small) number of
the $P$ signals $X_l$ were not correctly reconstructed, and these ``rogue"
vectors in $X_l$ were polluting the reconstruction of $F$. An exemple of
these finite size effects can be observed in
\cite{krzakala2013phase}.

In the work of \cite{parker2013bilinear,parker2014bilinear} a part of the problems with
convergence of the corresponding algorithm was mitigated by adaptive
damping (though maybe with not the most suitable cost function, see
Sec.~\ref{Sec:Bethe}) and expectation maximization learning. Why this
is helpful is theoretically explained in the recent work
\cite{CaltagironeKrzakala14} for compressed sensing. However, the
implementation of \cite{parker2014bilinear} does not match the
theoretical performance predicted in this paper either (we have
explicitly tried for the dictionary learning and robust PCA
examples). This shows that more work is needed in order to reach a
practical algorithm able to achieve the prediction at moderate
sizes. A more proper understanding of these problems, and further
developments of the algorithm are therefore the main direction of our
future work.

\begin{table}[!ht]
\begin{center}
\begin{tabular}{|l||c|c|c|c|} \hline
 {\bf  variable nodes}   & definition & $X$  &definition & $F$    \\ \hline  \hline
variables for elements & (\ref{z=Fx})  & $X_{il}$ & (\ref{z=Fx})   &  $F_{\mu i}$  \\ \hline
true elements & (\ref{MSE_X_def})  & $X^0_{il}$ & (\ref{MSE_F_def})   &  $F^0_{\mu i}$  \\ \hline
prior  distribution & (\ref{PX}) &   $P_X(X_{il})$ &  (\ref{PF})  &  $P_F(F_{\mu i})$    \\ \hline
incoming BP message  & (\ref{message_tm})  & ${\tilde m}_{\mu l \to i l}(X_{il})$  &  (\ref{message_tn})    & ${\tilde n}_{\mu l \to \mu i}(F_{\mu i})$  \\ \hline
outgoing BP message  &  (\ref{message_m})    &  $m_{il\to\mu l}(X_{il})$ & (\ref{message_n})    &  $n_{\mu i \to\mu l}(F_{\mu i})$  \\ \hline
approximate-BP mean  &  (\ref{eq_def_fa}) & ${\hat{x}}_{il \to \mu l}$ & (\ref{eq_def_fr})  & ${\hat{f}}_{\mu i \to \mu l}$  \\ \hline
approximate-BP variance  & (\ref{eq_def_fc})  & $c_{il \to \mu l}$  &(\ref{eq_def_fs})   & $s_{\mu i \to \mu l}$   \\ \hline
incoming-BP mean &  (\ref{eq:B}) & $B_{\mu l \to il}$  &   (\ref{eq:R}) & $R_{\mu l \to \mu i}$  \\ \hline
incoming-BP variance &  (\ref{eq:A}) & $A_{\mu l \to il}$ &   (\ref{eq:S}) & $S_{\mu l \to \mu i}$  \\ \hline
mean-input function  &  (\ref{f_ac_gen})  & $f_X $ &  (\ref{f_rs_gen})  &$ f_F $ \\ \hline
variance-input function  & (\ref{f_ac_gen})  &$ f_c$ &  (\ref{f_rs_gen})  & $f_s $ \\ \hline
incoming mean  & (\ref{eq:def_Sigma})  & $T_{il}$ &   (\ref{eq:def_WZ}) & $W_{\mu i}$  \\ \hline
incoming variance  &  (\ref{eq:def_Sigma}) & $\Sigma_{il}$ &  (\ref{eq:def_WZ})  &
  $Z_{\mu i}$ \\ \hline
GAMP estimate of mean  &  (\ref{gamp_ac}) & ${\hat{x}}_{il}$ & (\ref{gamp_rs})  & ${\hat{f}}_{\mu i}$  \\ \hline
GAMP estimate of variance  & (\ref{gamp_ac})  & $c_{il}$  &(\ref{gamp_rs})   & $s_{\mu i}$   \\ \hline
SE magnetization  & (\ref{mx})  & $m_X$ &  (\ref{mx})  &  $m_F$ \\ \hline
SE overlap  &  (\ref{qq}) & $q_X$   &  (\ref{qq})  & $q_F$   \\ \hline
SE variance  & (\ref{QD})   & $Q_X-q_X$   & (\ref{QD})    & $Q_F-q_F$   \\ \hline
\end{tabular}  

\vspace{5mm}

\begin{tabular}{|l||c|c|} \hline
 {\bf factor nodes}    & definition & output     \\ \hline \hline
  matrix to factorize        & (\ref{z=Fx})  &  $z_{\mu l}$  \\ \hline 
output &  (\ref{PY})  &  $y_{\mu l}$   \\ \hline
output probability&   (\ref{PY})   &  $P_{\rm out}(y_{\mu l}|z_{\mu l})$
\\ \hline
output realization function &   (\ref{eq:hPout})   &  $h(z,w)$  \\ \hline
cavity variance of $z_{\mu l}$ &   (\ref{eq:def_Vmuil}) &  $V_{\mu i l}$  \\ \hline
cavity estimation of $z_{\mu l}$&  (\ref{eq:def_omuil})  &  $\omega_{\mu i l}$  \\ \hline
output function &  (\ref{eq:def_gout})  &  $g_{\rm out}(\omega,y,V) $ \\ \hline
variance of $z_{\mu l}$ &   (\ref{eq:def_V}) &  $V_{\mu  l}$  \\ \hline
estimation of $z_{\mu l}$&  (\ref{eq:def_o})  &  $\omega_{\mu  l}$  \\ \hline
SE output $q$ & (\ref{eq:def_qhat})      &  $\hat q$  \\ \hline
SE output $\chi$ & (\ref{eq:def_chihat})      &  $\hat \chi$  \\ \hline
SE output $m$ & (\ref{eq:def_mhat})      &  $\hat m$  \\ \hline
\end{tabular}  
\caption{\label{table1} Glossary of notations with the number of
  equation where the quantity was defined or first used.}
\end{center}
\end{table}

\section*{Acknowledgments} The research leading to these results has
received funding from the European Research Council under the European
Union's $7^{th}$ Framework Programme (FP/2007-2013)/ERC Grant
Agreement 307087-SPARCS.  {Financial support by the JSPS Core-to-Core
  Program ``Non-equilibrium dynamics of soft matter and information''
  and JSPS/MEXT KAKENHI Nos. 23-4665 and 25120013 is also
  acknowledged. }

\bibliographystyle{nature}
\bibliography{refs}

\end{document}

%% file: algo_tex.tex
\begin{algorithm}[H]
    \caption{Bayes-optimal Approximate Message Passing for Matrix Factorization\label{alg:grbmamp}}  
  \begin{algorithmic}
    \STATE {\bfseries Input:} $\mathbf{y}$
    \STATE \emph{Initialize}: $\mathbf{\hat x}$,$\mathbf{\hat f}$,$\mathbf{c}$,$\mathbf{s}$,
                              ${\rm Iter} =1$
    \STATE \emph{Initialize}:  $\{\omega_{\mu l},V\}$
    \REPEAT   
    \STATE Compute averages, eqs.~(\ref{qq-1}) and (\ref{QD-1}).
      \STATE AMP Update of $\{\omega_{\mu l},V\}$, eq.~(\ref{eq:omega_fullTAP}) and eq.~(\ref{eq:V_fullTAP}).
      \STATE AMP Update of $\Sigma$  and $Z$, eq.~(\ref{eq:SZ_NL}).
      \STATE AMP Update of $\{ T_{il} \}$,$\{ W_{\mu i}\}$, eq.~(\ref{eq:T_}) and eq.~(\ref{eq:W_}).
      \STATE AMP Update of the estimed quantities  $\mathbf{\hat
        x}$,$\mathbf{\hat f}$,$\mathbf{c}$,$\mathbf{s}$, eq.~(\ref{gamp_ac}) and eq.~(\ref{gamp_rs}).
      \STATE ${\rm Iter} \gets {\rm Iter} + 1$
    \UNTIL{Convergence on $\mathbf{\hat x}$,$\mathbf{\hat f}$}
  \end{algorithmic}
\end{algorithm}